\documentclass [12pt, makeidx]{amsart}
\pdfoutput=1
\usepackage{graphicx}
\usepackage{colordvi}
\usepackage{amsthm}
\usepackage{amssymb}
\usepackage{enumerate}
\usepackage[all]{xy}
\usepackage{ulem}
\usepackage{pgf}
\usepackage{xkeyval}
\usepackage{sseq}
\usepackage{float}
\newtheorem{thm}{Theorem}[section]
\newtheorem{cor}[thm]{Corollary}
\newtheorem{lem}[thm]{Lemma}
\newtheorem{slem}[thm]{Sublemma}
\newtheorem{prop}[thm]{Proposition}

\theoremstyle{definition}
\newtheorem{defn}[thm]{Definition}
\newtheorem{prop-defn}[thm]{Proposition and Definition}
\theoremstyle{remark}
\newtheorem{rem}[thm]{Remark}

\newtheorem{example}[thm]{Example}
\newtheorem{examples}[thm]{Examples}
\numberwithin{equation}{section}

\newtheorem{exo}[thm]{Exercice}
\newenvironment{question}{\begin{bf} Question:\end{bf}}{}

\newcommand{\Ker}{\rm Ker}
\newcommand{\Hom}{{\rm Hom}}
\newcommand{\hHom}{{\mathcal Hom}}

\usepackage{pdfsync}

\newcommand{\Mor}{{\rm Mor}}

\newcommand{\C}{{\mathcal C}}

\newcommand{\F}{{\mathcal F}}
\newcommand{\G}{{\mathcal G}}

\newcommand{\I}{{\mathcal I}}
\newcommand{\cI}{{\mathcal I}^{\bullet}}
\newcommand{\cH}{{\mathcal H}^{\bullet}}

\newcommand{\J}{{\mathcal J}}

\newcommand \id {{\rm id}}
\newcommand \Cat [1] {{\rm{\bf #1}}}

\usepackage{scrtime}
\usepackage{enumitem}
\usepackage{comment}
\usepackage{varioref}
\usepackage{float}
\usepackage{epigraph}
\usepackage{stmaryrd}
\usepackage{fourier}
\usepackage{xfrac}
\usepackage{rotating}
 \usepackage{todonotes}
\usepackage[nolabel]{showlabels}
 \usepackage[mathscr]{eucal}
\newcommand {\area}{\mathrm {area}}

\newcommand {\cS}{\mathcal {S}^\bullet}
 \usepackage[letterpaper]{geometry}

\usepackage{setspace}
\usepackage[percent]{overpic}
\usepackage{datetime}
\usepackage{lscape}
\usepackage{xr}

\begin{document}
\def \Z {\mathbb Z}
\def \H {\mathcal H}
\def \cF {\F^{\bullet}}
\def \cG {\G^{\bullet}}
\def \cI {\I^{\bullet}}
\def \cJ {\J^{\bullet}}
\def \Char {{\rm Char}}
\def \card {{\rm card}}
\def \cstar {\varoast}

\title{Sheaf Quantization of Lagrangians\\  and \\ Floer cohomology}
\author{C. Viterbo }
\thanks{DMA, \'Ecole Normale Sup\'erieure, PSL-University, 45 Rue d'Ulm, 75230 Cedex 05, FRANCE and D\'epartement de Math\'ematiques, Universit\'e de Paris-Sud, Orsay. This research was  also supported  by ANR Microlocal, ANR-15-CE40-0007,  CMLS, \'Ecole polytechnique, Palaiseau, IAS Princeton, and the Eilenberg Chair at Columbia University,  by the National Science Foundation under Grant No. 1440140, while the author was in residence at the Mathematical Sciences Research Institute in Berkeley, California, during the Fall semester 2018.}

\begin{abstract}Given an exact Lagrangian submanifold $L$ in $T^*N$, we want to construct a complex of sheaves in the derived category of sheaves on $N\times {\mathbb R} $, such that its singular support, $SS(\cF_L)$,
is equal to $\widehat L$, the cone constructed over $L$. Its existence was stated in \cite{Viterbo-ISTST} in 2011, with a sketch of proof, which however contained a gap (fixed here by the rectification). A complete proof was shortly after provided by Guillermou (\cite{Guillermou}) by a completely different method, in particular Guillermou's method does not use Floer theory. The proof provided here is, as originally planned, based on Floer homology. 
Besides the construction of the complex of sheaves, we prove that the filtered versions of sheaf cohomology of the quantization and  of Floer cohomology coincide,  that is $FH^*(N\times ]-\infty, \lambda [, \cF_L)\simeq FH^*(L;0_N;\lambda)$, and so do their product structures.  
\end{abstract}

\maketitle
\today ,\;\; \currenttime
\tableofcontents
\section{Notations and basic definitions}
In this paper  $N$ will be a compact manifold, and we work in the symplectic manifold $(T^*N, d\lambda)$ where $\lambda=pdq$ is the Liouville form\footnote{While the microlocal theory of sheaves was invented by Kashiwara and Schapira and recorded in their book \cite{K-S}, published in the 90's, we usually also give the relevant reference to \cite{Viterbo-ISTST}. Needless to say, all results are due to Kashiwara and Schapira and the reference to \cite{Viterbo-ISTST} is only for ease of access to symplectic geometers, as we tried to make \cite{Viterbo-ISTST} and the present paper self-contained.
} .
We shall assume we are given a connected oriented Lagrangian submanifold $L$ such that 

\begin{enumerate} 
\item $L$ is exact and $\lambda_{\mid L}=df_L$ for some function $f_L: L \longrightarrow {\mathbb R} $
\item $L$ is spin, or our coefficient field is $\mathbb F=\Z / 2 \Z$
\item The Maslov class $\mu_L$ vanishes in $H^1(L,\mathbb F)$ (this is automatic if $\mathbb F=\Z / 2 \Z$, since $L$ is orientable)
\end{enumerate} 

Moreover we assume that together with, $L$ we are given $f_L$, a primitive of $\lambda_{\mid L}$ as above, and a grading, that is if $\Sigma^1(L)=\{(x,p)\in L \mid d\pi(x,p)_{\mid T_{(x,p)}L} \; \text {is not onto}\;\}$ is the set of singular points of the projection on the base, we have a function $m_L: L\setminus \Sigma^1(L) \longrightarrow \mathbb F$ such that for any smooth path $\gamma$ in $L$, we have
$$m_L(\gamma(1))-m_L(\gamma(0))=\#_{alg}(T_{\gamma(t)}L \cap \Sigma^1)$$
where the right hand side is the number of intersection points counted with sign, of the path $t \mapsto T_\gamma(t)L$ and the Maslov cocycle $ \Sigma^1=\{ T \in \Lambda(T_{(x,p)}(T^*N)) \mid T\cap V(x,p)\neq \{0\}$ where $\Lambda (E)$ is the Lagrangian Grassmannian in $E$ and $V(x,p)$ the tangent to the vertical (i.e. $\Ker (d\pi(x,p))$)). So if $\gamma $ is  a loop we indeed have  $m_L(\gamma(1))-m_L(\gamma(0))=\langle \mu_L, \gamma \rangle =0$. 

We refer to Arnold (cf. \cite{Arnold}) for this, and point out that it follows from the work of Kragh \cite{Kragh2} that $\mu_L$ vanishes for an embedded exact Lagrangianand was later reproved in Guillermou's paper (\cite{Guillermou}), using sheaf theory. 

We will call the triple $(L,f_L,m_L)$ an {\bf exact Lagrangian brane}. Note that given $L$ satisfying the above conditions, $f_L$ and $m_L$ are uniquely defined up to the addition of a constant (real for $f_L$,  belonging to $\mathbb F$ for $m_L$). Again by abuse of notation, we sometimes write $L$ instead of $(L,f_L,m_L)$.  Note also that for us  Lagrangian submanifolds are unparametrized embeddings, i.e. we identify $j : L \longrightarrow T^N$ and $j\circ \varphi : L \longrightarrow T^*N$ where $\varphi$ is an orientation preserving diffeomorphism. 

\begin{examples} \label{lim}
\begin{enumerate} 
\item  If $f$ is a smooth function on $N$ we set $\Gamma_f$ to be the Lagrangian brane with $L_f=\{(x,df(x)) \mid x\in N\}$, with $f_{L_f}(x,p)=f(x)$ and $m_L\equiv 0$.  
\item Even though we shall mostly restrict to situations where $L$ is compact, one important case is the conormal bundle of  a smooth submanifold $V$ of $N$. We set $\nu^*V= \{(x,p) \mid p=0 \;\text{on}\; T_xV\}$ and $f_L\equiv 0$, $m_L\equiv 0$. 
\item \label{lim2} Similarly if $U$ is codimension zero submanifold with smooth boundary $\partial U$ we set, writing $n(x)$ for the outside normal of $U$ at $x\in \partial U$ : 
$$\nu^*U= \{(x,p) \mid \; \text{either}\; x\in U, p=0\; \text{or}\; x\in \partial U, p=0 \;\text{on}\; T_x\partial U \; \text{and}\; \langle p, n(x)\rangle <0\}$$
Again we set $f_{\nu^*U}\equiv 0, m_{\nu^*U}\equiv 0$. 
Note that this example does not really fit in our framework, since $\nu^*U$ is not smooth, but it will repeatedly appear as limit of smooth Lagrangians. Note that in a certain weak sense that will be specified later, we have that if $f_k$ is a decreasing sequence of smooth functions converging to $-\infty (1-\chi_U)$, we have that $\lim_k \Gamma_{f_k}=\nu^*U$. 
\end{enumerate} 
\end{examples} 

Let  $L_0,L_1$ be  Lagrangian branes such that the underlying Lagrangians are in general position. Let  $J(t,z)$ (where $t\in [0,1], z\in T^*N$)  be an almost complex structure on $T^*N$ compatible with the symplectic form, depending smoothly on $(t,z)$, and such that the sets $\{(x,p)\in T^*N\mid \vert p \vert \leq r \}$ are $J$-pseudo-convex (at least for $r$ large enough)\footnote{This is needed to recover compactness of the set of holomorphic curves.}. We define $(FC_J^\bullet(L_0,L_1),d_J)$ to be the filtered Floer complex associated to $L_0,L_1, J$, that is the complex generated by intersection points $L_0\cap L_1$, where $x\in L_0\cap L_1$ has grading $m_{L_0,L_1}(x)=m_{L_0}(x)-m_{L_1}(x)$, action filtration $f_{L_0,L_1}(x)=f_{L_1}(x)-f_{L_0}(x)$, and\footnote{Watch out, many authors use the opposite convention, filtering by $-f_{L_0,L_1}$. Our convention will be compatible with the usual identification with $R\hHom(\cF_{L_0},\cF_{L_1})$ where $\cF_j$ are the natural quantizations of $L_j$.} such that if $m_{L_0,L_1}(x_-)=m_{L_0,L_1}(x_+)-1$ the differential is  defined by   $$\langle d_J x_-, x_+\rangle = \#_{alg}\mathcal M (x_-,x_+)/ {\mathbb R} $$
where, setting  $\bar\partial_Ju=\frac{\partial u}{\partial s}(s,t)+J(t,u(s,t))\frac{\partial u}{\partial t}(s,t)$, we have  $$\mathcal M (x_-,x_+)=\{u: {\mathbb R} \times [0,1] \longrightarrow T^*N \mid \bar\partial_Ju=0, \; u(s,i)\in L_i,\; \lim_{s\to \pm\infty}u(s,t)=x_\pm\}$$
 Note that $m_{L_0,L_1}(x_+)-m_{L_0,L_1}(x_-)$ is the Fredholm index of the operator $\bar\partial_J$ on the space of maps  $$u: {\mathbb R} \times [0,1] \longrightarrow T^*N \mid  \; u(s,i)\in L_i,\; \lim_{s\to \pm\infty}u(s,t)=x_\pm$$
Finally $FC_J^\bullet(L_0,L_1)$ is filtered by the action $f_{L_0,L_1}$, and we denote by $FC^\bullet_J(L_0,L_1;b)$ the quotient of $FC_J^\bullet(L_0,L_1)$ by the subspace generated by the intersection points with action greater than $b$, and for $a<b$, by $FC^\bullet_J(L_0,L_1;a,b)$ the quotient $FC^\bullet_J(L_0,L_1;a)/FC^\bullet_J(L_0,L_1;b)$. It is generated by the intersection points in $L_0\cap L_1$ with action in the window $[a,b[$.  Since $d_J$ is increasing with respect to the action filtration\footnote{i.e. $d_J(x)=y$ implies $f_{L_0,L_1}(x)\leq f_{L_0,L_1}(y)$}, and $d_J^2=0$, we may state

\begin{defn} 
We set $FH^\bullet(L_0,L_1;b)$ to be the cohomology of $(FC^\bullet_J(L_0,L_1;b),d_J)$ and $FH^\bullet(L_0,L_1;a,b)$ to be the cohomology of $(FC^\bullet_J(L_0,L_1;a,b),d_J)$.
\end{defn} 

Note that according  to the work of Floer (cf. \cite{Floer1,Floer2,Floer3,Floer4}), the (filtered) Floer cohomology does not depend on $J$. Of course this is not true for the complex itself. 

\begin{defn}
Let $\mathcal L$ be the set of compact Lagrangian branes. The "naive Fukaya category" has objects the elements in $\mathcal L$ and morphisms  the cofiltered complex $FC^\bullet_J(L_0,L_1)$, parametrized by $J$. Note that this is not a category, because $\Mor(L_0,L_1)$ is only defined if $L_0\pitchfork L_1$ and even when defined, composition\footnote{corresponding to the ``pant product'', here sometimes called ``triangle product'' since it is obtained by counting holomorphic triangles.} $$FC^\bullet_J(L_0,L_1;\lambda_1)\otimes FC^\bullet_J(L_1,L_2; \lambda_2) \longrightarrow FC^\bullet_J(L_0,L_2;\lambda_1+\lambda_2)$$
is not associative. 
\end{defn}
Note that one can get rid of the transversality requirements by introducing some perturbation data as in \cite{Seidel}. But we actually still do not get a category, but an $A_\infty$-category. We do get a category if we take the cohomology category, i.e. replace $FC^\bullet$ by $FH^\bullet$, this yields  the so-called {\bf Donaldson-Fukaya} category\index{Donsaldson-Fukaya category}\index{Category!Donaldson-Fukaya}.

In this paper, we shall replace this "naive Fukaya category" by a bona-fide category in which the Fukaya category can be "faithfully embedded" (we shall make this a precise statement when stating  the main theorem). We shall associate to $L\in \mathcal L$ an element $\cF_L \in D^b(N\times {\mathbb R} )$, the derived category of bounded complexes of sheaves on $N\times {\mathbb R} $, such that the singular support of $\cF_L$ (see \cite{K-S}, page 218, or \cite{Viterbo-ISTST} page 139, 
 for the definition) $SS(\cF_L)$ is the conification of $L$ :  $$\widehat L= \{(x,t,\tau p, \tau) \mid (x,p)\in L, t=f_L(x,p), \tau \geq 0\} \subset T^*(N\times {\mathbb R} )$$
 
  \begin{examples}
 \begin{enumerate}
 \item  Let $L=\Gamma_f$, then $\widehat \Gamma_f= \{ (x,f(x), \tau df(x),\tau) \mid \tau \geq 0\}$. We fid that  $FC^\bullet (0_N,\Gamma_f)$is the Morse complex of $f$.

 \item Let $L=\nu^*U$, then $\widehat {\nu^*U} = \{ (x,0,-\tau \nu(x), \tau) \mid \tau \geq 0\}$ where $\nu(x)=0$ when $x\in U$ and $n(x)\in T_x^*N$ is the outgoing normal at $x$. Thus $\widehat {\nu^*U} =0_U\times \{0\}\times {\mathbb R}_+ \cup \{(x,t,\tau n(x),\tau) \mid \tau \geq 0, x\in \partial U\} $
 \end{enumerate}
  Note that as above, in a certain weak sense, we have that if $f_k$ is a decreasing sequence converging to $-\infty (1-\chi_U)$, we have $\lim_k \widehat \Gamma_{f_k}=\widehat{\nu^*U}$. 
  \end{examples}

 Moreover, associated to each pair $(\cF_1,\cF_2)$ in $D^b(N\times {\mathbb R} )$ an element in $D^b(N\times {\mathbb R} )$, denoted $R\hHom^{\cstar} (\cF_1,\cF_2)$, such that 
\begin{gather*} \Mor_{D(N)}(\cF_1,\cF_2)=  H^0(N\times {\mathbb R}, R\hHom^{\cstar}(\cF_1,\cF_2))=RHom^{\cstar}(\cF_1,\cF_2)) \end{gather*} 
 defines a category $\Cat{D(N)}$. Note that objects of $\Cat{D(N)}$ are actually in $D^b_+(N\times {\mathbb R} )$ the set of objects of the derived category with singular support contained\footnote{This category was first used in this context in Tamarkin's work (\cite{Tamarkin})} in $\{\tau \geq 0\}$. Note that we can arrange so that $R\hHom^{\cstar} (\cF_1,\cF_2)$ also belongs to $D^b_+(N\times {\mathbb R} )$, so we get an enriched  and closed category\footnote{A category is said to be closed if the morphisms are also objects in the category.}. 

Our main theorem then reads

\begin{thm} [Main theorem]\label{Main-theorem}
To each $L\in \mathcal L$ we can associate $\cF_L \in \Cat{D(N)}$ such that 
\begin{enumerate} 
\item\label{1} $SS(\cF_L)=\widehat L$
\item\label{2} $\cF_L$ is pure (cf. \cite{K-S} page 309), $\F_L=0$ near $N\times \{-\infty\}$ and $\F_L=k_N$ near $N\times \{+\infty\}$
\item\label{3} We have an isomorphism\footnote{If $Z$ is locally closed, of the form $U\cap A$, with $U$ open and $A$ closed, we have a functor $\Gamma_Z$ (see \cite{Viterbo-ISTST}, page 83), 
that yields the relative cohomology $H_Z^*(X,\F)=H^*(U, U\setminus A,\F)$, fitting in the long exact sequence $ H^{*}(U,U\setminus A, \F) \longrightarrow H^*(U,\F) \longrightarrow H^*(U\setminus A,\F) \longrightarrow H^{*+1}(U,U\setminus A, \F) \longrightarrow $. As a result,  $H^*(N\times [a,b[, \cF )$ could also be denoted $H^*(N\times ]-\infty , b[, N\times ]-\infty, a[, \cF)$ .}
$$FH^\bullet(L_0,L_1;a,b)=H^*\left (N\times [a,b[, R\hHom^{\cstar}(\cF_{L_0},\cF_{L_1})\right)$$
\item \label{4}  $\cF_L$ is unique satisfying properties (\ref{1}) and (\ref{2}). 
\item \label{5} The pant product in Floer cohomology is induced through the above identification to a product map  $$
 R\hHom^\cstar(\cF_{L_1},\cF_{L_2}) \otimes R\hHom^\cstar(\cF_{L_2},\cF_{L_3}) \longrightarrow  R\hHom^\cstar(\cF_{L_1},\cF_{L_3})
$$
inducing a map
\begin{gather*}
H^*(N\times [\lambda , +\infty [, \check{\cF_{L_1}}\cstar\cF_{L_2}) \otimes H^*(N\times [\mu , +\infty [, \check{\cF_{L_2}}\cstar\cF_{L_3})
\\  \Big\downarrow \cup_{\cstar} \\ H^*(N\times [\lambda + \mu , +\infty [, \check{\cF_{L_1}}\cstar\cF_{L_3})
\end{gather*}

\end{enumerate} 
\end{thm} 
  
  Note that the construction of $\cF_L$ was sketched in \cite{Viterbo-ISTST} in our Eilenberg lectures held at Columbia in the spring 2011. It was  first proved by completely different methods (without the use of Floer cohomology) by Guillermou in \cite{Guillermou}. This is also related to the work of Nadler and Zaslow \cite{Nadler, Nadler-Zaslow}, the original idea that microlocal sheaf theory had something to say about symplectic topology goes back to the foundational paper of Tamarkin \cite{Tamarkin} and developed in \cite{G-K-S}.  

Apart from the construction, some of the new features here are the connection between Floer theory and sheaf theory, that is each one allows to recover the other. We can thus use sheaf theoretic properties to prove results on Floer cohomology and vice-versa.

The basic heuristic idea of the proof is as follows. First assume $L$ is obtained from a generating function quadratic at infinity (GFQI for short). That is 
$$L=L_S=\{ (x, \frac{\partial S}{\partial x}(x,\xi)) \mid \frac{\partial S}{\partial \xi}(x,\xi)=0\}$$  where $S\in C^\infty(N\times {\mathbb R} ^k, {\mathbb R} )$ is such that $(x,\xi)\mapsto \frac{\partial S}{\partial \xi}(x,\xi)$ is transverse to $0$, and $S(x,\xi)$ coincides with a non-degenerate quadratic form $Q(\xi)$ outside a compact set. 

Then, setting $U_S=\{(x,\xi,t) \in N\times {\mathbb R}^k\times {\mathbb R}  \mid S(x,\xi)<t \}$, $k_{U_S}$ the constant sheaf on $U_S$  and $\cF_S=(R\pi)_*(k_{U_S})$ (where $\pi: N\times {\mathbb R}^k\times {\mathbb R} \longrightarrow N\times {\mathbb R} $ is the projection, $(R\pi)_*$ the right-derived functor on sheaves) satisfies(see \cite{Viterbo-ISTST}, page 146 
$SS(\F_S)=\widehat L_S$.  Now  $\cF_S$ satisfies $$H^\bullet (U\times [a,b[; \cF_S)=H^\bullet(U\times {\mathbb R}^k\times [a,b[, k_{U_S})=H^\bullet(S^b_{\mid U}, S^a_{\mid U})=FH^\bullet (\nu^*U,L; a,b)$$ The meaning of the last equality (and the definition of $FH^*(\nu^*U,L;a,b)$, since $\nu^*U$ is not smooth)  will be  explained later on. 
So it is tempting to look for an element $\cF_L$ in $D^b(N\times {\mathbb R} )$ such that $FC^\bullet(\nu^*U,L;a,b)$ is quasi-isomorphic to $R\Gamma (U\times [a,b[,\cF_L)$. A first guess would be to take $U\times [a,b[ \mapsto FC^\bullet(\nu^*U,L;a,b)$. But this does not define a sheaf, not even a presheaf. 
Indeed, if $W\subset V \subset U$, there is indeed a natural map (that we shall define later) $FC^\bullet(\nu^*U,L;a,b) \longrightarrow FC^\bullet(\nu^*W,L;a,b)$, but it only coincides up to homotopy with the composition of
 $FC^\bullet(\nu^*U,L;a,b) \longrightarrow FC^\bullet(\nu^*V,L;a,b)$ and $FC^\bullet(\nu^*V,L;a,b) \longrightarrow FC^\bullet(\nu^*W,L;a,b)$. We shall call such a structure a {\bf quasi-presheaf}\index{Quasi-presheaf}. We shall precisely define  this in the next section and prove its main properties. In section \ref{Rectif}, we {\bf rectify} this quasi-presheaf to obtain a bona fide presheaf, and then, by sheafification, a sheaf.  We then conclude the proof of the construction of $\cF_L$ in the remaining sections\footnote{cf. \cite{Segal},\cite{Vogt},\cite{Goerss-Jardine},\cite{Lurie}}. 
 
 Note that one of the implications of the above theorem, is that even though $FC^\bullet(L_0,L_1;a,b)$ is not always well defined, and its multiplication not associative, we can replace it by the quasi-isomorphic complex, $R\Gamma \left (N\times [a,b[, R\hHom^{\cstar}(\cF_{L_0},\cF_{L_1})\right)$ which is always well defined, and does have  associative multiplication. 

\begin{rem} 
One could have tried to consider the complex of sheaves $\H_L^\bullet(U\times [a,b[)=FH^\bullet(\nu^*U,L;a,b)$, but this fails, because this sheaf has extra singular support at points of $\Sigma^1_L$, where  $\Sigma^1_L$ is the set of points of $L$ where the projection $\pi: L \longrightarrow N$ is singular. Note that a spectral sequence argument implies that its singular support always contains $\widehat L$, but  in general $SS(\H_L)$ strictly contains $\widehat L$.

Let us  give here an example where taking $\H^\bullet(U\times [a,b[)=FH^\bullet(\nu^*U,L;a,b)$ has singular support different from $\widehat L$.  Our example has a GFQI, so we can replace $FH^\bullet(\nu^*U,L;a,b)$ by $H^\bullet(S^b_{\mid U}, S^a_{\mid U})$. 
Consider the Lagrangian in $T^*S^1$ locally given by $x-p^2=0$, then the reader can check that it is locally generated by $S(x,\xi)=- \frac{1}{3}\xi^3+x\xi$. Thus $\lim_{\varepsilon \to 0}H^*(S_x^{c+\varepsilon},S_x^{c-\varepsilon})$ equals $0$ outside the cusp $\pm c=\frac{2}{3} ({\sqrt x})^3$ and $k$ inside. The singular set of this sheaf is $(p^2, \frac{2}{3} p^3, \tau p, \tau) \mid \tau >0 \} \cup 0\times \mathbb R \times 0\times {\mathbb R}_+$, of the form $\widehat L$ where $L=\{(x,x^2) \mid x \in \mathbb R\} \cup \{0\}\times \mathbb R$. 
\end{rem}

 \begin{rem} 
 We can add to the set $\mathcal L$ the Lagrangian branes of the form $\nu^*U$ where  $U$ is an open set with smooth boundary or a smooth submanifold and the fibers $V_x=T_x^*N$. These will always be considered in the category as limits of $\Gamma_{f_k}$ for some sequence, being understood that the results we state do not depend on the choice of the sequence $f_k$ (unless otherwise stated). Note that $FH^*(\nu^*U, \nu^*V; \lambda, \mu)$ only makes sense provided we define some kind of wrapping perturbation  (usually obtained by following the geodesic flow) to take care of transversality issues. We shall not however deal with this issue here, and refer to \cite{Abouzaid-Seidel} as well as \cite{Nadler, Nadler-Zaslow} for this. 
 Note however that if $L$ is compact, $FH^*(\nu^*U,L; \lambda, \mu)$ (or $FH^*(\nu^*U,L;a,b)$) is well defined, regardless of the perturbation.   
We shall usually denote by $\overline{\mathcal L}$ ,a category containing $\mathcal L$ as a full subcategory, and its objects contain some other non-compact Lagrangian branes. Our results hold for a suitable choice of $\overline{\mathcal L}$, and most likely, including conical Lagrangians (i.e. asymptotic to some legendrian in $ST^*N$) can be included in $\overline{\mathcal L}$, but this shall be discussed elsewhere. 
 \end{rem} 
 
 \subsection*{Outline of the paper :} In the first sections we construct a quasi-presheaf (i.e. a presheaf were restrictions are only associative up to homotopy), that will after rectification is performed in section 5, be our candidate $\cF_L$ for quantizing $L$. Then in section 6 to 8 we prove that this sheaf satisfies indeed the required properties (\ref{1}),(\ref{2}),(\ref{3}). Section 9 and 10 are devoted to the identification of the pant product with a product defined on sheaves. Note that in section 10 we explain how Floer cohomology behaves by symplectic reduction. Finally we sketch some applications in the last sections. 
 
 \subsection*{Acknowledgments:} Besides the supporting institutions, this paper started during a stay at  IAS Princeton and Columbia University on the Eilenberg chair. I wish to thank Helmut Hofer, Dusa McDuff for having made this stay possible. I also want to thank Mohammed Abouzaid for useful discussions. 
 
   \section{The finite dimensional case.}
In the next section we shall study parametrized families of Lagrangian branes. We here remind the reader of the finite dimensional situation, where Lagrangians are replaced by functions, i.e. the action functional replaced by a finite dimensional function. 

 Let $(f_\tau)_{\tau\in [0,1]}$ be a one parameter family of functions, with $f_0,f_1$ being Morse functions. We write $F(\tau,x)=f_\tau(x)$. The equation $$  \frac{d}{ds} x(s)=\nabla_x F(\tau(s), x(s)), \frac{d}{ds} \tau(s)= \varepsilon (\tau(s))$$  where 
 $ \varepsilon (0)= \varepsilon (1)=0, \varepsilon (\tau)>0$ for $\tau \in ]0,1[$ and $ \varepsilon '(0)>0, \varepsilon '(1)<0$. 
This has the following property
 
 \begin{gather*} \frac{d}{ds} f(\tau (s), x(s))= \frac{\partial }{\partial \tau}f(\tau (s), x(s))\tau'(s)+  \frac{\partial }{\partial x}f(\tau (s), x(s)) \frac{d}{ds}  x(s)=\\ \frac{\partial }{\partial \tau}f(\tau (s), x(s)) \varepsilon (\tau(s))+  \vert \frac{\partial }{\partial x}f(\tau (s), x(s)) \vert^2 \end{gather*}  
 
 Note that if $\beta'(\tau)= \varepsilon (\tau)$ we have that $\tau(s)$ is the gradient trajectory of the function $\beta$. And also that given any increasing function $\tau : {\mathbb R} \to [0,1]$ satisfying $\tau (-\infty)=0, \tau(+\infty)=1$ plus  some convergence requirements at $\pm \infty$, we can find a function $ \varepsilon $ as above, such that $ \tau'(s)= \varepsilon (\tau(s))$. 
 
 Now  using  the change of variable $\tau =\tau(s)$, where $\tau$ solves $\tau'(s)= \varepsilon (\tau(s))$,  we rewrite the equation as
 $$  \frac{d}{d\tau} x(\tau)= \frac{1}{ \varepsilon (\tau)} \nabla_x F(\tau, x(\tau)), \frac{d}{d\tau} \tau=1$$
  Thus $$ \frac{d}{d\tau}F(\tau,x(\tau)) = 1 +  \frac{1}{\varepsilon(\tau)} \left \vert \nabla_x F(\tau, x(\tau))\right \vert^2$$
 Thus if the set of critical values is contained in $[-A,A]$, we have that a trajectory connecting critical points cannot spend more than time 
 $T= \frac{2 \Vert\varepsilon \Vert_{C^0} A}{\delta_0}$ outside the region where  $\left \vert \nabla_x F(t, x(t))\right \vert^2 \leq \delta_0$. 
 Thus the trajectory of $x$ follows critical points of $f_\tau$, and jumps from one critical point to another at a higher level in very short time, as is usual in slow-fast dynamics (see for example \cite{Jones}).

 Now we assume (this follows for example from the Palais-Smale condition) that for all $\tau, \delta$ there is $\eta$ such that $\vert \frac{\partial }{\partial x}f(\tau , x) \vert^2 \leq \eta$ implies  that $d(x,K_\tau)\leq \delta$ where $K_\tau$ is the set of critical points of $f_\tau$. Note that analogously (because $K_\tau$ is compact), we could conclude $d(f(\tau,x), Crit(f_\tau))\leq \delta$, where $Crit(f_\tau)$ is the set of critical values of $f_\tau$. 
 
 So assuming $ \varepsilon (s) \frac{\partial}{\partial \tau}f(\tau,x) $ is small,  the family $(f_\tau)_{\tau \in [0,1]}$ is increasing, then either $f(\tau(s),x(s))$ is close to $Crit(f)=\bigcup_{\tau\in [0,1]} Crit(f_\tau)$ or  $s\mapsto f(\tau(s),x(s))$ is increasing. 

 Thus if $\lambda$ is at distance $\delta$ from $Crit(f)$, we have that all trajectories cross level $\lambda$ in the increasing sense, that is if we have a trajectory from $(0,x)$ to $(1,y)$, then we cannot have $f(0,x) > \lambda > f(1,y)$. Therefore if for all $\tau$,  $a,b$ are not critical values of $f(\tau,\bullet)$ a trajectory of the above equation can only cross the level $a$ or $b$ in the positive sense, and this still holds if the Cerf diagram- that is the set  of pairs $(\tau, f_\tau(z))$ where $z$ is a critical point of $f_\tau$- always crosses levels $a$ and $b$ in the positive direction. 
 So counting trajectories from $(0,x)$ to $(1,y)$ yields a continuation map $MC^*(f_0^b,f_0^a) \longrightarrow MC^*(f_1^b,f_1^a)$. Again it is sufficient that the family $(f_\tau)_{\tau\in [0,1]}$ be weakly increasing at $a,b$, i.e. the Cerf diagram crosses the lines $f=a$ and $f=b$ positively. 
 
 \begin{figure}[H]
 \includegraphics[width=6cm]{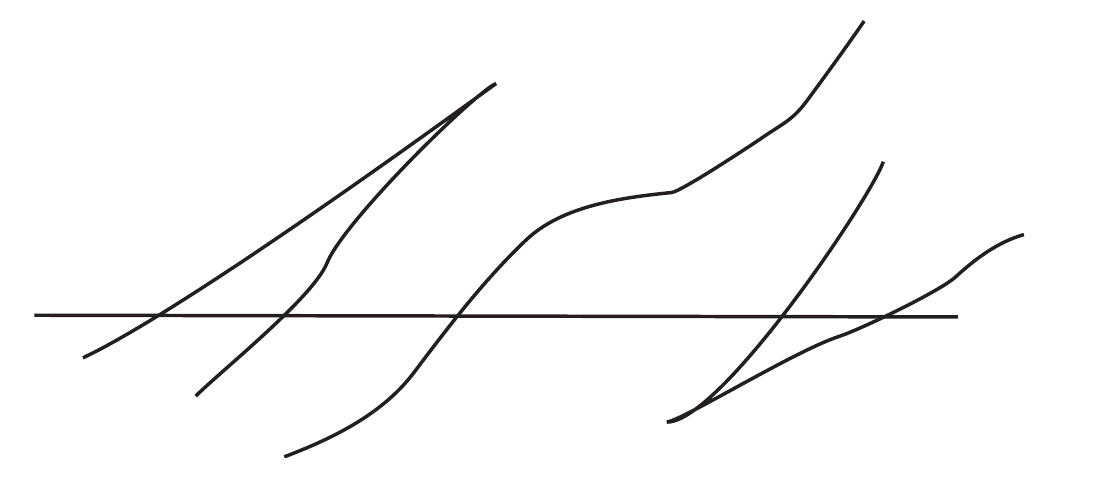}
 \caption{Positive Cerf diagram}
 \end{figure}
 \begin{prop} \label{Prop-2.1}
 Let $(f_\tau)_{\tau \in [0,1]}$ be an increasing family of smooth functions. There is then a continuation map
 $$T_{(f_\tau )} :  MC^*(f_0^b,f_0^a) \longrightarrow MC^*(f_1^b,f_1^a)$$
 Given two increasing homotopies between $f_0$ and f$_1$, they induce the same continuation map. 
 \end{prop} 
 
 \begin{proof} 
 The independence of the continuation map follows from the fact that any two increasing homotopies are homotopic among increasing homotopies. 
 \end{proof} 
 
The next section is just the translation the above in the context of Floer theory.

 \section{Recollections on Floer theory  and increasing homotopies}
 
 Let $L_0,L_1$ be elements in $\mathcal L$ such\footnote{Remember that by abuse of language, we denote, if no confusion can occur,  a Lagrangian brane and its underlying Lagrangian submanifold by the same letter.} that $L_0\pitchfork L_1$. Let $L_\tau=\varphi_H^\tau(L_0)$, where $\varphi_H^\tau$ is the Hamiltonian flow of $H$, Note that given $H$, $L_\tau=\varphi_H^\tau(L_0)$ is a well defined brane : we want to find $f_\tau$ on $L_0$
 such that $ df_{\tau}(z)=(\varphi_H^\tau)^*(\lambda)$ and by taking the derivative, 
 \begin{gather*} \frac{d}{d\tau}(df_{\tau}(z))= \frac{d}{d\tau}(\varphi_H^\tau)^*(\lambda)=(\varphi_H^\tau)^*(L_{X_H}\lambda)=\\ (\varphi_H^\tau)^*(di_{X_H}\lambda + i_{X_H}d\lambda) = d \left [(\varphi_H^\tau)^*\left ( i_{X_H}\lambda + H\right )\right ] \end{gather*} 
 
 So if $f_\tau$ solves the equation
 $$
 f_{\tau}(z)=f_0(z)+\int_0^\tau (\varphi_H^t)^*\left ( i_{X_H}\lambda + H\right ) dt
 $$
 
 we may set $f_{L_\tau}(\varphi_H^\tau(z))=f_\tau(z)$ and  this defines a function on $\varphi_H^\tau(L_0)$.  Note that $f_\tau$ depends on the choice of $H$ and changing $H$ to $H=c$ changes $f_\tau$ to $f_\tau + c\tau$.
 
 Note however that $f_\tau$ depends on the parametrization of $L_\tau=\varphi_H^t(L_0)$, while $f_{L_\tau}$ is well defined on $L_\tau$.

 \begin{defn} 
 A family $(L_\tau)_{\tau\in [0,1]}$ in $\mathcal L$ is increasing if for some parametrization of $L_\tau$, the family $(f_\tau)_{\tau\in [0,1]}$ is increasing.  
 \end{defn} 
  Indeed, note that to say that  the family $f_\tau$  is increasing only makes sense if it is defined on a fixed manifold, in our case if we are given a family $\varphi_\tau : L_0 \longrightarrow L_\tau$.
  \begin{example} 
  It is easy to see that if $(f_\tau)_{\tau\in [0,1]}$ is an  increasing family of functions on $N$, then $\tau \longrightarrow gr(df_\tau)$ is increasing. 
  
  \end{example} 
   Conversely assume we are given an increasing family $(L_\tau  f_\tau)$
   
   \begin{prop} 
   Given an increasing family $(L_\tau ,   f_\tau)$, there is a Hamiltonian $H$ such that $i_{X_H}\lambda+H$ is non-negative on $L_\tau$ and $\varphi_H^\tau(L_0)=L_\tau$
   \end{prop} 
 
 \begin{proof} 
 Indeed, let $\varphi^\tau$ be such that $f_{L_\tau}\circ\varphi^\tau=f_\tau$ is increasing. Now $\varphi_\tau$ extends to a Hamiltonian flow $\varphi_H^\tau$, and we can choose a normalization of $H$ such that $$ f_{\tau}(z)=f_0(z)+\int_0^\tau (\varphi_H^t)^*\left ( i_{X_H}\lambda + H\right ) dt$$
 Since $f_\tau$ is increasing, we get that $ \frac{d}{d\tau} f_\tau \geq 0$ hence $(\varphi_H^t)^*\left ( i_{X_H}\lambda + H\right ) \geq 0$ i.e. $ i_{X_H}\lambda + H \geq 0$ on $L_\tau$. 
 \end{proof} 
 We now have\footnote{In the Hamiltonian case, the emphasis on increasing homotopies goes back to \cite{Fl-Ho}, section 3.3.}
 
 \begin{prop}\label{Prop-2.5} Let $(L_\tau)_{\tau \in [1,2]}$ be an increasing family. Then we have a map
 $$T_{(L_\tau )}: FC^*(L_0,L_1;a,b) \longrightarrow FC^*(L_0,L_2;a,b)$$
 If two increasing homotopies are homotopic among increasing homotopies, the corresponding maps are chain homotopic, and in particular,  the induced map
  $$[T_{(L_\tau )}]: FH^*(L_0,L_1;a,b) \longrightarrow FH^*(L_0,L_2;a,b)$$
 is the same. 
 The same holds if we have an increasing family $(L_\sigma)_{\sigma \in [0,1]}$, and we have a map $$FC^*(L_1,L_2;a,b) \longrightarrow FC^*(L_0,L_2;a,b)$$
 \end{prop}
 
 \begin{proof} Since $(L_\tau, f_\tau )$ is increasing, there is $H$ as above such that $i_{X_H}\lambda + H \geq 0$ on $L_\tau$ and $\varphi_H^\tau(L_0)=L_\tau$. Now the boundary map for $FC^*(L_\tau,L_2)$ is obtained by counting solutions of 
 $$ \overline \partial_Ju_\tau(s,t)=0, u_\tau(s,0)\in L_2, u_\tau(s,1)=L_\tau$$
 But then $\varphi_H^{t\tau}(v_\tau(s,t)=u_\tau(s,t)$ where $v_\tau(s,t)$ satisfies 
 $$\overline \partial_{\widetilde J}v_\tau(s,t)=\widetilde\nabla H(t,v_\tau(s,t)), v_\tau(s,0)\in L_2, u_\tau(s,1)=L_0$$
 
  Moreover this is the Floer pseudo gradient for  the action 
  $$A_H^\tau(\gamma)=\int_0^1 [\gamma^*\lambda + \tau H(t\tau, \gamma(t))] dt$$ this part of the argument)
  where $\gamma(0)\in L_2, \gamma(1)\in L_0$. Note that the corresponding Cerf diagram is given by the $A_H^\tau(\gamma_\tau)$ where $\gamma_\tau$ is a solution of $\dot \gamma_\tau(t)=\tau X_H(\tau t , \gamma_\tau(t))$ (so that $\gamma_\tau(0)\in L_2, \gamma_\tau(1)\in L_\tau$). Thus the slope fo the Cerf diagram is given by $ \frac{d}{d\tau} A_H^\tau(\gamma_\tau)=\tau \int_0^1 (i_{X_H} \lambda+H) dt \geq 0$, and we are in the situation of the previous section (finite dimension played no role in this argument). 

  Then denoting by $\mathcal M(x,y)$ the set of such solutions
 such that $m_{L_0,L_2}(x)=m_{L_1,L_2}(y)$, and $\langle T_{(L_\tau)}x,y\rangle = \#\mathcal M(x,y)$ defines a chain map \label{Def-T}
 $$ T_{(L_\tau)}: FC^\bullet(L_0,L_1;a,b) \longrightarrow FC^\bullet(L_0,L_2;a,b)$$
 This map also depends on the choice of the monotone map $\tau : \mathbb R \longrightarrow [0,1]$ but we may take the limit as $\sup_{s\in \mathbb R} \tau'(s) \longrightarrow 0$. 
 The fact that the induced map in cohomology does not depend on the choice of the homotopy easily follows from section \ref{Floer-families}, Remark \ref{Rmk-4.2} (\ref{Rmk-4.2-2}). 
 \end{proof}

 \begin{rem}
 \begin{enumerate} 
 
 \item One should be careful, because there are several constructions of continuation maps, and they are not, in general, equivalent (or at least it is an open problem to figure out when they are) even in the finite dimensional case (see \cite{Hutchings}, \cite{Komani}).
\item  Also if we do not  take an increasing family, we can still build an induced map in Floer homology, but it will depend on the choice of the homotopy (the same of course holds for functions). This is already clear for a family of functions, for example when $f_0=f_1$, we can either take the constant path from $f_0$ to $f_1$, and the induced endomorphism of $MC^*(f^b,f^a)$ is then the identity, or consider the family $c(\tau)+f_0$ where $c(\tau)$ is chosen so that  $(c(1/2)+C)\cap  [a,b] = \emptyset$, where  $C$ is the set of critical values, and   the endomorphism of $MC^*(f^b,f^a)$ is then zero. 

\item In the Lagrangian case, we must {\bf assume} that the monotone homotopies are homotopic among monotone ones, while this is obvious for functions (as in Proposition \ref{Prop-2.1}).

\item The proposition holds under the much weaker assumption of a {\bf weakly increasing family}\index{Weakly increasing} with respect to $L_2$. This is a family $(L_\tau)_{\tau \in [0,1]}$ such that for all $z \in L_{\tau_0} \cap L_2$, we have whenever $f_{\tau_0}(z)-f_2(z)\in \{a,b\}$, that $ \frac{d}{d\tau}f_\tau (z) _{\mid \tau=\tau_0} \geq 0$. In other words, looking at the Cerf diagram \index{Cerf diagram} $C$ of the family, that is the set  of pairs $(\tau, f)$ where $f=f_\tau(z)$ with $z\in L_\tau\cap L_2$, that is generically a one dimensional manifold with singularities, we have that it crosses the lines $f=a,b$ positively. 
\end{enumerate} 
 \end{rem} 

  \section{Floer theory for families}\label{Floer-families}
  The content of this section can be found in several  papers but here we follow  Hutchings, (see \cite{Hutchings}, section 6 and also Igusa in \cite{Igusa}). Let us first consider the case of finite dimensional manifolds, before going to the Floer setting. Let $X$ be a compact manifold, $B^k$ the unit ball in $ {\mathbb R} ^{k}$, $f_0,f_1$ be Morse functions on $X$ and $F: B^k\times X \longrightarrow {\mathbb R} $ a family of functions on $X$ such that $F$ restricts to $f_0$ near the south pole (i.e. the point $(0,....,0,-1)$)  denoted $S$ and to $f_1$ in a neighborhood of the north pole (i.e. the point $(0,...,0,1)$) denoted $N$. We now consider the height function $g$ on $B^k$ and $\xi_g$ a pseudo-gradient preserving the boundary $S^{k-1}$ of $B^k$, and flowing from $S$ to $N$, and $\nabla_xF$ a vector field on $B^k\times X$ that is a pseudo-gradient for $x \mapsto F(\lambda,x)$. In other words, $$ \frac{\partial }{\partial x}F(\lambda,x) \cdot  \nabla_xF (\lambda,x) \geq C \left\vert \frac{\partial }{\partial x}F(\lambda,x) \right\vert^2$$
  We set $\tilde \xi_g$ be the lift of $\xi_g$ to $B^k\times X$, and set $Z_\tau(\lambda,x)=\nabla_xF(\lambda,x)+\tau\tilde\xi_g$.We then have
  $$(dF+\tau dg)(\lambda,x)\cdot Z_\tau= \vert \nabla_xF(\lambda,x)\vert^2+\tau dF(\lambda,x)\cdot \tilde\xi_g +\tau^2 \vert \tilde\xi_g\vert^2$$
  
  For $\tau$ large enough, this has the following properties :
  \begin{enumerate}\label{en1}
  \item it only vanishes at $(S,x)$ or $(N,y)$, where $x$ (resp. $y$) is a  critical point of $f_0$ (resp. $f_1$)
  \item \label{en1-2} away from the critical points of $g$, $g$ is strictly increasing on the trajectories of $Z_\tau$
  \item the flow of $Z_\tau$ preserves $S^{k-1}\times X$, the boundary of $B^k\times X$
  \item \label{en4} the Morse index of $(S,x)$ as a critical point of $F_\tau(\lambda,x)=f_0+\tau g$ is equal to the Morse index of $x$ for $f_0$. The Morse index of $(N,y)$ as a critical point of $F_\tau(\lambda,x)$ is equal to the Morse index of $y$ plus $k$. In standard notation $$i_M((S,x),F_\tau)=i_M(x,f_0)\; \text{and} i_M((N,y),F_\tau)=i_M(y,f_1)+k$$
  \end{enumerate} 
  
  Let $A^*, B^*$ be two differentially graded algebras, we denote by $\Mor_{dg}(A^*,B^*)$ the graded differential algebra defined by $\Mor_{dg}(A^*,B^*)_k=\bigoplus_p\Hom(A^p, B^{p+k})$ (where $\Hom(V,W)$ is the set of linear maps from $V$ to $W$) with differential $D( (f_p)_{p\in {\mathbb Z} })=(df_p+(-1)^{\vert f \vert}f_pd)_{p\in \ {\mathbb Z} }$ where$\vert f\vert =k$ for $f\in \Hom(A^p,B^{p+k})$.
  
  Assuming $Z$ satisfies some standard genericity assumptions, we have
  
\begin{defn} \label{Def-4.1}
  The vector field $Z$ defines  an element  $\varphi_B$ in $\Mor_{dg}(M^*(f_0), M^*(f_1))_{k-1}$ defined as follows:
  $\langle \varphi_B(x),y\rangle $  is the algebraic count of the number of trajectories from $(S,x)$ to $(N,y)$.
 \end{defn} 
 
 Note that the morphism depends on the choice of $Z$, but we assume $Z$ is chosen once for all. 
 
 \begin{prop} 
Let $Z$ be fixed and set $\varphi_B$to be the above map, and $\varphi_{\partial B}$ the same map, but associated to the restriction of $F$ to $\partial B \times X$. Then we have $D(\varphi_B)=\varphi_{\partial B}$
 \end{prop} 
 \begin{proof} 
 Indeed, given two critical points of index difference $2$, let us  consider trajectories of $Z$ connecting them.  The set of these trajectories divided out by the $ {\mathbb R} $-action make a $1$-dimensional manifold, and counting its boundary points with sign must yield zero. So let us describe the trajectories in this boundary : 
 
 we first have the broken trajectories, but since their projection on $B^k$ must be a trajectory from $S$ to $N$, the broken trajectory will either be a trajectory contained in $X_S=\{S\}\times X$ followed by one going from $X_S$ to $X_N=\{N\}\times X$ and the number of such trajectories counted with sign and multiplicity yields $\varphi_B\circ d_S$ or a trajectory going first from $X_S$ to  $X_N$ followed by a trajectory contained inside $X_N$ and this yields $d_N\circ \varphi_B$. 
 
 The other boundary points correspond to trajectories with projection in $\partial B^k$, and these count $\varphi_{\partial B}$. 
 \end{proof} 
 
 Note that we may consider $F$ as a  smooth map from $B^k$ to $C^\infty(X, {\mathbb R})$, we the denote $\varphi_B$ as $(F)_*$. In particular for $k=1$, we get the continuation map from Proposition \ref{Prop-2.1}.
 \begin{rem} \label{Rmk-4.2}   
 \begin{enumerate} 
 \item 
 The above construction and result also works for any smooth fiber bundle over a manifold $M$ with boundary $\partial M$. However the dynamics of a gradient vector field on $M$ preserving $\partial M$ could be much more complicated. We shall work the case of a simplex later in this section
\item \label{Rmk-4.2-2}
Applying the above to $k=2$ (in fact one should replace $B^2$ by a square), we get that given a homotopy of Lagrangians with fixed endpoints, we get that $d(F)_*-(F)_*d=(f_1)_*-(f_0)_*$, so $T_0$ and $T_1$ induce the same map in cohomology. We shall see in Proposition \ref{Prop-5.3} that when the homotopy is increasing, this also holds for the relative case (i.e. $FH^*(L_\tau,L_2;a,b)$) and proves Proposition \ref{Prop-2.5}.
 \end{enumerate} \end{rem} 
 \renewcommand \l {{\lambda}}
 Note that the same can be done if we replace the ball by the $k$ simplex, $\Delta^k= \{(\lambda_1,....,\lambda_k)\mid 0\leq \lambda_1\leq ...\leq \lambda_k\leq 1\}$, 
 and $g(\l_1,...,\l_k)=\sum_{j=1}^k \frac{\lambda_j^2}{2}- \frac{\lambda_j^3}{3}$ and $\xi_g( \l_1,...,\l_k)=\sum_{j=1}^k \l_j(1-\l_j)\frac{\partial }{\partial \lambda_j}$ . But then the critical points of $g$ are the $u_p$ with coordinates $\l_j=0$ for $j\leq p+1$ and $\l_j=1$ for $j\geq p$. In other words, the critical points are the vertices of $\Delta^k$, and $u_k$ has index $k$. Note also the the vector field $\xi_g$ is actually defined in a neighborhood of $\Delta^k$ and preserves $\Delta^k$. Indeed, the boundary of $\Delta^k$ is the union of its faces, given by $\lambda_1=0$ for the first face, by $\lambda_j-\lambda_{j+1}=0$ (for $1\leq j \leq k-1$), and  $\lambda_k=1$ for the last face. Since the flow of $\xi_g$ is given by the equations  $$ \dot \lambda_j = \lambda_j (1-\lambda_j)$$ hence 
 $$\dot\l_j-\dot\l_{j+1}=(\l_j-\l_{j+1})(1-\l_j-\l_{j+1})$$ and the faces $\l_j=\l_{j+1}$ is indeed preserved.  We write $u_0=\overline 0 = (0,....,0), \overline 1= (1,1...,1)=u_k$.
 
 Denoting by $i_p(\sigma)$ the simplex  obtained by restricting $\sigma$ to  $\l_1=....=\l_p=0$, $t_p(\sigma)$ the restriction to $\l_{p+1}=....=\l_k=1$ and  $\partial_j(\sigma)$ the restriction to $\lambda_j=\l_{j+1}$. We also set $i(\sigma)=i_0(\sigma)=\sigma(u_0)$  and $t(\sigma)=t_{k}(\sigma)=\sigma(u_k)$.
 Now we have a differential,  $\partial$ on the set of simplices $\partial \sigma = \sum_{j=0}^{k} (-1)^j \partial_j\sigma$. We actually need another differential, defined by $\delta \sigma = \sum_{j=1}^{k-1} (-1)^j \partial_j\sigma$ so that $\partial\sigma = \delta \sigma + \partial_0\sigma + (-1)^k\partial_k\sigma$. 

\begin{defn} Let $\sigma : \Delta^k \longrightarrow C^\infty(X,{\mathbb R} )$ be a smooth simplex. We define $(\sigma)_*$, as the map 
$$(\sigma)_* : MC^*(f^b,f^a) \longrightarrow MC^{*-k+1}(f^b,f^a)$$

\end{defn} 
 \begin{prop}[Compare with \cite{Igusa}, cor. 4.11]\label{Prop-5.2} We have the identity
  $$d\sigma_*+ (-1)^k \sigma_*d= (\delta\sigma)_*+\sum_{j=1}^{k-1} (-1)^{j-1}(i_j\sigma)_*\circ (t_j\sigma)_*$$
 
  \end{prop} 

\begin{proof} 
Indeed, if we look at how $1$-dimensional families of trajectories from $X_{u_0}$ to $X_{u_k}$ break-up, we find either a trajectory on $X_{u_0}$ followed by a trajectory from $X_{u_0}$ to $X_{u_k}$ corresponding to $\sigma_*\circ d$, or a trajectory from $X_{u_0}$ to $X_{u_k}$ followed by a trajectory in $X_{u_k}$ corresponding to $d\circ \sigma_*$, or a trajectory with projection in a pair of faces $i_p\sigma$ and then $t_p\sigma$, corresponding to $(i_pj\sigma)_*\circ (t_p\sigma)_*$, or a trajectory contained in one of the faces connecting $u_0$ to $u_k$, and this corresponds to $(\delta\sigma)_*$

\end{proof} 
  \begin{figure}[H]
 \begin{center}  \begin{overpic}[width=6cm]{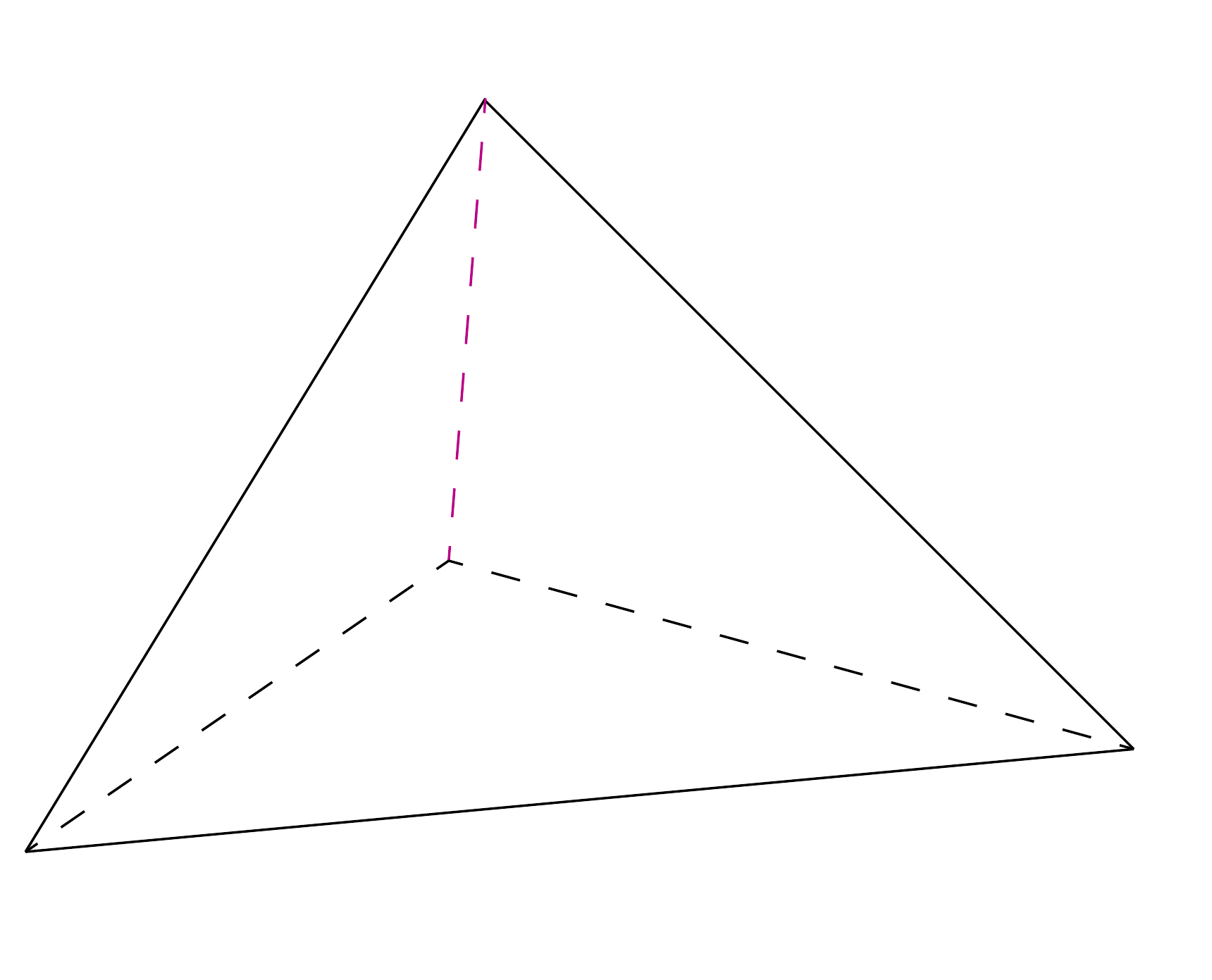}
 \put (35,75) {$u_1$}  
  \put (0,5) {$u_0$}  \put (35,27) {$u_2$}   \put (95,10){$u_3$} \put (90,5) {}
\end{overpic}
\end{center}
\caption[A\newline B \newline C]{\begin{minipage}[t]{.8\linewidth} Terms of the expression of Proposition \ref{Prop-5.2}\\ $(u_0,u_1,u_3); (u_0,u_1,u_3)\longrightarrow \delta\sigma$ \\
$(u_0,u_1)\circ (u_1,u_2,u_3) \longrightarrow (i_1\sigma)\circ (t_2\sigma) $ \\ 
$(u_0,u_1,u_2)\circ (u_2,u_3) \longrightarrow (i_2\sigma)\circ (t_1\sigma)    $ \end{minipage}}
\end{figure}

Now we just repeat the above for the Floer complex. 
\begin{prop}\label{Prop-5.3} 
 Let $\sigma : \Delta^k \longrightarrow \mathcal L$ be a smooth map, and assume that on each
  trajectory of $\xi_g$ the Lagrangians are increasing. Setting $L_0=\sigma(u_0), L_1=\sigma(u_k)$, we get a map
  $$\sigma_*:  FC^*(L,L_0;a,b) \longrightarrow FC^{*-k+1}(L,L_1;a,b)$$
We have the identity in the Floer complex
  $$d\sigma_*+ (-1)^k \sigma_*d= (\delta\sigma)_*+\sum_{j=1}^{k-1} (-1)^{j-1}(i_j\sigma)_*\circ (t_j\sigma)_*$$
 The same holds if we have a family  $\sigma : \Delta^k \longrightarrow \mathcal L$ with $L_1=\sigma(u_0), L_2=\sigma(u_k)$, and  increasing  on each trajectory of $\xi_g$, and consider the map $$(\sigma)_*: FC^*(L,L_0;a,b) \longrightarrow FC^{*-k+1}(L,L_1;a,b)$$
  \end{prop} 
\begin{proof} The proof follows the same lines as Proposition \ref{Prop-2.5} and \ref{Prop-5.2}.
Indeed, set $L_\lambda= L_{\sigma(\lambda )}$. We look at a one parameter family of trajectories of the mixed Floer equation,  \begin{equation*}\left \{ \begin{aligned}   \frac{\partial}{\partial s}u_{\lambda(s)}(s,t)&=&-J \dot u_{\lambda(s)}(s,t)  \\ \frac{\partial }{\partial s} \lambda (s) &=& \tau \xi_g(\lambda(s))\\ \end{aligned} \right . \end{equation*}

with conditions $u_\lambda(0)\in L_\lambda, u_\lambda(1)\in L$ and $\lim_{s\to \pm \infty}u_\lambda (s,t)\in L\lambda\cap L$ and $\lim_{s\to - \infty}\lambda(s)=u_0, \lim_{s\to + \infty}\lambda(s)=u_k$. 
We saw in the proof of Proposition \ref{Prop-2.5} that this reduces to a standard Floer equation, so has the same properties, in particular as far as trajectory breaking is concerned. 
The fact that the $ L_{\lambda (s)}$ are increasing implies the maps are all well defined from $FC^*(L,L_0;a,b)$ to $ FC^{*-k+1}(L,L_1;a,b)$. Now the argument is the same as in the case of functions : for a one parameter family of such Floer trajectories, they will break up in the same way, yielding the same terms as above. 

\end{proof}

\section{Rectification of the Floer quasi-sheaf }\label{Rectif}

Let $\Cat{F(N)}$ be the category of smooth functions on $N$ with the usual partial order. A $k$-simplex in $\Cat{F(N)}$ is an ordered $(k+2)$-uple\footnote{we think of $\sigma$ as a $k$-simplex of paths from the first vertex to the last vertex, this explains the shift in dimension : for example a map $\sigma : \Delta^1 \longrightarrow X$ defines a path in $X$ so a $0$-simplex in the path space of $X$.}, denoted $\sigma=(f_0\leq ...\leq f_{k+1})$ and identified to the map $\sigma: \Delta^{k+1} \longrightarrow C^\infty(N)$ given by 
$\sigma(\lambda_1, \lambda_2,...,\lambda_{k+1})=\sum_{j=0}^{k+1} (\lambda_{j+1}-\lambda_j) f_{k+1-j}$ (we set $\l_0=0, \l_{k+2}=1$), so that $\sigma(u_j)=f_j$. We also set $k=\vert \sigma\vert$. 
Now if we have a trajectory $\lambda(t)$  of $\xi_g$, then $\sigma (\l (t))$ is increasing. According to Proposition \ref{Prop-5.3}, we get maps
$$\sigma_* :  FC^*(\Gamma_{f_{k+1}},L;a,b) \longrightarrow FC^{*-k+1}(\Gamma_{f_0},L;a,b)$$ 
such that 
$$d\sigma_*+ (-1)^k \sigma_*d=(\delta\sigma)_*+\sum_{j=1}^{k-1} (-1)^{j-1} (i_j\sigma)_*\circ (t_j\sigma)_*$$
and if we set by convention\footnote{\label{fn}Caveat : if $\sigma$ is a degenerate $k$-simplex, i.e. with identical vertices,  $(\sigma)_*=d$ for $k=0$,  for $k\geq 1$, $(\sigma)_*=\id$ for $k=1$ and $(\sigma)_*=0$ for $k\geq 2$. Thus for any $k$-simplex, $(i_k\sigma)_*=d$.}  for the constant 0-simplex, $\sigma_*=d$, we may rewrite this as
$$\sum_{j=0}^{k} (-1)^{j} (i_j\sigma)_*\circ (t_j\sigma)_*= (\delta\sigma)_*$$

Now set $V^*(f)=FC^*(\Gamma_f,L;a,b)$ and we set 
\begin{defn} Let us consider $C_*(\Cat{F(N)}, V^*)$ to be the vector space of {\bf $V^*$-valued chains}\index{chains!$V^*$-valued} generated by the formal objects $\sigma\otimes x$ where $x \in V^*(t(\sigma))$, endowed with differential
$$D(\sigma\otimes x)=(\delta\sigma \otimes x )+(-1)^{\vert \sigma\vert} \sigma\otimes dx + \sum_{j=1}^{\vert \sigma\vert} (-1)^{j}(i_j\sigma)\otimes (t_j\sigma)_*(x)$$
\end{defn}  

\begin{rem} 
Note that since $\sigma$ is monotone, so is $i_p\sigma$, and the maps $(i_p\sigma)_*$, so that $D$ respects the action filtration. Note that $\delta^2=0$. We set $\delta(\sigma\otimes x)=\delta\sigma \otimes x$
\end{rem} 
We refer to  \cite{Brown} and \cite{Igusa2} for what follows. Note that the definition of $D$ is essentially the cobar construction of Adams (\cite{Adams}) or the twisting cochain construction of Brown (\cite{Brown}). 

Note that $\sigma\otimes x \mapsto \delta \sigma \otimes x= \sum_{j=1}^{\vert \sigma\vert}(-1)^j\partial_j\sigma \otimes x$ is well defined in $C_*(\Cat{F(N)}, V^*)$  because $i(\partial_j\sigma)=i(\sigma), t(\partial_j\sigma)=t(\sigma)$ for $k-1\geq j\geq 1$.

 \begin{prop} 
 The differential $D$ on the complex $C^*(\Cat{F(N)},V^*)$satisfies $D^2=0$. 
 \end{prop} 
 \begin{proof} 
 For simplicity, we denote the generic simplex as $(0,..,k+1)$, so for example, $\partial_j\sigma=(0,...,j-1,j+1,...,k+1), i_p\sigma =(0,...,p)$.
  We have \begin{gather*} D^2(\sigma\otimes x)= D\left (\delta\sigma + \sum_{p=0}^{k+1}(-1)^{p}i_p\sigma\otimes (t_p\sigma)_*x \right )=\\
 \delta^2(\sigma)\otimes x+  \sum_{p=0}^{k+1}(-1)^{p}\sum_{l=1}^{p-1}(-1)^l\partial_l i_p\sigma\otimes (t_p\sigma)_*x  +  \sum_{l=1}^{k}(-1)^{l}\sum_{p=0}^{k}(-1)^p i_p\partial_l\sigma\otimes (t_p\partial_l\sigma)_* x+ \\ \sum_{p=0}^{k+1}(-1)^{p}\sum_{q=0}^{k+1-p}(-1)^{q}i_pi_q\sigma\otimes (t_pt_q\sigma)_*(t_q\sigma)_*
 \end{gather*} 
 to compute the above sum, we notice that $\delta^2=0$, and denoting $\sigma=(0,...,k+1)$, we can gather the terms of the form $\tau\otimes y$ where $\tau$ is either of the form $(0,....,l-1,l+1,...,k+1)$ or of the form $(0,.....,p)$. The above sum is made of three  double sums, and the contribution of each of them to the coefficient of the  $\tau$ is equal
 \begin{enumerate} 
 \item for $\tau=(0,...,l-1,l+1,...,p)$ to
$$ (-1)^{p}(-1)^{l}(t_p\sigma)_*x + (-1)^{p}(-1)^{l-1}(t_p\sigma)_*x + 0 =0 $$
 
 \item for  $\tau=(0....,p)$
\begin{gather*} 0+ (-1)^{p}\sum_{l\geq p}(-1)^{l}(p,...,l-1,l+1,...k+1)_*x+\sum_{q\geq p}(-1)^{p+q}(p,...,q)_*(q,...,k+1)_* x \end{gather*} 
and setting  $\rho=(p,...,k+1)$, this is equal to 
 $$(\delta\rho)* + \sum_{m=0}^j (-1)^{q}(i_m\rho)_* (t_m \rho)_*x=0$$ 
 according to Proposition \ref{Prop-5.3}. 
 \end{enumerate} 
 \end{proof}

 Now for a smooth function $f$ (i.e. an object in $\Cat{F(N)}$), we can consider the functor
 $$ f\mapsto \bigoplus_{f=f_0\leq f_{1} \leq ..... \leq f_{p+1}}(f_0\leq f_{1} \leq ..... \leq f_{p+1}) \otimes V^q(f_{p+1})$$
where we denote $\widehat V_{(p,q)}(f)$ the above summand for index $(p,q)$,  $\widehat {V_n}=\bigoplus_{p+q=n}\widehat V_{(p,q)}$, and finally $$\widehat {FC_L^{p,q}}(f)=\bigoplus_{f_0\leq f_{1} \leq ..... \leq f_{p+1}}(f_0\leq f_{1} \leq ..... \leq f_{p+1}) \otimes FC^q(L, \Gamma_{f_{p+1}}) $$
Now if $g\leq f$ we have a map $\widehat {FC_L^n}(f)\longrightarrow \widehat {FC_L^n}(g)$ obtained by sending 
$(f\leq f_{1} \leq ..... \leq f_{p+1}) \otimes FC^q(\Gamma_{f_{p+1}},L)$ to  $(g\leq f_{1} \leq ..... \leq f_{p+1}) \otimes FC^q(\Gamma_{f_{p+1}},L)$ since if $f\leq f_1$ we have $g\leq f_1$. This induces a map $$\rho_{f,g}: \widehat {FC_L^n}(f) \longrightarrow \widehat {FC_L^n}(g)$$  and clearly we have
$\rho_{g,h}\circ \rho_{f,g}=\rho_{f,h}$. 
We therefore have defined a functor  $\Cat{F(N)} \longrightarrow \Cat{Ch^b}$  given by $f \mapsto \bigoplus_n\widehat {FC_L^n}(f) $  that is a presheaf on the category $\Cat{F(N)}$. 
Now we have a natural map
$$ FC_L^n(f) \longrightarrow \widehat {FC_L^n}(f)$$
given by $x \mapsto (f\leq f)\otimes x$ where $(f\leq f)$ is a $0$-simplex. 

\begin{prop}
The map  $$ FC_L^n(f) \longrightarrow \widehat {FC_L^n}(f)$$  given by $x \mapsto (f\leq f)\otimes x$ induces a quasi isomorphism. 
\end{prop} 
\begin{proof} Remember that the differential is $$D(\sigma\otimes x)=\delta\sigma \otimes x + \sum_{j=0}^{\vert \sigma\vert} (-1)^{j}(i_j\sigma)\otimes (t_j\sigma)_*(x)$$
Let us consider the spectral sequence associated to the filtration by $p$ of $\widehat {FC_L^n}(f)=\bigoplus_{p,q}\widehat {FC_L^{p,q}}(f)$. The map $d_0: E_0^{p,q} \longrightarrow E_0^{p,q+1}$  is given by $\sigma\otimes x \longrightarrow \sigma \otimes dx$, so $E_1^{p,q}$ is generated by the $\sigma\otimes u$ where $u \in FH^q(L, \Gamma_{f_{p+1}};a,b)$  (here $\sigma=(f_0\leq ...\leq f_{p+1})$), and $d_1(\sigma\otimes u)=\delta\sigma \otimes u +(-1)^k i_k\sigma\otimes (t_{k}\sigma)_*u$. 

\begin{lem} We have $E_2^{0,q}=FH^q(\Gamma_f,L;a,b)$ and $E_2^{p,q}=0$ for $p\neq 0$. 
\end{lem} 
\begin{proof} 
First of all if  $(f_j)_{j\geq 0}$ an increasing  sequence, the simplices with vertices in the sequence are preserved by the differentials. We thus consider a fixed sequence, and to simplify notations, we write $j$ instead of $f_j$, so we write $(j_0\leq j_1\leq  .... j_k\leq j_{k+1})$ instead of $(f_{j_0} \leq f_{j_1}\leq..... \leq f_{j_{k}}\leq f_{j_{k+1}})$, and write $\sigma(l)=j_l$. Now we compute the cohomology of the map $d_1: \sigma\otimes u \mapsto \delta\sigma\otimes u + (-1)^k i_k\sigma\otimes (t_{k}\sigma)_*u$. 

\begin{slem} 
Let  us consider the chain complex $\Gamma_*$ where $\Gamma_k(p,q)$ is generated by $(j_0\leq ..... \leq j_{k+1})$ with $p\leq j_0\leq  j_{k+1}\leq q$, and  the map $\delta$ given by  $$\delta(j_0\leq .... \leq j_{k+1})=\sum_{l=1}^{k} (-1)^{k-l} (j_0\leq ...\leq j_{l-1}\leq j_{l+1} \leq ...\leq j_{k+1})$$ Then $H_*(\Gamma_*,\delta)$ vanishes. Same holds for $$D_1(j_0\leq ..... \leq j_{k+1})\otimes u=\delta(j_0\leq ..... \leq j_{k+1})\otimes u + (j_{0}\leq....\leq j_{k})\otimes (j_{k}\leq j_{k+1})_*u$$ except in degree $0$, where  $H_0$ equals $\mathbb F$. 
\end{slem} 
\begin{proof} 
For $\delta$, omitting the endpoints $j_0, j_{k+1}$, we recover the usual $\partial$ operator on simplices. Therefore for  $k\geq 1$ the homology of $\Gamma_k$ vanishes. 
Now  $\delta(j_0\leq j_1) = 0$, so  $H_0$ is generated by $(j_0\leq j_1)$. 

Now for $D_1$.  We consider the lexicographic order on $k$-uples starting from the right. In other words, $(j_0\leq j_1\leq ....\leq j_p)< (j'_0\leq j'_1\leq ....\leq j'_p)$ if and only if $j_p\leq j'_p$ or $j^{}_l=j'_l$ for $q\leq l \leq p$ and $j^{}_{q-1}<j'_{q-1}$. Given an element $(j_0\leq .... \leq j_{k+1})$, we have $(j_{0}\leq....\leq j_{k}) \leq (j_{0}\leq...\leq j_{l-1}\leq j_{l+1}\leq ...\leq j_{k+1})$, so $D_1(j_{0}\leq....\leq j_{k+1}\otimes u)$ is made of two terms, $\delta (j_{0}\leq....\leq j_{k})\otimes u$ and a term strictly less than it, except in degree $0$, where $D_1(j_0\leq j_1)\otimes u= 0$. As a result the cohomology in degree greater than $1$ vanishes, and the $0$-cohomology is generated by $(j_0\leq j_1)\otimes u$. But $D_1(j_0\leq j_0\leq j_1)\otimes u=(j_0\leq j_1)\otimes u - (j_0\leq j_0)\otimes (j_0\leq j_1)_*u$ so each element is cohomologous to an element of the form $(j_0\leq j_0)\otimes v$. 
\end{proof} 
\end{proof} 
End of proof of the proposition: According to the lemma, the spectral sequence for $\widehat {FC}_L^n(f)$ degenerates at the $E_2$ term and the only nonzero term is $E_2^{0,q}$. But this corresponds to $\widehat {FC}_L^{0,n}(f)=FC_L^n$ hence the map $$\widehat {FC}_L^{0,n}(f) \longrightarrow \widehat {FC}_L^{n}(f)$$ yields an isomorphism of $E_2$ terms hence an isomorphism in cohomology. 
\end{proof}  
 \begin{rem} 
 One can check directly, that anything is cohomologous to something of the form $(f_0\leq f_0)\otimes u$, for $du=0$ and then $$D_1((f_0\leq f_0\leq f_1)\otimes u)=(f_0\leq f_1)\otimes u + (f_0\leq f_0\leq f_1)\otimes du - (f_0\leq f_0)\otimes (f_0\leq f_1)_*u$$
 Thus $(f_0\leq f_1)\otimes u$ and $(f_0\leq f_0)\otimes (f_0\leq f_1)_*u$ are cohomologous, so anything is cohomologous to something of the form $(f_0\leq f_0)\otimes u$ for $u\in FC^*(\Gamma_{f_0},L)$.
 \end{rem} 
Since all these constructions respect the action filtration, we get

 \begin{prop} \label{Prop-4.10}
  
Let $ \widehat{FC}^n_L(f;\lambda)$ be obtained by the above construction by replacing ${FC}^n_L(f)$ by the submodule generated by elements of action less than $\lambda$. Then similarly 
 
 $$ FC^n_L(f, \lambda) \longrightarrow \widehat{FC}^n_L(f, \lambda)$$ is a quasi-isomorphism. 
 \end{prop} 
 
 We denote by $\Cat {({\mathbb R}, \leq)}$ the category of the poset $ ({\mathbb R}, \leq)$ : objects are real numbers and there is a unique morphism from $a$ to $b$ if and only if $a\leq b$. 
 
 \begin{thm} \label{Thm-4.8}
 
 The map $$(f,\lambda) \mapsto \widehat{FC}^n_L(f, \lambda)$$ yields a presheaf over the category $\Cat{F(N)}\times \Cat {({\mathbb R} , \leq)}$ (that is a functor $\Cat{F(N)}\times \Cat {({\mathbb R} , \leq)} \longrightarrow \Cat{Ch^b}$ ) with cohomology $H^*(\widehat{FC}^\bullet_L(f, \lambda))=FH^*_L(f;\lambda)=FH^*(\Gamma_f,L;\lambda)$. 
 \end{thm} 
 
Note that the cohomology in the proposition, is just the cohomology of the complex for $f,\lambda$ fixed (it is not any kind of "sheaf cohomology or Cech cohomology). 

It will be important to point out that the map $\widehat {FC}^*_L(f) \longrightarrow \widehat {FC}^*_L(g)$ for $f\leq g$ is induced by the continuation map $FC^*(\Gamma_f,L; \lambda) \longrightarrow FC^*(\Gamma_g,L; \lambda)$.

\begin{rem} We can replace $FC^*(L_1,L_2;a,b)$ by $FC_*(L_1,L_2;a,b)$ the set of classes with action above $a$, modulo those with action above $b$. Then the map $\delta^*$ preserves this quotient, and as a result, setting $V_q(g)=FC_*(\Gamma_g,L;a,b)$, we get that the set of element of the form $(g_0\geq g_1\geq g_2,\geq ....\geq g_{p+1})\otimes V_q(g_{p+1})$ will also yield a rectification $\widehat {FC}_n^L(f,\lambda)$. After sheafification, this is related to the Verdier dual of the sheaf obtained from the above construction, and will be useful in relating the quantization of $L$ with the quantization of $-L$ (see section \ref{Quant-L}).

\end{rem} 
 \section{Floer cohomology and generating functions}
 
Let $S$ be a generating function quadratic at infinity, that is $S: N\times {\mathbb R}^{q} \longrightarrow {\mathbb R} $ such that $\Sigma_S=\{(x,\xi)\mid \frac{\partial S}{\partial \xi}(x,\xi)=0$ and $$L_S=\{(x, \frac{\partial S}{\partial \xi}(x,\xi))\mid  \frac{\partial S}{\partial \xi}(x,\xi)=0\}$$ assuming  $S(x,\xi)=Q(\xi)$ is a non-degenerate quadratic form for $\vert \xi \vert$ large enough. Note that $S$ is defined up to fiber preserving diffeomorphism, stabilization (i.e. adding a quadratic form in new variables) and addition of a constant (see \cite{Viterbo-STAGGF}, \cite{Theret}), but this last freedom can be removed for a brane, by requiring that  $f_L(x, \frac{\partial S}{\partial \xi}(x,\xi))) =S(x, \frac{\partial S}{\partial \xi}(x,\xi))) $ on $\Sigma_L$.

The secret goal of this section - in the situation when $L=L_S$ has a G.F.Q.I., $S$-  is to identify in the derived category $\cF_L$ and $\cF_S=(R\pi)_*(k_{U_S})$, where $U_S=\{(x,\xi,t) \in N\times {\mathbb R}^k\times {\mathbb R}  \mid S(x,\xi)<t \}$, $k_{U_S}$ being the constant sheaf over $U_S$. However we have not yet defined $\cF_L$ (this will be done in the next section), but we in fact just need to identify Floer cohomology and generating function cohomology in a setting slightly more general than in \cite{Viterbo-FCFH2}. In the last two propositions, we localize this result to the case where $L=L_S$ only near some fiber $T_{x_0}^*N$. 

 From \cite{Viterbo-FCFH2}, we  have
 
 \begin{thm} \label{thm-5.1}
 For $L_0=L_{S_0}, L_1=L_{S_1}$ be Hamiltonianly isotopic to the zero section and generated by two G.F.Q.I., $S_0,S_1$. We then have an isomorphism  $$FH^*(L_0,L_1;a,b)=H^{*-i}((S_0\ominus S_1)^b,(S_0\ominus S_1)^a)$$
where $(S_0\ominus S_1)(x,\xi_0,\xi_1)= S_0(x,\xi_0)-S_1(x,\xi_1)$. 
 \end{thm}  
 \begin{proof} 
 For the original proof we refer to \cite{Viterbo-FCFH2}. However, since we shall use the mechanism of the proof later on, we summarize  here  a slightly simpler proof. We shall assume first $L_0=0_N$, the general case will be easily reduced to this one. 
 
 First let $H(t,q,p)$ be a Hamiltonian with flow $\varphi^t$, and assume $\varphi^1(0_N)=L_S$. Then we can consider the usual Floer cohomology of $L_S$ and $0_N$, obtained by taking for generators the intersection points of $L$ and $0_N$ and Floer trajectories the holomorphic strips between two such intersection points $z_\pm$, with boundary in $L$ and $0_N$. This is denoted by $FH^*(L,0_N;a,b)$. But there is also the Floer complex $FC^*(0_N,L_1,H;a,b)$, where generators are  points $x$ in $0_N$ such that $\varphi^1(x)\in L_1$ with action\footnote{the action is given by $\int_0^1 p\dot q-H(t,q(t),p(t))]dt +f_L(q(1),p(1))$, where $(q(t),p(t))=\varphi^t(q(0),p(0))$} in $[a,b[$, and the Floer trajectories are the  strips which are solutions of $\overline\partial_J u(s,t)=\nabla H(t,u(s,t))$ converging to the trajectories $\varphi^t(x_\pm)$ where $x_\pm\in L_0$, $\varphi^1(x_\pm)\in L_1$ and such that $u(s,0)\in L_0, u(s,1) \in L_1$. 
 
 \begin{lem} \label{lem-7.2}
 For suitable almost complex structures, we have a chain homotopy equivalence
 $$FC^*(0_N, \varphi^1(0_N);a,b) \longrightarrow  FC^*(0_N, 0_N,H;a,b)$$  
 \end{lem} 
 \begin{proof} 
 Clearly the generators are the same. We shall identify  Floer trajectories, but for different almost complex structures. Indeed, 
  if $u(s,t)$ is a Floer trajectory satisfying $u(s,0)\in L_0, u(s,1)\in L_1$ and $\overline\partial_Ju=0$, we can take $v(s,t)=\varphi^{-t}(u(s,t))$ and this will satisfy
$\overline \partial_{\tilde J}v=\nabla H(t,u(s,t)), v(s,0)\in L_0, v(s,1)\in L_1$, where $\tilde J(t,z)=(\varphi^{-t})_*J$. 

 \end{proof} 
 
 In the sequel we denote $FC^*(0_N,0_N,H;a,b)$ by $FC^*(H;a,b)$. 
 Note that similarly, if we consider a smooth function $f$ on $N$ and consider $FC^*(\Gamma_f,0_N;a,b)$ it will be chain equivalent to 
 $FC^*(0_N,0_N, f;a,b)$.

Indeed for $FC^*(0_N,0_N, f;a,b)$, the  Floer trajectories are given by $\overline\partial_J v = df(q)$, that can be rewritten, if we take standard coordinates and  $J$ is the standard complex structure associated to the connection $\nabla$, as 
\begin{equation*}\left \{ \begin{aligned}  \frac{\partial q}{\partial s} &= \nabla_t p + \nabla f(q)  \\ \nabla_s p &=- \frac{\partial q}{\partial t}\\ \end{aligned} \right . \end{equation*}
The solutions are given by $p(s,t)= 0$ and $q(s,t)=q(s)$ where $\frac{d}{ds} q(s) = \nabla f(q(s))$, and we exactly get the gradient trajectories of $f$. 
Note that $\nabla_s\nabla_s+ \nabla_t\nabla_t=\Delta p +(1^{st} {\rm order\; terms}) = d^2f(q)  \frac{\partial p}{\partial s}$ hence with our boundary conditions, we may apply the maximum principle, and we have necessarily $p\equiv 0$. 
Using the standard notation $f^\lambda=\{x\ in N \mid f(x)\leq \lambda \}$,  we thus just proved, 

\begin{lem} 
For a suitable choice of the almost complex structures, we have chain complex isomorphisms $$FC^*(0_N,\Gamma_f,;a,b) \longrightarrow FC^*(0_N,0_N, f;a,b)\longrightarrow MC^*(f^b, f^a)$$

\end{lem} 

Now consider on $T^*N\times T^*( {\mathbb R} ^k)$ the flows $\Phi_{t_0}^{t_1}$ of $H(t,q,p)$ (i.e. if $\varphi_{t_0}^t$ is the flow of $H$, then $\Phi_{t_0}^t = \varphi_{t_0}^t\times \id$) and $\Psi_S^t$ of $S(q,\xi)$. 

That is 
for the first flow, it is given by the equations 
\begin{equation*}\left \{ \begin{aligned} \dot q (t)&=\frac{\partial}{\partial p} H(t,q(t),p(t)) \\ \dot p (t)&=-\frac{\partial}{\partial q} H(t,q(t),p(t))  \\ \dot\xi (t)& =0\\ \dot\eta (t) &=0 \\ \end{aligned} \right . \end{equation*}
 and for the second
 
 \begin{equation*}\left \{ \begin{aligned} \dot q (t)&=0 \\ \dot p (t)&=\frac{\partial}{\partial q} S(q,\xi)  \\ \dot\xi (t)& =0\\ \dot\eta (t) &=\frac{\partial}{\partial \xi} S(q,\xi) \\ \end{aligned} \right . \end{equation*}
 that is $\Psi_S^t(q,p,\xi,\eta)=(q,p+t\frac{\partial}{\partial q} S(q,\xi), \xi, \eta + t\frac{\partial}{\partial \xi} S(q,\xi))$.

 The flow $\Xi^t=\Psi_S^{-t}\Phi_0^t$ is generated by the Hamiltonian $K_{H,S}(t,q,p,\xi,\eta)$ defined by 
 $$K_{H,S}(t,q,p,\xi,\eta)= H(t,q,p+\frac{\partial}{\partial q} S(q,\xi))-S(q,\xi)$$
 and the points of $\Xi^1(0_{N\times {\mathbb R}^k} ) \cap 0_{N\times {\mathbb R}^k} $ are in one-to-one correspondence with $\varphi_0^1(0_N)\cap L_S$, since $\Phi_0^1(0_{N\times {\mathbb R}^k})= \varphi_0^1(0_N)\times 0_{ {\mathbb R}^k}$ and $\Psi_S^1(0_{N\times {\mathbb R}^k})=graph (dS)$. 
 
 The  Floer trajectories are given by the following equations, where $K=K_{H,S}$,  $u(s,t)=(q(s,t),p(s,t))$ and $\zeta (s,t)=(\xi(s,t), \eta(s,t))$
 
 \begin{equation*}\left \{ \begin{aligned} \overline\partial u(s,t) (t)&=\nabla_uK(t,u,\zeta) \\ \overline\partial \zeta &=\nabla_\zeta K(t,u,\zeta) \\ \end{aligned} \right . \end{equation*}
 
 Now if $S= \frac{1}{2} \langle A \xi, \xi\rangle $, this reduces to 
  \begin{equation*}\left \{ \begin{aligned} \overline\partial u(s,t) &=\nabla H(t,u) \\ \frac{\partial }{\partial s}\xi (s,t)  &=\frac{\partial }{\partial t}\eta (s,t)+A\xi
   \\ \frac{\partial }{\partial s}\eta (s,t)  &=-\frac{\partial }{\partial t}\xi(s,t) \end{aligned} \right . \end{equation*}

   But the last two equations imply $\Delta \eta-A \frac{\partial \eta}{\partial t} =0$ and the boundary conditions, are as follows :
   
    for $\xi_\pm=\lim_{s\to \pm\infty}\xi(s,t), \eta_\pm=\lim_{s\to \pm\infty}\eta(s,t)$ will be $$\frac{\partial }{\partial t}\eta_\pm (t)+A\xi_\pm=0, \frac{\partial }{\partial t}\xi_\pm(t)=0; \eta_\pm(0)=\eta_\pm(1)=0$$  this implies $\xi_\pm\equiv 0$.  But then $\Delta \eta=0$, hence we also have $\eta_\pm\equiv 0$ and then $\xi_\pm\equiv0$, and  for $t=0,1$, we have $\eta(s,0)=\eta(s,1)=0$. 
    
    Using the maximum principle we get that $\eta\equiv 0$ hence $\xi \equiv 0$. 
    
    This implies that the trajectories of the  the above Floer equation are in one-to-one correspondence with those of  $ \overline\partial u(s,t) =\nabla H(t,u)$. 
   
   On the opposite end, if $H\equiv 0$, we get 
   
    \begin{equation*}\left \{ \begin{aligned} \frac{\partial }{\partial s}q(s,t)  &= \nabla_t p(s,t)- \frac{\partial }{\partial q}S (q(s,t),\xi(s,t)) \\ 
    \nabla_s p(s,t)&=- \frac{\partial }{\partial t}q(s,t) \\ \frac{\partial }{\partial s}\xi (s,t)  &=\frac{\partial }{\partial t}\eta (s,t) - \frac{\partial }{\partial \xi}S(q(s,t),\xi(s,t)) \\
  \frac{\partial }{\partial s}\eta (s,t)  &=-\frac{\partial }{\partial t}\xi(s,t) \end{aligned} \right . \end{equation*}

   and as we saw, this yields the Morse complex of $S$. 
   
   Now let $H_\lambda, S_\lambda$ be such that $\Psi^1_\lambda(0_N)\cap \Phi_\lambda^{1}(0_N)$ does not depend on $\lambda$. Here $\Phi_\lambda^t$ is the flow of $H_\lambda$, and $\Psi_\lambda^t$ the flow of $S_\lambda$ (i.e. $\Psi_\lambda^t=\Psi_{S_\lambda}^t$, and  $\Psi_\lambda^1(0_N)= \Gamma_{S_\lambda}$). We set $L_\lambda=L_{S_\lambda}$ and then  $\Psi^1_\lambda(0_N)\cap \Phi_\lambda^{1}(0_N) \simeq \varphi_\lambda^1(0_N)\cap L_\lambda$, 
  so if we have 
   $L_\lambda=\varphi_\lambda^1(L)$, this is equal to $\varphi_\lambda^1(0_N\cap L)$. One can check that the action of the intersection points does not depend on $\lambda$ either, i.e. the action spectrum  of $K_\lambda$ does not depend on $\lambda$.
   
   As a result we have that the $FC^*(\Psi_\lambda^1(0_N), \Phi_\lambda^1(0_N);a,b)$ are equivalent chain complexes.  For $\lambda=0$, we get $FC^*(\Psi_\lambda^1(0_N), \Phi_\lambda^1(0_N);a,b)\simeq FC^*(0_N, \phi^1(0_N);a,b)\simeq FC^*(H;a,b)$ the last equivalence follows from lemma \ref{lem-7.2}, while for $\lambda =1$, we get $FC^*(\Psi_\lambda^1(0_N), \Phi_\lambda^1(0_N);a,b)\simeq FC^*(L_S,0_N;a,b)\simeq MC^*(S^b,S^a)$. 
   
   As a result, we have a chain homotopy equivalence
   $$ FC^*(K_0; a,b) \longrightarrow FC^*(K_1;a,b)$$
   and in particular the Floer cohomology $FH^*(K_\lambda ; a,b)$ does not depend on $\lambda$. 
   If  $L_0$ coincides with $0_N$, we can take $S_0= \frac{1}{2}\langle A\xi, \xi \rangle$ and $\varphi^1(L)=0_N$, while for $\lambda=1$, $L_1=L$ and then  $H_0=0$ and $S_1=S$, so  we get a chain homotopy equivalence
   $$ FC^*(H;a,b)\longrightarrow MC^*(S^b, S^a)$$
   But we proved that $FC^*(H;a,b)$ is chain equivalent to $FC^*(L,0_N;a,b)$. This proves the theorem for $L_1=0_N$. Now in the general case, if $\rho(0_N)=L_1$ we have $FC^*(L_0,L_1;a,b)=FC^*(L_0,\rho(0_N);a,b)=FC^*(\rho^{-1} (L_0),0_N; a,b)$ and if $S_1^\lambda$ is a G.F.Q.I. for $(\rho^{\lambda})\rho^{-1}(L_0)$ and $S_0^\lambda$ for $\rho^\lambda (0_N)$  (in particular we may choose $S_0^0$ to be a quadratic form), $S_1^\lambda\ominus S_0^\lambda$ is a GFQI for $(\rho^{\lambda})\rho^{-1}(L_0)\ominus \rho^\lambda (0_N) = \rho^\lambda (\rho^{-1}(L_0)-0_N)$, and it has the same critical points (corresponding to intersection points of $\rho^{-1}(L_0)\cap 0_N \simeq L_0\cap L_1$) independently of $\lambda$. Thus $$H^{*-i}((S^1_0\ominus S_1^1)^b,(S^1_0\ominus S^1_1)^a)=H^{*-i}((S^0_0\ominus S_0^1)^b,(S^0_0\ominus S^0_1)^a)$$
   Now $S_0^1, S_1^1$ are GFQI for $L_0,L_1$ and $S_0^1,S_1^1$ for $\rho^{-1}(L_0), 0_N$, and by uniqueness theorem (see \cite{Viterbo-STAGGF, Theret}, $S_1^1$ is equivalent to a quadratic form,  so up to shift in index 
   \begin{gather*} FH^*(L_0,L_1;a,b)=FH^*(\rho^{-1}(L_0), 0_N;a,b)=H^{*-i}((S_0^1)^b,(S_0^1)^a)=\\ H^{*-i}((S_0^1\ominus S_0^0)^b, (S_0^1\ominus S_0^0)^a) = H^{*-i}((S_1^1\ominus S_1^0)^b, (S_1^1\ominus S_1^0)^a) \end{gather*}

\end{proof} 
 We now define $FH^*(\nu^*U,L;a,b)$ for $U$ an open set with smooth boundary. First if $f$ is a smooth function such that $\Gamma_f\pitchfork L$ we have defined $FC^*(L,\Gamma_f;a,b)$. Note that  $FC^*(L,\Gamma_f;a,b)=FC^*(L-\Gamma_f;a,b)$. 
 
 \begin{defn} We denote by $FC^*(\nu^*U,L;a,b)$ the limit $\varinjlim FC^*(\Gamma_{f_k},L;a,b)$, where 
$f_k$ is a sequence such that (by abuse of notations) $f_k \longrightarrow -\infty\cdot (1-\chi_U)$, that is $f_k$ is a decreasing sequence of smooth functions such that $f_k$ converges pointwise to $-\infty$ on $N\setminus U$, and to $0$ in $U$ and we assume $\Gamma_{f_k} \pitchfork L$, so we get an inductive system
 $$ FC^*(L-\Gamma_{f_k};a,b)  \longrightarrow FC^*(L-\Gamma_{f_{k+1}};a,b)$$
 \end{defn} 
 
  It is not hard to see that the limit vanishes for large (positive or negative) degrees, since it is bounded by the values of the grading $m_L$. By commutation of direct limits and homology, we have  $$FH^*(\nu^*U,L;a,b) =\varinjlim FH^*(\Gamma_{f_k},L;a,b)$$
 
 \begin{rem}
 \begin{enumerate} 
 \item
 For $a=-\infty$ or $b=+\infty$  we define $$FC^*(\nu^*U,L;-\infty,b) = \lim_{a\to -\infty}FC^*(\nu^*U,L;a,b)$$ and $$FC^*(\nu^*U,L;a,+\infty)= \lim_{b\to +\infty)}FC^*(\nu^*U,L;a,b)$$ and of course $$FC^*(\nu^*U,L)=\lim_{\substack {a\to -\infty \\ b\to +\infty}}FC^*(\nu^*U,L;a,b)$$
 \item
 Note that if $L$ is "transverse" to $\nu^*U$, that is $L$ does not intersect $\partial U\times \{0\}$ and $L$ is transverse to the smooth manifold $\nu^*U\setminus \partial U\times \{0\}$, then for $k$ large enough, $L\cap \Gamma_{f_k} \simeq L\cap \nu^*U$, and the sequence $FC^*(L-\Gamma_{f_k};a,b) $ is stationary. 
 \end{enumerate} 
 \end{rem}
 
 Note that analogously, we have
 
 \begin{prop}\label{Prop-5.5}
 Let $(f_k)_{k\geq 1}$ be a sequence as above. Then denoting by $S^\lambda_U=S^\lambda \cap U$, we have $$\lim_kH^*((S-f_k)^b, (S-f_k)^a)=H^*(S_{U}^b, S_{U}^a)$$
 
 \end{prop}  
 
 \begin{proof} 
 Let $(S-f_k)^\lambda=\{(x,\xi) \mid  S(x,\xi)-f_k(x)\leq \lambda\}$, then $$((S-f_k)^b,(S-f_k)^a) \cap \{x=x_0\} = (\{\xi \mid S(x_0,\xi)< b+f_k(x_0)\}, \{\xi \mid S(x_0,\xi)< a+ f_k(x_0)\}) $$
 
 For $f_k(x_0)=0$, this equals $(S_{x_0}^b, S_{x_0}^a)$ while for $f_k(x_0) \geq A$ we get $(S_{x_0}^{b-C}, S_{x_0}^{a-C})$, but for $C$ large enough, there is no critical value of $S$ in $[a-C,b-C]$, hence the pair $(S_{x_0}^{b-C}, S_{x_0}^{a-C})$ is homotopically trivial. Now let $U_ \varepsilon $ be a neighborhood of $U$, and assume $ -C(1-\chi_U) \leq f_k \leq -C(1-\chi_{U_ {\varepsilon}})$. Then $(S+C(1-\chi_U))^\lambda \subset (S-f_k)^\lambda \subset (S+C(1-\chi_{U_ \varepsilon }))^\lambda$  and for $C$ large enough (remember that $a,b$ are fixed), we have $$H^*((S+C(1-\chi_U))^b, (S+C(1-\chi_U))^a) \simeq  H^*(S_{U}^b, S_{U}^a)$$ and since 
 $H^*(S_{U}^b, S_U^a) = \lim_{ \varepsilon \to 0} H^*(S_{U_ \varepsilon }^b, S_{U_ \varepsilon }^a)$ the proposition holds. 
 \end{proof} 
 
 From the Proposition \ref{Prop-5.5} and Theorem \ref{thm-5.1} we conclude
 
 \begin{thm}
 Let $U$ be an open set with smooth boundary. Let $L$ be a Lagrangian brane with g.f.q.i. $S$ with quadratic form at infinity of index $i$. Then there is an isomorphism
 $$FH^*(\nu^*U,L;a,b) \simeq H^{*-i}(S_U^b, S_U^a)$$
 \end{thm}  
 
 We now consider the ``restriction'' map induced by the inclusion $V \subset U$ where $V$ is another open set with smooth boundary. 
 $$FH^*(\nu^*U,L; a,b) \longrightarrow FH^*(\nu^*V,L; a,b)$$
 
 It is constructed as follows. 
 Let $f_k, g_k$ be sequences converging to $-\infty\cdot (1-\chi_U), -\infty\cdot (1-\chi_V)$. Since $V\subset U$ we may assume $f_k\geq g_k$, so there is a map
 $$FH^*(L, \Gamma_{f_k}; a,b) \longrightarrow FH^*(L,  \Gamma_{g_k}; a,b)$$ and in the limit we get  the above restriction map. 
 
 \begin{prop} \label{Prop-5.7}
 We have a commutative diagram
$$ \xymatrix {FH^*(\nu^*U,L; a,b) \ar[r]\ar[d]^{\simeq}&FH^*(\nu^*V,L; a,b) \ar[d]^{\simeq} \\
 H^*(S_{U}^b, S_U^a) \ar[r]& H^*(S_{V}^b, S_V^a)
 }
 $$
 where the horizontal maps are restriction maps. 
 \end{prop} 
 \begin{proof} 
 This follows from the diagram obtained for $f_k\geq g_k$ 
 $$ \xymatrix {FH^*(\Gamma_{f_k},L; a,b) \ar[r]\ar[d]^{\simeq}&FH^*(\Gamma_{g_k},L; a,b) \ar[d]^{\simeq} \\
 H^*((S+f_k)^b, (S+f_k)^a) \ar[r]& H^*((S+g_k)^b, (S+g_k)^b)
 }
 $$
  \begin{figure}[H]
 \begin{center}  \begin{overpic}{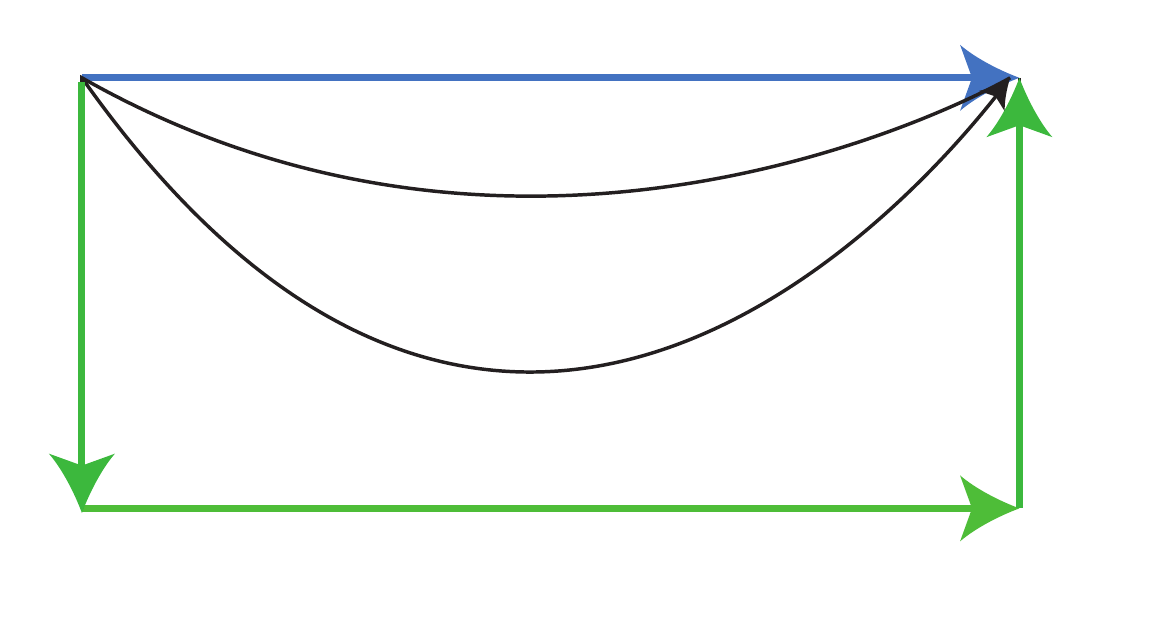}
 \put (5,50) {$L+\Gamma_f$}  \put (45,50) {${\color{blue}{T_u^0}}=L+\Gamma_{(1-t)f+tg}$} \put (85,50) {$L+\Gamma_g$}
  \put (5,5) {$S+f$}  \put (25,22) {$T_u^t$}  \put (30,5){${\color{green}{T_u^1}}$} \put (45,5){$S+{(1-t)f+tg}$} \put (90,5) {$S+g$}
\end{overpic}
\end{center}
\caption{The homotopy proving Proposition \ref{Prop-5.7}}\end{figure}
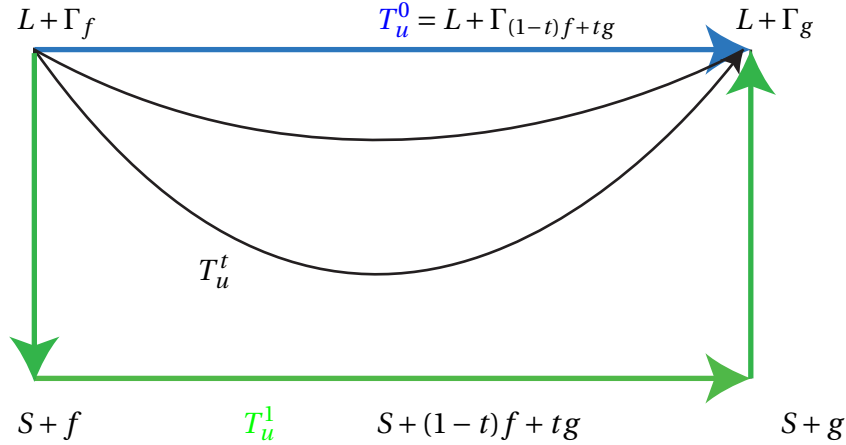

In the above figure, the vertical arrows are obtained by the deformation from Floer cohomology to GF-cohomology (i.e. the cohomology associated to a generating function) as in the proof of theorem \ref{thm-5.1}, i.e. the family $FC^*(K_\lambda; a,b)$ while the horizontal arrow correspond to a linear homotopy from $f$ to $g$.
More precisely if $K_\lambda(L,S)$, where $L=L_S$, is the deformation in the proof of Theorem \ref{thm-5.1}, we have 
for a path $\lambda(u), h(u)$ such that $\lambda(0)=\lambda(1)=0, h(0)=f, h(1)=g$, the family
$T^0_u=K_{\lambda(u)}(L+\Gamma_{h(u)}, S+h(u))$ which for a proper choice of the path (as suggested by the above figure) deforms $K_0(L+\Gamma_f, L+\Gamma_g)$ (here $h(u)=(1-u)f+ug)$) and the family $T^1_u$ composition of the path $K_{u}(L+\Gamma_f, S+f)$, the path $K_1(L+\Gamma_{h(u)}, S+h(u))$ and the opposite of the path $K_u(L+\Gamma_g, S+g)$. The first and last paths correspond to the chain isomorphisms 
$FC^*(L+\Gamma_f;a,b) \longrightarrow MC^*((S+f)^b, (S+f)^a)$ and $FC^*(L+\Gamma_g;a,b) \longrightarrow MC^*((S+g)^b, (S+g)^a)$ and 
the middle path, to the map $MC^*((S+f)^b, (S+f)^a) \longrightarrow MC^*((S+g)^b, (S+g)^a)$  obtained in  Proposition \ref{Prop-2.5}. 

This yields the commutativity of the diagram. 

 \end{proof} 
 
 Now we shall deal with Lagrangians which may not necessarily have a G.F.Q.I., however these always exists locally (and we mean locally in space {\bf and} action). To exploit this we need
 
 \begin{prop} 
 Assume $L_0,L_1$ are Lagrangian submanifolds coinciding in a neighborhood of $V_{x_0}=T_{x_0}^*N$ where they are assumed to be analytic. Then there is a neighborhood $U$ of $x_0$ and $ \varepsilon _0>0$ such that for $V\subset U$ and $\delta< \varepsilon $ we have
 $$FH^*(\nu^*V, L_0;c-\delta, c+\delta) = FH^*(\nu^*V,L_1;c-\delta, c+\delta)$$
 
 \end{prop} 
 \begin{proof} 
Indeed, assume $W$ is a convex neighborhood of $x_0$, so that we can find $J$ such that $T^*U$ is pseudo-convex whenever $U$ is convex in $W$ (this is the case for the standard complex structure on $T^* {\mathbb R}^n$). Then for $V\subset U$, a holomorphic strip with one boundary component in $\nu^*V$ is either  contained in $T^*W$, or exits from $T^*U$. If this was not the case, we would have a sequence of holomorphic discs $D_n$ with boundary in $\nu^*V_n$ such that $V_n \longrightarrow \{x_0\}$ and $D_n$ exits from $T^*U$.  But the sequence of such holomorphic discs will then converge to a holomorphic disc $D_\infty$ with boundary in $T^*_{x_0}N \cup L$ (since $\nu^*V_n \longrightarrow  T_{x_0}^*N$). Then since $D_n\cap \complement T^*W \neq \emptyset$, $D_\infty$ is not reduced to a point, and then $\lim_n\area(D_n)\geq \area (D_\infty)>0$. As a result, since given  $x_0\in N$ and $L$ analytic near $T_{x_0}^*N$,  we have only finitely many holomorphic discs\footnote{It is easy to find, in the non-analytic case, in real dimension $2$, a curve having infinitely many intersection points with a line, and such that the area between consecutive intersection points goes to $0$.} with boundary in  
$T_{x_0}^*N\cup L$, therefore  there exists a number $A(x_0,L)>0$ such that all such discs have area greater than $A(x_0,L)$. 

The conclusion is that either $D_n \subset T^*U$ or $\area (D_n)\geq A/2$ for $n$ large enough. So if $\delta<A(x_0,U)/4$ we have that the holomorphic strips representing $FC^*(\nu^*V,L, c-\delta,c+\delta)$ are contained in $T^*U$ and hence only depend on $L\cap T^*U$. This concludes our proof. 
\end{proof}

We shall need that $L\cap T^*U$ has a bounded number of connected components for a sequence of open sets shrinking to a point. This is for example the case when $L$ is analytic. 

\begin{prop} 
Given a real analytic  $L$ and $(x_0,t_0) \in N\times {\mathbb R} $, and $L_j$ the connected components of $L\cap T^*U$ such that $f_L(z_j)$ is near $t_0$ for $z_j\in L_j\cap T_{x_0}^*N$. Then for $U$ and $\delta$ small enough, we have
$$FH^*(\nu^*U,L;t_0-\delta, t_0+\delta) =\bigoplus_j FH^*(\nu^*U, L_j;t_0-\delta, t_0+\delta) $$
\end{prop} 

\begin{proof} 
Indeed, as in the above proof, there is no connecting trajectory between $L_j$ and $L_k$ for $j\neq k$ with action less than $ \varepsilon $, since $d(L_j,L_k)$ is bounded from below by a constant independent from $U$, and according to the previous proposition, $FH^*(\nu^*U,L_j; c-\delta, c+\delta )$ is well defined. Note that components $L_q$ of $L\cap T^*U$ such that $f_L(z_q)$ is far from $t_0$ do not contribute to the direct sum.  
\end{proof} 

\begin{prop} \label{Prop-5.10}
Under the above assumptions, there is a family $S_j$ of G.F.Q.I. such that for $U, \varepsilon $ small enough, we have

$$FH^*(\nu^*U,L;t_0-\delta, t_0+\delta) =\bigoplus_j FH^*(\nu^*U,L_j; t_0-\delta, t_0+\delta)=\bigoplus_j H^*({S_j}_{\mid U}^{t_0+ \delta},{S_j}_{\mid U}^{t_0- \delta}) $$
\end{prop} 
\begin{proof} 
Each $L_j$ can be assumed to be contained in $\varphi_j(0_N)$, where $\varphi_j$ is the time one of a time dependent Hamiltonian isotopy. Then $L_j$ is the connected component of some $L_{S_j} \cap T^*U = \varphi_j(0_N)\cap T^*U$. By a small perturbation of $L_{S_j}$ we may assume that the critical values of $\xi \mapsto S_j(x,\xi)$ that do not correspond to points in $L_j$ are not in $[c_0- \delta, c_0+\delta]$.  Then 

$$FH^*(\nu^*U,L_j;c-\delta, c+\delta)=FH^*(L_{S_j},\nu^*U;c-\delta, c+\delta)= H^*({S_j}_{\mid U}^{c+\delta}, {S_j}_{\mid U}^{c-\delta})$$
\end{proof} 

\section{The complex of sheaves $\F_L$ and its singular support.}\label{SS-FL}

Let us consider the presheaf of chain complexes $f\mapsto {\widehat {FC}}_L^*(f)$ on $\Cat{F(N)}$, defined in section \ref{Rectif}, filtered by the  ${\widehat {FC}}_L^*(f,\lambda)$. We can also consider this as a presheaf of chain complexes over  $\Cat{F(N)} \times ( {\mathbb R} , \leq)$.
This yields a presheaf of complexes on $N\times ( {\mathbb R}, \leq)$ defined as follows: if $f_k$ "converges" to $-\infty \cdot (1-\chi_U)$ we have $${\widehat {\cF_L}}(U\times ]-\infty, \lambda[)=\lim_k {\widehat {FC}}_L^*(f_k,\lambda)$$
The goal of this section is to prove

\begin{thm} \label{thm-8.1}
The sheafification of ${\widehat {\cF_L}}$  defines a complex of sheaves on $N\times {\mathbb R} $ , denoted $\cF_L$, such that $$SS(\cF_L)=\widehat L$$

\end{thm}

\begin{rem} \label{Rem-7.2}
\begin{enumerate} 
\item If $ ({\mathbb R} , \leq)$ is the set $ {\mathbb R} $ endowed with the topology for which open sets are the $]-\infty , \lambda[$, and $ {\mathbb R} $ denotes the real line with the usual topology, we have a continuous map $i: {\mathbb R} \longrightarrow ( {\mathbb R}, \leq )$. Now if $\F$ is  a presheaf on  $( {\mathbb R} , \leq)$ then $i^{-1}\F$ is a presheaf on $ {\mathbb R} $. Note that by definition, $i^{-1}( \F) (]a,b[)= \F(]-\infty, b[)$. Note also that for such a sheaf, we have $SS (i^{-1}\F)\subset {\mathbb R} \times {\mathbb R} _+$, since the map $(i^{-1}\F)(]a,b[) \longrightarrow (i^{-1}\F)(]a',b[)$ for $a>a'$ is an isomorphism. 
Finally $(i^{-1}\F)([a,b[)$ is defined as the kernel of $ \F(]-\infty, b[) \longrightarrow  \F(]-\infty, a[)$. 
\item \label{Rem-7.2-2} It will be useful to remember that since the maps $FC^*(L, \Gamma_f ; \lambda) \longrightarrow FC^*(L, \Gamma_g ; \lambda) $ are induced by continuation for $f\leq g$, the same holds for the maps  $\widehat {FC}^*(L, \Gamma_f ; \lambda) \longrightarrow \widehat {FC}^*(L, \Gamma_g; \lambda)$, hence for $\widehat {FC}^*(\nu^*U,L ; \lambda) \longrightarrow \widehat {FC}^*(\nu^*V,L; \lambda)$ for $V \subset U$, and finally for the restriction map $\cF (U) \longrightarrow \cF(V)$. 
\end{enumerate} 
\end{rem} 
\begin{proof} 
First note that $\widehat{\cF_L}$ is indeed a presheaf. since if $W \subset V \subset U$ we have functions $f_k\leq g_k \leq h_k$ with $\lim_k f_k =- \infty\cdot (1-\chi_U), \lim_k g_k = -\infty\cdot (1-\chi_V), \lim_k h_k = -\infty\cdot (1-\chi_W)$ and thus maps
$$\xymatrix{{\widehat {FC}}_L^*(f_k,\lambda) \ar[r]\ar@/_0.9pc/[rr]& {\widehat {FC}}_L^*(g_k,\lambda)  \ar[r]&{\widehat {FC}}_L^*(h_k,\lambda)\\
}$$

which form a commutative diagram. Taking the limit as $k$ goes to $\infty$ we get a commutative diagram 
$$\xymatrix{{\widehat {\F}}_L^*(U\times ]-\infty, \lambda[) \ar[r]\ar@/_1.2pc/[rr]& {\widehat \F}_L^*(V\times ]-\infty, \lambda[))  \ar[r]&{\widehat {\F}}_L^*(W\times ]-\infty, \lambda[)\\
}$$

It is easy to see that we also have a commutative diagram as we restrict $\lambda$, that is for $\lambda < \mu < \nu$ we have the commutative diagram

$$ \xymatrix{{\widehat {\F}}_L^*(U\times ]-\infty, \nu[) \ar[r]\ar@/_0.9pc/[rr]& {\widehat \F}_L^*(U\times ]-\infty, \mu[))  \ar[r]&{\widehat {\F}}_L^*(U\times ]-\infty, \lambda[)\\
}$$

hence for $W\times ]-\infty , \lambda [ \subset V \times [-\infty, \mu[ \subset U \times [-\infty, \nu[$, we get a diagram 

$$\xymatrix{{\widehat {\F}}_L^*(U\times ]-\infty, \nu[) \ar[r]\ar@/_0.9pc/[rr]& {\widehat \F}_L^*(V\times ]-\infty, \mu[))  \ar[r]&{\widehat {\F}}_L^*(W\times ]-\infty, \lambda[)\\
}$$
Since the sets $U\times ]-\infty, \lambda[$ form a fundamental basis of open sets stable by intersection, this implies that $\widehat {\F}_L$ is a presheaf of complexes.

To finish the proof of Theorem \ref{thm-8.1} we need to prove the  the statement about the singular support. This will take up the the rest of this section.

Our strategy is to use the previous section and the fact that the singular support is local (in $x$, i.e. $SS(\cF) \cap T_{x_0}^*N$ only depends on  $\cF$ near $x_0$) to replace $\F_L$ by $\F_S$ where $L=L_S$ near $T_{x_0}^*N$. 

Indeed let $S$ be a G.F.Q.I. such that $L_S=L$ near $T_{x_0}^*N$, chosen as in the proof of Proposition \ref{Prop-5.10}. Let $\cF_S$ be the element of $D^b(N\times {\mathbb R} )$ defined as $(R\pi)_*(k_{U_S})$ where $U_S=\{(x,\xi,\lambda) \mid S(x,\xi)< \lambda \}$ and $\pi$ is the map $\pi(x,\xi,t)=(x,t)$. Note that $\cF_S (W)$ is quasi-isomorphic to $\pi_*( \Omega^*_{U_S}(W))=\Omega^*_{U_S}(\pi^{-1}(W))$ where $\Omega^*_U$ is the de Rham complex on $U$, yielding a flabby resolution of $k_U$, and more specifically, $\cF_S(V\times [a,b[)\simeq(R\pi)_*(k_{U_S})(V\times [a,b[)\simeq \Omega^*(S_{ V}^b,S_{ V}^a)$. 

We have a morphism of presheaves $$ \widehat {\cF_{L _S}} \longrightarrow  \widehat {\cF_S }$$ obtained by defining $\widehat {\cF_S }$ in the same way as $\widehat {\cF_L }$  but with $FC^*(\nu^*U,L; \lambda)$ replaced by $MC^*(S_U^\lambda)$ where  $MC^*(S_U^\lambda)$ is the Morse complex associated to $S_U$ and generated by the critical points inside $U$, and the ``entering'' critical points of $S_{\partial U}$ (i.e. those such that $\nabla S$ points inward), which is quasi-isomorphic to $\pi_*(\Omega^*(S_{ U}^\lambda))$ (see \cite{Laudenbach}). This, according to the following proposition,  induces a morphism of the associated sheaves, 
$${\cF_{L_S}} \longrightarrow  {\cF_S }$$
\begin{prop} 
The chain homotopy equivalence defined in the proof of Proposition \ref{Prop-5.7}
$$FC^*(\nu^*U,L_S; a,b ) \longrightarrow MC^*(S_{ U}^b, S_{ U}^a)$$
induces a sheaf morphism
$$\cF_{L_S} \longrightarrow \cF_S$$ that is a quasi-isomorphism on $U \times [a,b[$. 

\end{prop}
\begin{proof} 
Self-evident, since $FC^*(\nu^*U,L_S; a,b )$ is chain homotopy equivalent to  $\widehat {\cF_{L }}$ and   $MC^*(S_{ U}^b, S_{ U}^a)$ is chain homotopy equivalent to $\widehat {\cF_{S }}$, and 
the quasi-isomorphism $$FC^*(\nu^*U,L_S; a,b ) \longrightarrow MC^*(S_{ U}^b, S_{ U}^a)$$  induces a quasi-isomorphism $ \widehat {\cF_{L _S}} \longrightarrow  \widehat {\cF_S }$, hence a quasi-isomorphism between the corresponding sheaves. 
\end{proof} 
From this we deduce the 

\begin{thm} \label{thm-6.5}
Let $S$ be a  G.F.Q.I. for $L$. Then the map $$\cF_L \longrightarrow \cF_S$$ induces a quasi-isomorphism on all open sets, i.e.
$$H^*(W,\cF_L) \longrightarrow H^*(W,\cF_S)$$ is an isomorphism for all open sets $W$. 
\end{thm}  
 \begin{proof} 
 Since the induced morphisms over stalks $$(\cF_L)_{(x,t)} \longrightarrow (\cF_S)_{(x,t)}$$
 that is the map $$H^*((\cF_L)_{(x,t)}) \longrightarrow H^*((\cF_S)_{(x,t)})$$ is an isomorphism, and since $H^*(W,\cF_L)$ and $H^*(W,\cF_S)$ are limits of spectral sequences with $E_2$ terms given by 
$H^p(W, {\mathcal H}^q((\cF_L)_{(x,t)}))$ and $H^p(W, {\mathcal H}^q((\cF_S)_{(x,t)}))$ and the morphism induces an isomorphism
of these $E_2$ terms. The theorem then follows from the spectral sequence comparison theorem. 

 \end{proof} 
 
 Now remember that by definition of $\cF_S$, we have $H^*(W, \cF_S)=H^*(W, (R\pi)_*(k_{U_S}))=H^*((\pi)^{-1}(W), k_{U_S})=H^*(W_S)$ where $W_S=\{(x,\xi,\lambda) \mid (x,t)\in W, S(x,\xi)<t\}$. 
 Using the five-lemma and exact sequence of a triple, we get 
 
 \begin{lem} 
 We have $$H^*(W,A,\cF_{L_S})=H^*(W_S,A_S)$$
 
 \end{lem} 
 
 We may now conclude that for $L_S$, we have $SS(\cF_{L_S})=\widehat {L_S}$.
 Indeed, for a sheaf $\cF$ the set $SS(\cF)$ only depends on $H^*(W,A,\cF)$. But we just proved that $H^*(W,A,\cF_{L_S})=H^*(W_S,A_S)=H^*(W,A, (R\pi)_*(k_{U_S}))$, and we already know that  $SS((R\pi)_*(k_{U_S}))={\widehat {L_S}}$ (see \cite{Viterbo-ISTST}, Proposition 9.12 
 ). 
 \end{proof} 
 
 We can now finish the proof of Theorem \ref{thm-8.1}. 
 
 Let $x_0\in N$, and $p_1,...., p_k$ the intersection points of $L$ and $T_x^*N$, with action $f_L(x_0,p_j)=a_j$, and $L_j$ the connected component of $L\cap T_U^*N$ through $(x_0,p_j)$. We denote by $S_j$  a G.F.Q.I. such that $L_j=L_{S_j}$ near $(x_0,p_j)$, and with no other intersection point with $T_{x_0}^*N$ having action near $a_j$ (as we saw in the proof of Prop \ref{Prop-5.10}, this can be achieved by perturbation). Then according to Proposition \ref{Prop-5.10}, we have $H^*(W,A,\cF_L)=H^*(W,A,\cF_S)$ for $A\subset W$, where $W$ is  a neighborhood of $(x_0,t_0)$. By locality of the singular support, this implies that $SS(\cF_L)=\bigcup_j SS(\cF_{S_j})=\bigcup_j{\widehat L}_j$ near $(x_0,t_0)$. But since $SS(\cF_S)=\bigcup_j{\widehat L}_j$ near $(x_0,t_0)$ and this coincides with $\widehat L$ near $(x_0,t_0)$ we conclude our proof.

Note that we also easily prove
\begin{thm} \label{Thm-6.8}
For $U$ and $ \varepsilon $ small enough, we have an isomorphism $$FH^*(\nu^*U,L, a- \varepsilon , a+ \varepsilon ) \longrightarrow H^*(U\times [a - \varepsilon , a+ \varepsilon [, \cF_L)$$

\end{thm} 
\begin{proof} 
Let $S$ be a G.F.Q.I. generating $L_S$ such that $L_S=L$ near   $T_{x_0}^*U$ and such that the critical values  of $S(x_0,\bullet)$ which do not correspond to points in $L\cap T_{x_0}^*N$ do not belong to $[a- \varepsilon , a+ \varepsilon ]$.  Then $$FH^*(\nu^*U,L, a- \varepsilon , a+ \varepsilon ) \simeq H^*(S_{ U}^{a+ \varepsilon }, S_{ U}^{a- \varepsilon }) = H^*(U\times [a- \varepsilon , a+ \varepsilon [, \cF_S) \simeq H^*(U\times [a- \varepsilon , a+ \varepsilon [, \cF_L)$$
The first isomorphism is Proposition \ref{Prop-5.10}, the second is by definition of $\cF_S$ and  the third is by Theorem \ref{thm-6.5}
\end{proof} 
 \section{Identifying Floer theory and sheaf theory}\label{section-8}
 
 We start from the following theorem, that we already proved for small $U, \varepsilon $  (see Theorem \ref{Thm-6.8}). We denote by $\alpha_{\lambda,\mu}$ the map $$\alpha_{\lambda,\mu}: FH^*(\nu^*U,L;\lambda,\mu) \longrightarrow H^*(U\times [\lambda,\mu[, \cF_L)$$ induced by $FC^*(\nu^*U,L;a,b) \longrightarrow \widehat {\cF_L}(U\times [\lambda,\mu[) \longrightarrow {\cF_L}(U\times [\lambda,\mu[) $, 
and by $\alpha_\lambda=\alpha_{-\infty,\lambda}$, so $$\alpha_\lambda : FH^*(\nu^*U,L;\lambda) \longrightarrow H^*(U\times [-\infty, \lambda [, \cF_L)$$

\begin{thm} For all $\lambda < \mu$
we have a canonical isomorphism  $$\alpha_{\lambda,\mu}: FH^*(\nu^*U,L;\lambda,\mu) \longrightarrow H^*(U\times [\lambda,\mu[, \cF_L)$$ and in particular we have an isomorphism
$$\alpha_{\lambda,\mu}: FH^*(L,0_N; \lambda,\mu) \longrightarrow H^*(N\times [\lambda,\mu[,\cF_L) $$
\end{thm} 

\begin{proof} 
 
We first deal with the case $(\lambda, \mu)=(a- \varepsilon , a+ \varepsilon )$ with $ \varepsilon >0$, small. We may assume $U$ is a domain with smooth boundary $\partial U$. Let us pick a function $\varphi$ such that 
\begin{enumerate} 
\item $U=\{x \mid \varphi (x) <0 \}$, $\partial U=\varphi^{-1}(0)$ and $0$ is a regular value of $\varphi$. 
\item $\inf \{ \varphi(x) \mid x \in U\}= \inf \{ \varphi(x) \mid x \in {\mathbb R}^n\}=-1$
\item \label{8-3} for all $(x,p)\in L$ such that $f_L(x,p)=a$ and $\varphi(x)<0$ we have for all $\lambda >0$ that $p\neq \lambda d\varphi(x)$

\end{enumerate} 
Note that property (\ref{8-3}) holds generically for $\varphi$, since $L_a=\{(x,p)\in L\mid f_L(x,p)=a\}$ has generic dimension $n-1$, and $\{(x,\lambda d\varphi(x))\mid \lambda >0, \exists p, (x,p)\in L_a\}$ has generic dimension $n$, so the intersection is generically empty. Note that $a$ is given and then we are free to choose $\varphi$, and that we cannot say anything on the boundary (i.e. $\varphi(x)=0$), since $\varphi^{-1}(0)=\partial U_0$ is fixed. 

According to the sheaf theoretic Morse lemma (see \cite{Viterbo-ISTST}, Corollary 9.22 
), for $s<0$, provided for $x\in U_{0}\setminus U_{s}$ and $t\in [a- \varepsilon , a+ \varepsilon ]$ we have $(x,t,0,1)\notin SS(\cF_L)=\widehat L - \nu^*(U_{0}\setminus U_s)$  then  $$H^*((U_{0},U_{s})\times (]-\infty,a+ \varepsilon [, ]-\infty, a+ \varepsilon [; \cF_L)=0$$ hence applying the exact sequence of a pair, 
$$H^*(U_{0}\times [a- \varepsilon , a + \varepsilon [;\cF_L) \simeq H^*(U_{s}\times [a- \varepsilon , a + \varepsilon [;\cF_L)$$
 
Now let us analyse the condition $(x,t,0,1)\notin SS(\cF_L)=\widehat L - \nu^*(U_{0}\setminus U_s)$. Indeed, this means, denoting $n_s(x)$ the exterior normal to $\partial U_s$ at $x\in \partial U_s$
\begin{enumerate}[label=(\alph*)]
\item \label{aa} either $x_0\notin \partial U_0\cup \partial U_s$ and $(x_0,0)\notin L $ and $f_L(x,0)\in [a- \varepsilon , a+ \varepsilon ]$
\item \label{bb} or $x_0\in \partial U_0$, $p_0=\lambda n_0(x_0)$ for $\lambda >0$ and $f_L(x_0,p_0)\in [a- \varepsilon , a+ \varepsilon ]$
\item \label{cc} or $x_0\in \partial U_s$, $p=\lambda n_s(x_0)$ for $\lambda >0$ and $f_L(x_0,p_0)\in [a- \varepsilon , a+ \varepsilon ]$
\end{enumerate}

If there is no such point, then  we just proved that $H^*(U_0\times [ a- \varepsilon , a+ \varepsilon[; \cF_L )= H^*(U_{-1-\delta}\times [ a- \varepsilon , a+ \varepsilon[,, \cF_L)=H^*(\emptyset, \cF_L)=0$ and this is equal to $FH^*(\nu^*U,L, a- \varepsilon , a+ \varepsilon )$, since the Floer complex has no generator. 

Otherwise, we may assume by genericity assumption, choosing $ \varepsilon$ small enough, that we have a unique such point, as we can restrict ourselves to the cases $s=-1+\delta$, so we are either in case \ref{aa}, \ref{bb} or \ref{cc}. 
\begin{itemize}
\item In case \ref{aa}, since we may choose $\varphi$ such that $\varphi(x_0)=-1$, as in as in Figure \ref{Fig-2}, we are reduced to considering $H^*(U_{-1+ \delta } \times [a- \varepsilon , a+ \varepsilon [; \cF_L)$, so we are reduced to the the case where $U$ and $ \varepsilon $ are small, and we know that 
\begin{gather*}  FH^*(\nu^*U,L, a- \varepsilon , a+ \varepsilon )\simeq FH^*(\nu^*U,L_{-1+\delta};a- \varepsilon , a+ \varepsilon )\simeq \\  H^*(U_{-1+\delta}\times [a- \varepsilon , a+ \varepsilon [, \cF_L)\simeq H^*(U_{0}\times [a- \varepsilon , a+ \varepsilon [, \cF_L)
\end{gather*} 
\item In case \ref{bb} we choose our family $U_s$ as in Figure \ref{Fig-3}, and we are again reduced to the case $U_{-1+\delta}$ is small. 
\item In  case \ref{cc} for $s$ close to $-1$, we see that  $\nu^*(U_s)$ converges to $T_{x_0}^*N$ where $\varphi(x_0)=-1$. But we can choose $x_0$ so that $(x_0,p_0)\in L$ implies $f_L(x_0,p_0)\notin [a- \varepsilon , a+ \varepsilon ]$ so this case may be avoided.  
\end{itemize}

 \begin{figure}[H]
 \begin{center}  \begin{overpic}[width=6cm]
 {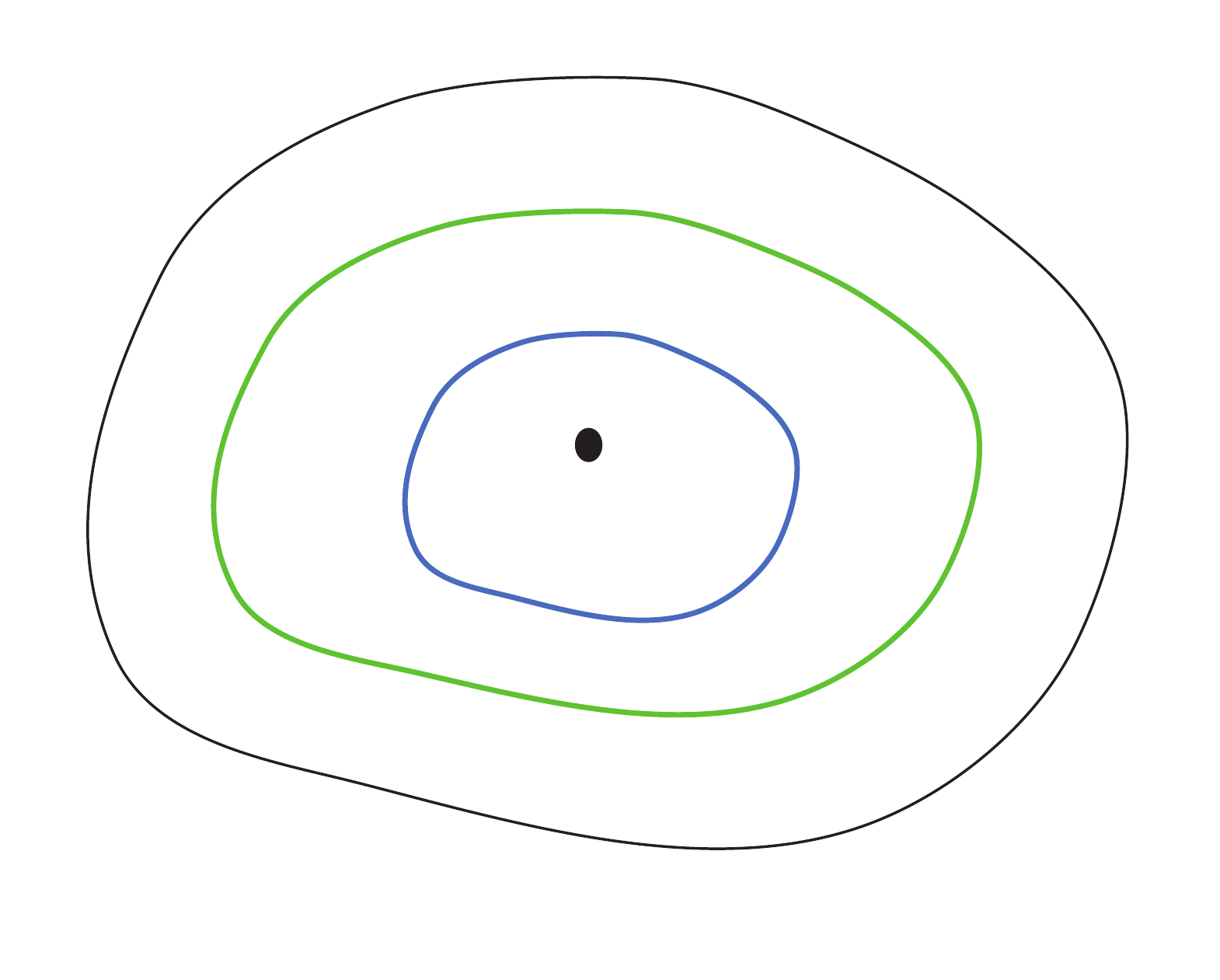}
 \put (45,38){$x_0$}
\end{overpic}
\end{center}
\caption{The family $U_s$ when $x_0$ is in the interior of $U_0$}\label{U_s}\label{Fig-2}\end{figure}
  \begin{figure}[H]
 \begin{center}  \begin{overpic}[width=6cm]
 {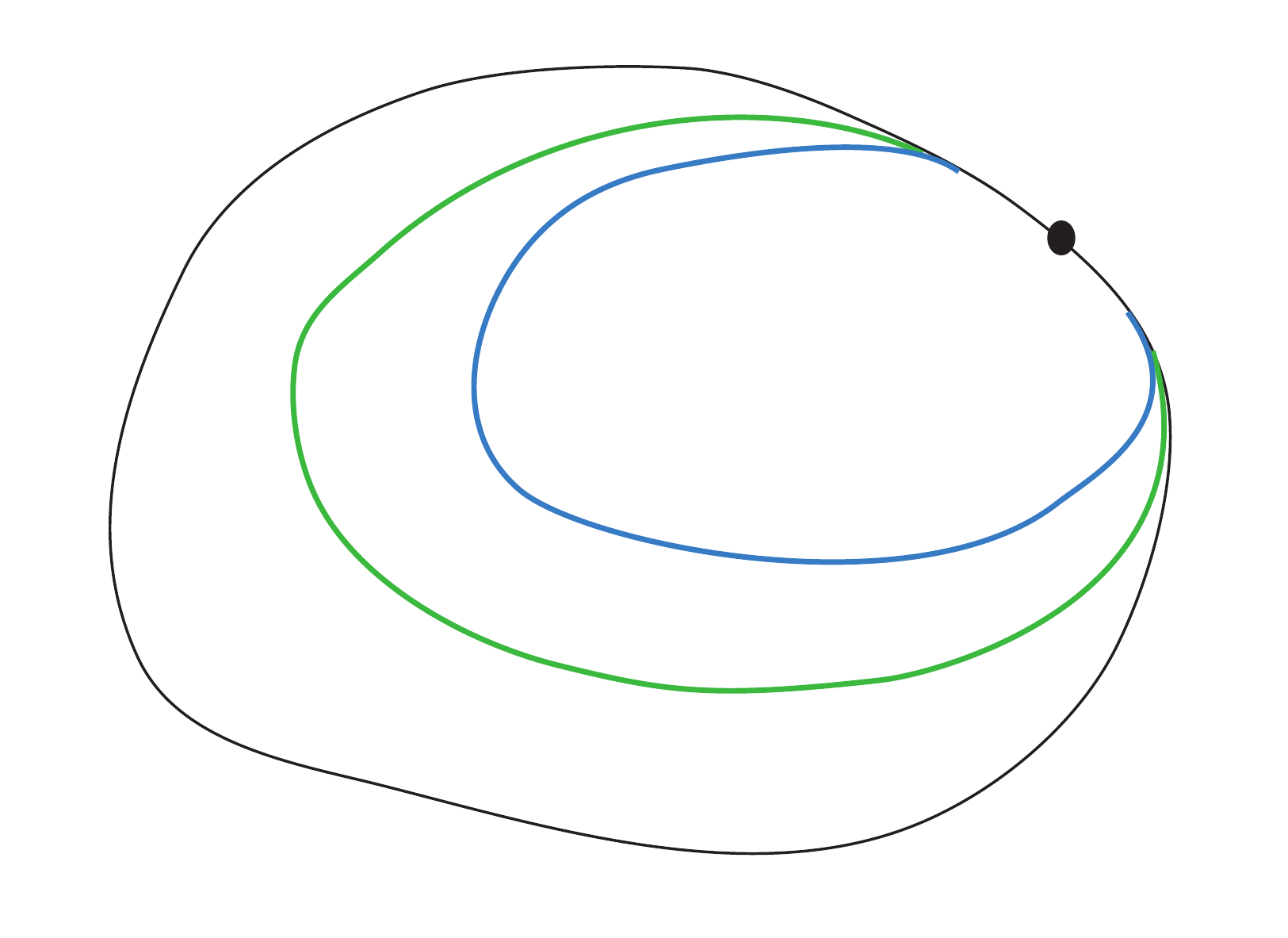}
 \put (90,55){$x_0$}
 \end{overpic}
\end{center}
\caption{The family $U_s$ when $x_0$ is on the boundary of $U_0$}\label{U_sb}\label{Fig-3}\end{figure}

So we reduced our result to the case where $U$ and $ \varepsilon $ are small, that is Theorem \ref{Thm-6.8}. 
In other words,  we just proved that for arbitrary $U$, and for  $\mu-\lambda $ small enough, we have  an isomorphism $$\alpha_{\lambda,\mu}: H^*(U\times [\lambda,\mu[, \cF_L) \longrightarrow FH^*(\nu^*U,L;\lambda , \mu)$$
We may now conclude by the five lemma. Assume we have for $\lambda <c$, that  $\alpha_\lambda$ is an isomorphism $$\alpha_\lambda : H^*(U\times ]-\infty, \lambda [, \cF_L) \longrightarrow FH^*(\nu^*U,L;\lambda)$$  then we have for $\lambda <c<\mu$ the folowing commutative diagram where the lines are the natural exact sequences

 \begin{equation}\tag{MV}\resizebox{.85\hsize}{!}{
\xymatrix {\ar[r]&FH^*(\nu^*U,L;-\infty,\lambda) \ar[r]^-{r_{\mu,\lambda}^*}\ar[d]^-{\alpha_\lambda}&FH^*(\nu^*U,L;-\infty,\mu) \ar^-{p_{\mu,\lambda}^*}[r]\ar[d]^-{\alpha_\mu}&  FH^*(\nu^*U,L;\lambda,\mu) \ar[d]^-{\alpha_{\lambda,\mu}}\ar[r]& \\
\ar[r]&H^*(U\times ]-\infty,\lambda[, \F_L^\bullet) \ar[r]^-{\rho_{\mu,\lambda}^*}&H^*(U\times ]-\infty,\mu[,\F_L^\bullet) \ar^-{\pi_{\mu,\lambda}^*}[r]&H^*(U\times [\lambda,\mu[, \F_L^\bullet)\ar[r]&
}}\end{equation} \label{MV}

Commutativity just follows from the fact that on $( {\mathbb R} , \leq)$ (i.e. in the $\lambda$ variable) everything is a sheaf from the beginning. In other words  $\lambda \mapsto FC^*(\nu^*U,L; -\infty,\lambda)$ 
induces the maps on both lines, which are just sheaf cohomology exact sequences for the pair $(]-\infty, \mu[, -\infty, \lambda[)$.

Now for $\mu-\lambda$ small enough, we just proved that $\alpha_{\lambda,\mu}$ is an isomorphism, and since by assumption $\alpha_\lambda$ is also an isomorphism, the same holds for $\alpha_\mu$.
Now for $\lambda=-\infty$ all the cohomology groups are zero, and we may therefore prove our theorem  by induction, crossing the action values of points in $L\cap \nu^*U$ one by one. Once we proved the isomorphism 
$$\alpha_\lambda : H^*(U\times ]-\infty, \lambda [, \cF_L) \longrightarrow FH^*(\nu^*U,L;\lambda)$$  holds for all $\lambda$, applying again the five lemma and the exact sequence of a triple yields the result for 
$$\alpha_{\lambda,\mu}: H^*(U\times [\lambda , \mu [, \cF_L) \longrightarrow FH^*(\nu^*U,L;\lambda,\mu)$$ 
\end{proof} 

\begin{cor} \label{Cor-9.2}
Let $Z$ be a closed submanifold in $N$. Then for generic $L$ we have $$H^*(Z\times [a,b[,\cF_L)=FH^*(L,\nu^*Z;a,b)$$
\end{cor} 
\begin{proof} 
This is obtained by a limiting argument. Let $U_\delta$ be a tubular neighborhood of $Z$ of size $\delta$. Then by definition,
$H^*(Z\times [a,b[,\cF_L)= \lim_{\delta\to 0} H^*(U_\delta\times [a,b[, \cF_L)$. Now for a generic $L$, we have $L\cap \nu^*Z =\lim_{\delta\to 0}L\cap \nu^*U_\delta$. As a result  both  generators and Floer trajectories for the pair $(\nu^*U,L_\delta)$ converge\footnote{For the Floer trajectories, this is because they have uniformly bounded area.} to those for $(L,\nu^*Z)$, and $FH^*(L,\nu^*Z;a,b)=\lim_{\delta \to 0} FH^*(\nu^*U,L_\delta;a,b)$. 
\end{proof} 

In the non-generic case, we shall define $H^*(L,\nu^*Z;a,b)$ as $H^*(Z\times [a,b[, \cF_L)$, and this is the same as $\lim_{\delta \to 0} FH^*(\nu^*U,L_\delta;a,b)$. We may thus assume that the above equality holds for all $L$. 

\section{Quantization of $(-L)$ and sheaf products}\label{Quant-L}

Let  $L$ be  a Lagrangian brane in $T^*N$ and $-L$ the image of $L$ by the antisymplectic map  $(q,p) \longrightarrow (q,-p)$. Note that $f_{-L}(q,-p)=-f_L(q,p)$ and $$\widehat{(-L)}=\{(q,t,-\tau p, \tau) \mid t=-f_L(q,p), (q,p)\in L, \tau\geq 0\}$$  is different  from $-\widehat L$ (since $\widehat {(-L)} \subset \{\tau \geq 0\}$ while  $-\widehat L\subset \{\tau \leq 0\}$). However if $\gamma(q,t,p,\tau)=(q,-t,p, -\tau)$, we have $\gamma (-\widehat{L})=\widehat{(-L)}$. 
Note also that $\gamma$ is the cotangent map associated to  $c(q,t)=(q,-t)$. Finally to avoid confusion, we denote by $V^\#$ the dual of the vector space $V$, and we assume that the dual of a graded complex has the opposite grading, i.e. $(V^p)^\#$ has grading $-p$. This way the dual of a cochain complex is again a cochain complex\footnote{ Thus $((V^\bullet)^\#)^p =(V^{-p})^\#$. Be careful since $V^\bullet$ only means that we are dealing with a complex, while $*$, is a dumb symbol for the grading: there is no $V^{-\bullet}$ but we can write $(V^\bullet [m])^*=(V^\bullet)^{*-m}$. For this reason we shall use the notation $FC^\bullet$ for the Floer complex. } (i.e. its differential has still degree $+1$). 

We have\footnote{Remember that we choose conventions so that the Maslov index of an intersection point of $\Gamma_f$ and $0_N$ coincides with the Morse index of the critical point of $f$. This convention has also the advantage that the pant product sends classes of degree $p$ and $q$ to a class of degree $p+q$ (and not $p+q-n$ as is the case for other conventions). On the other hand, changing $f$ to $-f$ changes the index from $p$ to $n-p$. } 
\begin{gather*} FC^*(-\nu^*U,L;\lambda, \mu)= [ FC^{n-*}(L, -\nu^*U; -\mu, -\lambda)]^{\#}=\\ [FC^{n-*}(L, \nu^*(N\setminus U); -\mu, -\lambda)]^\#=[FC^{\bullet}(L, \nu^*(N\setminus U); -\mu, -\lambda)^\#]^{*-n}= \\ 
FC^\bullet (L, \nu^*(N\setminus U; -\mu,-\lambda)^{\#}[-n]\end{gather*} 

Now the Verdier dual of $\cF_L$ is obtained as follows. 

By definition, for any complex  of sheaves $\cF$ and any open set $U$, we have an exact sequence
$$0 \longrightarrow \Gamma_c(U,\cF) \longrightarrow \Gamma (N,\cF) \longrightarrow  \Gamma (N\setminus U, \cF)$$ and this is right exact when $\cF$ is a complex of $c$-soft sheaves that we denote by  $\cS$, i.e. we have 
$$0 \longrightarrow \Gamma_c(U,\cS) \longrightarrow \Gamma (N,\cS) \longrightarrow  \Gamma (N\setminus U, \cS) \longrightarrow 0$$
 
Note that going to the dual (we are dealing with vector spaces) we get an exact sequence of vector spaces\footnote{Watch out, these are not sheaves not even presheaves !}
$$0 \longleftarrow \Gamma_c(U,\cF)^\# \longleftarrow \Gamma (N,\cF)^\# \longleftarrow  \Gamma (N\setminus U, \cF)^\#$$
and we can consider $U \mapsto \Gamma_c(U,\cF)^\#$ and $U\mapsto \Gamma (N\setminus U, \cF)^\#$ as presheaves, while $ \Gamma (N,\cF)^*$ is a constant presheaf. Replacing $\cF$ by a soft resolution $\cS$, we get that
 $U \mapsto \Gamma_c(U,\cS)^\#$ and $\widetilde{\mathcal S}^\bullet: U\mapsto \Gamma (N\setminus U, \cS)^\#$ become complexes of sheaves, the first one representing the Verdier dual ${\mathbb D}\cF$ of $\cF$ in the derived category,  and we
get the exact sequence of sheaves
$$0 \longleftarrow {\mathbb D}\F \longleftarrow K_\cS \longleftarrow  \widetilde \cS  \longleftarrow  0$$

Now we apply the above procedure to $\cF_L$, but only with respect to the $N$ factor, doing nothing on the $ {\mathbb R} $ factor. 
In other words $\widetilde \cS$ is quasi-isomorphic to the total complex associated to $0 \longrightarrow K^\bullet_{\mathcal S} \longrightarrow {\mathbb D}\F \longrightarrow 0$, where $ K^\bullet_{\mathcal S}$ is a constant sheaf. 
Since $c^*(\widetilde \cS )$ is quasi-isomorphic to the sheafification of the rectification of the quasi-presheaf  $$U \mapsto  FC^{\bullet}(L, \nu^*(N\setminus U); -\mu, -\lambda)^\#=[FC^{\bullet}(-L, \nu^*(U); \lambda, \mu)]^\#[-n]$$ and we can identify it to $\cF_{-L}[-n]$. 
 Thus $SS(c^*(\widetilde \cS))=\widehat {(-L)}$ and since 
$SS(k_N)=0_N$, we have $SS(c^*{\mathbb D}\F_L)={\widehat {(-L)}}$. Comparing behavior at infinity, we also see that $K_\cS \simeq k_{N\times \mathbb R}$.  As a conclusion, we get (noticing that $c^*(k_{N\times {\mathbb R}} )=k_{N\times {\mathbb R}} )$

\begin{prop} The sheaf complex  $\cF_{-L}$ is quasi-isomorphic to $$0 \longrightarrow k_{N\times {\mathbb R}}[-n] \longrightarrow c^*({\mathbb D} \cF )[-n] \longrightarrow 0$$
\end{prop}

We now give an alternative construction of $\cF_{-L}$. We mention that since $L$ is orientable, the orientation sheaf $\omega_N$ is equal to $k_N[n]$ and 
$${\mathbb D} \cF= R\hHom (\cF, \omega_{N\times {\mathbb R} })=R\hHom (\cF[-n],k_{N\times {\mathbb R} })[n]
$$
\begin{lem} 
Let $\cF \in D^b(N\times {\mathbb R} )$ be such that $SS(\cF)=L$ and satisfies (\ref{1}), (\ref{2}) of the main theorem. Then we have that $R\hHom( \cF, k_{N\times {\mathbb R} })$ is pure and simple (i.e. satisfies   (\ref{1}) of the main theorem),   $R\hHom( \cF, k_{N\times {\mathbb R} }) \equiv k_N$ near $+\infty$, $R\hHom( \cF, k_{N\times {\mathbb R} })\equiv 0$ near $-\infty$,  and $SS(R\hHom( \cF, k_{N\times {\mathbb R} }))=-SS({\cF})$. 
\end{lem} 
\begin{proof} See \cite{K-S}, proposition 5.4.14, page 236 (or \cite{Viterbo-ISTST}, Proposition 9.43 
, page 162 
)
. Note that we use the fact that $\cF$ is constructible. 
\end{proof}

Note that as a result $SS(R\hHom( \cF, k_{N\times {\mathbb R} })\subset \{\tau \leq 0\}$, but we want a sheaf such that $SS( \cG) \subset \{\tau \geq 0\}$, $\cG=k_N$ near $+\infty$, $\cG=0$ near $-\infty$ and $SS(\cG)=\gamma (- SS( \cF))$. 

We have for $\cF=\cF_L$ that $R\Hom(k_{N\times {\mathbb R} }, \cF)= k$ since according to Proposition 8.49 
of \cite{Viterbo-ISTST}, we have  $$R\Hom(k_{N\times {\mathbb R} },\cF)=H^0(N\times {\mathbb R}, R\hHom(k_{N\times {\mathbb R} },\cF))$$ and the singular support of $R\hHom(k_{N\times {\mathbb R} },\cF))$ is contained in $\tau \geq 0$, so $$H^0(N\times {\mathbb R}, R\hHom(k_{N\times {\mathbb R} },\cF))= H^0(N\times \{+\infty\}, R\hHom(k_{N\times {\mathbb R}},\cF))= R\Hom (k_N, k_N)=k$$

Now take a generator in $R\Hom(k_{N\times {\mathbb R} },\cF)$, it induces a morphism in the derived category $$   k_{N\times {\mathbb R}} \overset{u} \longrightarrow  \cF$$
This complex yields an element $\overline \cF$ in $D^b(N\times {\mathbb R} )$ equal to $0$ near $t=+\infty$ (where $\cF \equiv 0$) and to $k_N$ near $-\infty$ (where $\cF \equiv k_N$) obtained by completing the triangle (remember this is always possible in a unique way, up to a non-unique isomorphism)
$$   \overline \cF \overset{v} \longrightarrow k_{N\times {\mathbb R}} \overset{u} \longrightarrow  \cF$$
It is easy to see that this does not depend on the choice of $u$ (provided it is nonzero !). Moreover we have $SS(\overline \cF)=SS(\cF)$. 

\begin{prop} 
Let $\check{\cF_L}=c^{*}(R\hHom(\overline{\cF_L}, k_{N\times {\mathbb R} }))$. Then $\cF_{-L}=\check{\cF_L}$ in particular  $$SS(\check{\cF_L})=\widehat{(-L)}$$
and $\check {\cF_L}=k_N$ near $+\infty$ and $0$ near $-\infty$. 
\end{prop} 
\begin{proof} Derived functors preserve distinguished triangles, so that applying $R\hHom ( \bullet , k_{N\times {\mathbb R}}[n])$ to the triangle $$   \overline \cF \overset{v} \longrightarrow k_{N\times {\mathbb R}} \overset{u} \longrightarrow  \cF$$

yields a triangle

$$  R\hHom(\overline \cF , k_{N\times {\mathbb R}}[n])  \overset{v^*} \longrightarrow R\hHom(k_{N\times {\mathbb R}}, k_{N\times {\mathbb R}}[n]) \overset{u^*} \longrightarrow  R\hHom(\cF, k_{N\times {\mathbb R}}[n])$$

and since $R\hHom(k_{N\times {\mathbb R}}, k_{N\times {\mathbb R}}[n])\simeq k_{N\times{\mathbb R} }$ and ${\mathbb D}\cF_L=R\hHom(\cF, k_{N\times {\mathbb R}}[n] )$ we get that $\check{\cF_L}$ is quasi-isomorphic to the total complex of
$$0 \longrightarrow k_{N\times {\mathbb R} }[-n]  \overset{u^*}\longrightarrow c^*({\mathbb D}\cF_L)[-n]$$  and this is quasi-isomorphic to $\cF_{-L}$. 
\end{proof} 

Remember all ours sheaves are constructble. In particular, which is not surprising (since $-(-L)=L$), we have ${\mathbb D}^2\cF=\cF$ and
$R\hHom(\cF,\cG)=R\hHom({\mathbb D}\cG, {\mathbb D}\cF)$. 

 \begin{examples} 
\begin{enumerate} 
\item  For $L=\Gamma_f$ where $f$ is a smooth function, we have $\check{\cF_f}=\cF_{-f}$. 
\item For $L=\nu^*U$, we have $\check {k}_{U\times [0,+\infty[}=k_{(N\setminus  U)\times [0,+\infty[}$
\end{enumerate}  \end{examples} 
\subsection{Sheaf products}
Let now $\cF_1,\cF_2$ be two sheaves in $D^b(N\times {\mathbb R} )$ and $s: {\mathbb R} \times {\mathbb R} \longrightarrow {\mathbb R} $ the map $s(u,v)=u+v$. By abuse of notation, we write also $s$ for the map $N\times {\mathbb R} \times N \times {\mathbb R} \longrightarrow N\times N\times {\mathbb R} $ given by $s(x,u,y,v)=(x,y,u+v)$ and set $d_N(x,u,v)=(x,u,x,v)$. 

\begin{defn} 
Let  $\cF_i$ ($i=1,2$) be two sheaves in $D^b(N_i\times {\mathbb R} )$ . We set $$\cF_1\ast\cF_2=(Rs)_!(\cF_1\boxtimes \cF_2)$$ in $D^b(N_1\times N_2\times {\mathbb R} )$ and for $N_1=N_2$,  $$\cF_1\cstar\cF_2=d_N^{-1}(\cF_1* \cF_2
) \in D^b(N\times {\mathbb R} )$$
Finally we set\footnote{we denote by $R\hHom^{\cstar}$ what is denoted $Hom^*$ in \cite{Vichery-th}.}
$R\hHom^{\cstar} (\cF_1,\cF_2)$ to be the adjoint functor of $\cstar$, that is for all  $\cF_1, \cF_2, \cF_3$
$$\Mor_{D^b(N\times {\mathbb R})}(\cF_1, R\hHom^{\cstar} (\cF_2,\cF_3) ) = \Mor_{D^b(N\times {\mathbb R})}(\cF_1{\cstar} \cF_2,\cF_3)$$
\end{defn} 

\begin{exo} 
\begin{enumerate} 
\item \label{exo-1} For functions $f,g \in C^\infty(N)$, we have $\cF_f\cstar \cF_g\simeq \cF_{f+g}$.
\item We have $k_{[\lambda, +\infty [} \cstar  k_{[\mu, +\infty [} \simeq k_{[\lambda+\mu, +\infty [}$
\item We also have  $\check{k}_{U\times [0,+\infty[} \cstar  k_{U\times [0,+\infty[}= k_{N\times [0,+\infty[}$. This follows from (\ref{exo-1}).
\item $\cstar$ is associative, commutative, and has $k_{N\times [0, +\infty [}$ as neutral element. 
 \end{enumerate} 
 \end{exo} 
Note that $d_N^{-1}$ and $(Rs)_!$ commute, so we have $$\cF_1\cstar \cF_2= (Rs)_!(\cF_1\otimes_N \cF_2)$$  where $\otimes_N$ means the tensor product is only over $N$ (so $\cF_1\otimes_N \cF_2 \in D^b(N\times {\mathbb R} \times {\mathbb R} )$).
Assume now the $\cF_i$ ($i=1,2,3$) satisfy the assumptions of our theorem, that is 
\begin{enumerate} \label{page-38}
\item\label{8-i} $SS(\cF_i)=\widehat {L_i}$
\item \label{8-ii}   $\cF_i$ is pure and simple
\item \label{8-iii} $\cF_i$ coincides with $k_N$ near $+\infty$ and with $0$ near $-\infty$. 
\end{enumerate} 

\begin{lem} 
We have $\cF_{L_1}* \cF_{L_2}$ is quasi-isomorphic to $\cF_{L_1\times L_2}$ hence $$SS(\cF_{L_1}* \cF_{L_2})=\widehat{L_1\times L_2}$$

\end{lem} 

\begin{proof}  It will be convenient to denote by $FC^*(L_1,L_2)$ the sheaf on $( {\mathbb R}, \leq)$ given by $$ ]-\infty, \lambda[ \mapsto FC^*(L_1,L_2;-\infty, \lambda)$$
 Then $FC^*(L_1\times L_2, L_3\times L_4)$ coincides with $s_!(FC^*(L_1,L_3)\otimes FC^*(L_2,L_4) )$ since an intersection point of 
$(L_1\times L_2) \cap (L_3\times L_4)$ corresponding to a pair of intersection points of $L_1\cap L_3$ and $L_2\cap L_4$, and the action is the sum of the actions (moreover $\langle \partial (x_1,x_3),(x_2,x_4)\rangle=\langle \partial x_1,x_3\rangle \langle \partial x_2, x_4\rangle$). Now the quasi-presheaves $(U_1,U_2) \mapsto FC^*(L_1\times L_2, \nu^*U_1\times \nu^*U_2)$  and $(U_1,U_2) \mapsto s_!(FC^*(L_1,L_3)\otimes FC^*(L_2,L_4) )$ do then coincide, so does their rectification and then sheafification. But $s!$ coincides with $(Rs)_!$ because these are flabby sheaves on $( {\mathbb R}, \leq)$ (see Appendix). Finally rectification and sheafification on $(U_1,U_2) \mapsto FC^*(L_1\times L_2, \nu^*U_1\times \nu^*U_2)$ yields $\cF_{L_1\times L_2}$, while the rectification of $(U_1,U_2) \mapsto s_!(FC^*(L_1,L_3)\otimes FC^*(L_2,L_4) )$ yields $(Rs)_!(\cF_{L_1}\boxtimes \cF_{L_2})=\cF_1* \cF_2$. 
\end{proof}

Note that because $SS(\cF_i)$ is Lagrangian, the inclusion (\cite{Viterbo-ISTST}, Proposition 9.33
, page 155
$$SS(\cF_1\boxtimes \cF_2) \subset SS(\cF_1)\times SS(\cF_2)=\widehat {L_1}\times \widehat {L_2}$$ is an equality, since the right hand side is a Lagrangian submanifold and the left-hand side is coisotropic (see \cite{K-S}, Th. 6.5.4 and \cite{Gabber}). Now $$SS((Rs)_!(\cF_1\boxtimes \cF_2)) \subset (\Lambda_s)^{\#}(SS((\cF_1\boxtimes \cF_2)))$$ the right hand side  is the reduction of $\widehat L_1\times \widehat L_2$ by $\tau_1=\tau_2$, that is $\widehat{L_1\times L_2}$. 
Indeed
$$\widehat{L_1\times (-L_2)} =  (\widehat{L_1}\times \widehat{(-L_2)})\cap \{\tau_1=\tau_2\}/\simeq$$
where $\cap \{\tau_1=\tau_2\}/\simeq$ is the symplectic reduction by $\tau_1=\tau_2$. This is easily checked, since
\begin{gather*} \widehat{(L_1)}\times \widehat{(-L_2)}=\\ \{(q_1,-f_{L_1}(q_1,-p_1), \tau_1p_1, \tau_1, q_2,f_{L_2}(q_2,p_2), \tau_2p_2, \tau_2) \mid (q_1,-p_1)\in L_1, (q_2,p_2)\in L_2, \tau_1, \tau_2\geq 0\} \end{gather*} 
and its reduction is 
$$ \{(q_1,q_2, f_{L_2}(q_2,p_2)-f_{L_1}(q_1,-p_1), \tau p_1, \tau p_2, \tau) \mid (q_1,-p_1)\in L_1, (q_2,p_2)\in L_2, \tau_1, \tau_2\geq 0\} $$
But $(Rs)_!$ corresponds on the singular support to reduction by $\tau_1=\tau_2$, so 
$$SS((Rs)_!({\cF_1} \boxtimes \check{\cF_2}))= \widehat{L_1\times (-L_2)}$$
 Again we have equality, because $SS((Rs)_!(\cF_1\boxtimes \check{\cF_2}))$ is coisotropic and $\widehat{L_1\times L_2}$ is Lagrangian.

\begin{prop} \label{Prop-9.8}
If the $\cF_i$ satisfy the above three conditions, then so does $ {\cF_1}\cstar\check{\cF_2}$ and we have $R\hHom^{\cstar} (\cF_2,\cF_1)= \cF_1\cstar \check{\cF_2}$. Moreover $\cF_{L_1\times (-L_2)}=\check {\cF_{L_2}}*\cF_{L_1}$ and
$SS(\cF_1* \check{\cF_2})=\widehat{L_1\times (-L_2)}$. 
\end{prop} 
\begin{proof}

Indeed, $SS(\check \cF_2)= \widehat{-L_2}$, and according to the previous lemma,  $SS(  \check{\cF_2}* \cF_1)=\widehat{L_1\times (-L_2)}$. 
Moreover  $(Rs)_!(\check{\cF_2} \boxtimes\cF_1))$ is simple.

To  prove that $(Rs)_!(\check{\cF_2}\boxtimes \cF_1))$  coincides with $0_N$ near $-\infty$ and with $k_N$ near $+\infty$, it is enough to prove that 
$(\check{\cF_2}\boxtimes \cF_1)$ coincides with $0_{N\times N}$ near $\{ (t_1,t_2) \mid t_1+t_2 <-A\}$ and  with $k_{N\times N}$
near $\{ (t_1,t_2) \mid t_1+t_2  < A\}$  for $A$ large enough.  
Since $SS (\check{\cF_2}\boxtimes \cF_1) \subset \{ (\tau_1,\tau_2) \mid \tau_1 >0, \tau_2 >0\}$  we have for $T$ large enough an isomorphism

$$H^*(N\times \{ (t_1,t_2) \mid t_1+t_2  >T\}, \check{\cF_2}\boxtimes \cF_1) \longrightarrow H^*(N\times \{ (t_1,t_2) \mid t_1>T, t_2  >T\},\check{\cF_2}\boxtimes \cF_1)$$

hence 
$$
R\Gamma(N\times s^{-1}\{t >T\}, \check{\cF_2}\boxtimes \cF_1) \longrightarrow R\Gamma(N\times \{ (t_1,t_2) \mid t_1>T, t_2  >T\},\check{\cF_2}\boxtimes \cF_1)
$$
 is an isomorphism and $\check{\cF_2}\boxtimes \cF_1)\simeq k_{N\times N}$ near $+\infty$. 

On the other hand if $t_1+t_2<-T$ either $t_1<-T/2$ or $t_2< -T/2$ hence either $\check{\cF_2}$ or $ \cF_1$ is isomorphic to $0_N$ and $\check{\cF_2}\boxtimes \cF_1$ is isomorphic to $0_{N\times N}$. 
That $R\hHom^{\cstar} (\cF_2,\cF_1)= \cF_1\cstar \check{\cF_2}$ either follows from uniqueness  or can be proved as 3.4.4 in \cite{K-S}, p. 159 since the sheaves are (cohomologically) constructible.
\end{proof}

\begin{prop} 
We have $$H^*(N\times ]-\infty, \lambda[, \cF_{L_1}\cstar \check{\cF_{L_2}})=FH^*(L_1,L_2;\lambda)$$
\end{prop} 
\begin{proof} 
Applying Corollary \ref{Cor-9.2}, with $Z=\Delta_N$, and since $\check{\cF_2}\cstar\cF_1$ is the quantization of $L_1\times (-L_2)$, we have   \begin{gather*} H^*(N\times ]-\infty, \lambda[, \check{\cF_{L_2}}\cstar\cF_{L_1})=H^*(N\times ]-\infty, \lambda[, d_N^{-1}(\check{\cF_{L_2}}*\cF_{L_1})=\\  FH^*(L_1\times (-L_2), \nu^*\Delta_N; \lambda)= FH^*(L_1,L_2;\lambda) \end{gather*} 

\end{proof} 
\begin{prop}\label{cup-product}If $\cF, \cG$ have bounded support near $-\infty$ (this is always the case for $\cF_L$), 
there is a product map  $$\cup_{\cstar} :  H^*(N\times [\lambda, +\infty [, \cF) \otimes H^*(N\times  [\mu, +\infty [, \cG) \longrightarrow H^*(N\times  [\lambda + \mu , +\infty [, \cF\cstar \cG)$$ 
\end{prop} 
\begin{proof} 
Indeed, $H^*(N\times  [\lambda, +\infty [, \cF)=H^*(N\times {\mathbb R} , \cF\cstar k_{N\times [\lambda, +\infty [})$ and since $\cstar$ is associative and commutative, we have 
$$(\cF\cstar k_{N\times [\lambda, +\infty [})\cstar ({\cG\cstar k_{N\times [\mu +\infty [}))=\cF\cstar \cG \cstar k_{N\times [\lambda+ \mu, +\infty [}}$$
and we only have to define a map 
$$H^*(N\times {\mathbb R} , \cF) \otimes H^*(N\times {\mathbb R}, \cG) \longrightarrow H^*(N\times {\mathbb R}  , \cF\cstar \cG)$$
Since $\cF\cstar \cG = (Rs)_!(d_N^{-1}(\cF \boxtimes \cG))$ we have 
\begin{gather*} H^*(N\times {\mathbb R} [, \cF\cstar \cG)= H^*(N\times {\mathbb R} , d_N^{-1}(Rs)_!(\cF \boxtimes \cG))=\\
H^*(\Delta_N \times {\mathbb R}, (Rs)_!(\cF\boxtimes_N \cG))=H_s^*(\Delta_N \times {\mathbb R}^2,\cF\boxtimes_N \cG)
\end{gather*} 
In the last equality we should be careful about the support (on the real component, since $N$ is compact anyway), and $H^*_s$ means that we should take sections with a support on which  $s$ is proper. But since we will look at sheaves vanishing at $-\infty$, their support is of the type $[a,+\infty[$, and  $s$ is always proper on the product of two such sets. 
Now for such sheaves, we have obviously a map $$H^*(N\times {\mathbb R} , \cF) \otimes H^*(N\times {\mathbb R} , \cG) \longrightarrow H_s^*(N\times {\mathbb R} \times N \times {\mathbb R} , \cF \boxtimes \cG) \longrightarrow  H^*(\Delta_N \times {\mathbb R}^2,\cF\boxtimes_N \cG)$$
the last map being restriction. 
\end{proof} 
\label{altproof-9}Alternate proof: Remember that $H^*(N\times [\lambda , +\infty [, \cF)$ is a relative cohomology of $\cF$ (that of $\cF$ for the pair 
$( {\mathbb R} , ]-\infty, \lambda [)$, and note that $$( {\mathbb R} , ]-\infty, \lambda [)\times ( {\mathbb R} , ]-\infty, \mu [) = ( {\mathbb R}^2 , ]-\infty, \lambda [ \times {\mathbb R} \cup {\mathbb R} \times ]-\infty, \mu [))$$
So writing $P_{\lambda, \mu}=\{(u,v) \mid u<\lambda,\; v <\mu\}$ and  
$P_{\nu}=\{(u,v) \in {\mathbb R}^2 \mid u+v < \nu\}$ and $\overline P_{\lambda,\mu}, \overline P_\nu$ for the pairs $( {\mathbb R} ^2, P_{\lambda,\mu})$, $( {\mathbb R} ^2, P_\nu)$,  we have an inclusion $$P_{\lambda+ \mu}
\subset P_{\lambda,\mu}$$ so that if $\cF$ is in $D^b( {\mathbb R}^2)$, we have a map
$$i(\overline P_{\lambda,\mu},\overline P_{\lambda +\mu}): H^*( \overline P_{\lambda,\mu}; \cF) \longrightarrow H^*({ \overline P_{\lambda+\mu}}; \cF)$$
Thus we have an isomorphism 
$$H^*(N\times [\lambda, + \infty[, \cF)\otimes H^*(N\times [\mu, + \infty[, \cG)\longrightarrow H^*(N^2\times \overline P_{\lambda,\mu}; \cF\boxtimes \cG)
$$

and a map induced by the restriction as described above
$$
\sigma^*:  H^*(N^2\times\overline P_{ \lambda,\mu}; \cF\boxtimes \cG) \longrightarrow H^*( \overline P_{\lambda + \mu }; \cF\boxtimes \cG)$$
 and an isomorphism 
 $$H^*(N^2 \times \overline P_{\lambda + \mu }; \cF\boxtimes \cG) \simeq H^*(N^2\times [\lambda + \mu , +\infty[; (Rs)_!(\cF\boxtimes \cG)) =H^*( N^2 \times [\lambda + \mu , +\infty[;\cF\ast\cG)$$
 Finally we compose this with the restriction to the diagonal
 $$\rho_{\Delta_N}: H^*( N^2 \times [\lambda + \mu , +\infty[;\cF\ast \cG) \longrightarrow H^*( N \times [\lambda + \mu , +\infty[;\cF\cstar\cG)$$
and we get the map
$$H^*(N\times [\lambda, + \infty[, \cF)\otimes H^*(N\times [\mu, + \infty[, \cG)\longrightarrow H^*( N \times [\lambda + \mu , +\infty[;\cF\cstar\cG)$$
We shall prove in the next section that this map induces the Floer cohomology pant product.

We have now

\begin{prop} \label{Prop-9.11}
\begin{enumerate} 
\item\label{9-11-i} Denoting by $\pi: N\times {\mathbb R} \longrightarrow {\mathbb R} $  the projection, we have the isomorphism   $$(R\pi)_* (\check \cF_L\cstar \cF_L)=H^*(N)\otimes k_{[0,+\infty[}$$
\item \label{9-11-ii} There are morphisms $k_{N\times [0,+\infty [} \overset{u}\longrightarrow \check {\cF_L}\cstar \cF_L \overset{v}\longrightarrow k_{N\times [0,+\infty [} $ such that the composition is an isomorphism. 
\end{enumerate} 
\end{prop} 
\begin{proof} (\ref{9-11-i})
Indeed, we have $$H^*([\lambda, \mu[, (R\pi)_! (\check \cF_L\cstar \cF_L))=H^*(N\times [\lambda , \mu[, \check \cF_L\cstar \cF_L)
$$ and this last group is $FH^*(L,L;\lambda,\mu)$ so this is 
\begin{enumerate} 
\item [(a)] $H^*(N)$ if $0 \in [\lambda, \mu[$
\item [(b)] $0$ otherwise
\end{enumerate} 

A sheaf $\cH$ satisfying this properties must have $SS(\cH)=T_{\{0\}}^*N$ and is pure, and moreover, is quasi-isomorphic to  $H^*(N)$ at $+\infty$ and $0$ at $-\infty$ (since this is the case for $\check \cF_L\cstar \cF$). The only such sheaf is $H^*(N)\otimes k_{[0,+\infty[}$. 

(\ref{9-11-ii}) We have, using the fact that $\check k_{N\times [0,+\infty[}=k_{N\times [0,+\infty[}$, 
\begin{gather*} \Mor ( k_{N\times [0,+\infty[} , \check \cF_L\cstar \cF_L)=\Mor( k_{N\times [0,+\infty [} , R\hHom^\cstar (\cF_L, \cF_L))=\\ \Mor(k_{N\times [0,+\infty [} \cstar \cF_L, \cF_L)=\Mor(\cF_L,\cF_L)
\end{gather*} 
So  the identity map in this last set yields a "natural morphism" from $k_{N\times [0,+\infty [} $ to $\check \cF_L\cstar \cF$. Conversely, 
$$\Mor (  \check \cF_L\cstar \cF_L , k_{N\times [0,+\infty[})=\Mor( \cF_L,  R\hHom^\cstar (\check{\cF_L}, k_{N\times [0,+\infty [}))=\Mor(\cF_L, \cF_L)
$$ and again we select the identity map. We let it to the reader to check that composition yields the identity. 
\end{proof} 

In the sequel we shall write $FH^*(L_1,L_2;t^-,t^+)$ for  the limit as $ \varepsilon $ goes to $0$ of $FH^*(L_1,L_2, t- \varepsilon , t+ \varepsilon )$. 
We finally define a general product map

 \begin{gather*}
H^*(N\times [\lambda , +\infty [, R\hHom^\cstar(\cF_{L_1},\cF_{L_2}) \otimes H^*(N\times [\mu , +\infty [, R\hHom^\cstar(\cF_{L_2},\cF_{L_3}))\\ \Big\downarrow \\ H^*(N\times [\lambda + \mu , +\infty [, R\hHom^\cstar(\cF_{L_1},\cF_{L_3})
\end{gather*}

that is 

 \begin{gather*}
H^*(N\times [\lambda , +\infty [, \check{\cF_{L_1}}\cstar\cF_{L_2}) \otimes H^*(N\times [\mu , +\infty [, \check{\cF_{L_2}}\cstar\cF_{L_3})\\ \Big\downarrow \\ H^*(N\times [\lambda + \mu , +\infty [, \check{\cF_{L_1}}\cstar\cF_{L_3})
\end{gather*}

or else 

 \begin{gather*}
H^*(N\times \overline P_{\lambda,\mu}, (\check{\cF_{L_1}}\cstar\cF_{L_2}) \boxtimes  (\check{\cF_{L_2}}\cstar\cF_{L_3})))\\ \Big\downarrow \\ H^*(N\times [\lambda + \mu , +\infty [, \check{\cF_{L_1}}\cstar\cF_{L_3})
\end{gather*}
 by composing the following maps
 
 $$
\begin{gathered}
H^*(N\times [\lambda , +\infty [, \check{\cF_{L_1}}\cstar\cF_{L_2}) \otimes H^*(N\times [\mu , +\infty [, \check{\cF_{L_2}}\cstar\cF_{L_3})\\ \Big\downarrow \simeq \\
H^*(N\times  N \times  \overline P_{\lambda,\mu}, (\check{\cF_{L_1}}\cstar\cF_{L_2}) \boxtimes  (\check{\cF_{L_2}}\cstar\cF_{L_3})) \\ \qquad \qquad \qquad \Big\downarrow i(\overline P_{\lambda,\mu};\overline P_{\lambda+ \mu})  \\ H^*(N\times N \times \overline P_{\lambda+ \mu}, (\check{\cF_{L_1}}\cstar\cF_{L_2}) \boxtimes  (\check{\cF_{L_2}}\cstar\cF_{L_3}))) \\ \quad \Big\downarrow  d_N^{-1}
\\
H^*(N \times \overline P_{\lambda+ \mu}, (\check{\cF_{L_1}}\cstar\cF_{L_2}) \otimes_N  (\check{\cF_{L_2}}\cstar\cF_{L_3}))) \\ \simeq  \Big\downarrow   (Rs)_! \\
H^*(N\times [\lambda + \mu , +\infty [, \check{\cF_{L_1}}\cstar\cF_{L_3} \cstar \check{\cF_{L_2}} \cstar \cF_{L_2}) \\ \qquad \quad  \Big\downarrow \id \cstar \id \cstar v \\
H^*(N\times [\lambda + \mu , +\infty [, \check{\cF_{L_1}}\cstar\cF_{L_3} ) 
\end{gathered}
$$
Note that $(Rs)_! d_N^{-1}$ can also be identified to the map
$$
\begin{gathered} 
H^*(N\times  N \times  \overline P_{\lambda,\mu}, (\check{\cF_{L_1}}\cstar\cF_{L_2}) \boxtimes  (\check{\cF_{L_2}}\cstar\cF_{L_3}))\\ \Big \downarrow \\
H^*(N\times N \times [\lambda + \mu , +\infty [, ((\check{\cF_{L_1}}\cstar\cF_{L_3}) \ast (\check{\cF_{L_2}} \cstar \cF_{L_2}))\cstar k_{\Delta_N\times [0,+\infty[})\\
 \Big \downarrow \\
H^*(N\times N \times [\lambda + \mu , +\infty [, ((\check{\cF_{L_1}}\cstar\cF_{L_3}) \ast k_{N\times [0,+\infty[})\cstar k_{\Delta_N\times [0,+\infty[})
\end{gathered} $$
\begin{defn} 
The map defined above
$$ \begin{gathered} H^*(N\times [\lambda , +\infty [, \check{\cF_{L_1}}\cstar\cF_{L_2}) \otimes H^*(N\times [\mu , +\infty [, \check{\cF_{L_2}}\cstar\cF_{L_3})\\
 \Big\downarrow \\
 H^*(N\times [\lambda + \mu , +\infty [, \check{\cF_{L_1}}\cstar\cF_{L_3} ) 
\end{gathered} $$
is denoted $\cup_{\cstar}$. 
\end{defn} 

Finally we have the following

\begin{defn} There is a category $\Cat{DF(N)}$ such that 
\begin{enumerate} 
\item It is generated by the constructible objects  of $\Cat{D^b(N\times {\mathbb R} )}$ such that (\ref{8-ii}), (\ref{8-iii}) hold 
\item The morphisms between $\cF$ and $\cG$ are $R\hHom^{\cstar} (\cF,\cG)\simeq \check \cF \cstar \cG$  (Proposition \ref{Prop-9.8})
\item The composition is given by a morphism
$$R\hHom^\cstar(\cF,\cG)\cstar R\hHom^\cstar(\cG,\cH) \longrightarrow R\hHom^\cstar(\cF,\cH)$$ since the left hand side can be identified to 
$$(\check \cF\cstar \cG)\cstar (\check \cG\cstar \cH)$$  the above map is induced by the map $$\cG \cstar \check \cG \longrightarrow k_{N\times [0,+\infty[}$$  defined in Proposition \ref{Prop-9.11} and we then use the fact that $\cF \cstar k_{N\times [0,+\infty[}=\cF$
\item It has a monoidal structure derived from $(\cF, \cG) \longrightarrow \cF\cstar \cG$ with identity $k_{N\times [0,+\infty[}$
\end{enumerate} 
\end{defn} 

In the following section we define the analogue of the pant-product, that will be a functor
$$R\hHom^{\cstar} (\cF,\cG) \cstar R\hHom^{\cstar} (\cG,\cH) \longrightarrow R\hHom^{\cstar} (\cF,\cH)$$
thus making the above an enriched category.

\section{Symplectic reduction and pant product identification}

 Our goal here is to identify a number of operations on the symplectic topology side with operations on sheaves. Besides an interpretation of symplectic reduction as restriction of sheaves, our main goal is to identify  the pant product in Floer cohomology to the product
$$ \cup_{\cstar}: R\hHom^\cstar(\cF_1,\cF_2)\otimes R\hHom^\cstar(\cF_2,\cF_3) \longrightarrow R\hHom^\cstar(\cF_1,\cF_3)
$$

We shall first redefine the pant product as follows.   
Remember that we can identify $FH^*(L_1,L_2)$ to $FH^*((-L_1)\times L_2, \nu^*\Delta_N)$, and the same for its filtered version. 
Moreover there is a canonical isomorphism (this is K\"unneth, since we are with coefficients in a field)
from $FH^*((-L_1)\times L_2 \times (-L_2)\times L_3, \nu^*\Delta^{(1,2)(3,4)}_N)$  to $FH^*((-L_1)\times L_2, \nu^*\Delta_N)\otimes FH^*((-L_2)\times L_3, \nu^*\Delta_N)\simeq FH^*(L_1,L_2)\otimes FH^*(L_2,L_3)$. 

In the sequel we denote by $\Delta^{i,j}_N$ the subset of $N^4$ given by $\{(q_1,q_2,q_3,q_4) \mid q_i=q_j \}$ by $\Delta_N^{(i,j)(k,l)}=\Delta^{i,j}_N\cap \Delta^{k,l}_N$
and by $\Delta_N$ the set $\{(q,q,q,q) \mid q\in N\}$.
\subsection{Symplectic reduction}
We now have the following two propositions

\begin{prop}\label{Prop-9.1}
Let $Y\subset X$ be a smooth closed submanifold. Let $L$ be an exact Lagrangian brane with symplectic reduction by  $T_Y^*X$ denoted by $L_Y$. We assume $T_Y^*X$ is transverse to $L$, and that $L_Y$ is embedded. Then 
\begin{enumerate} 
\item We have an isomorphism 
$$FH^*(L, \nu^*Y;a,b) \simeq FH^*(L_Y,0_Y;a,b)$$
\item There is a natural map
$$r_Y: FH^*(L,0_X;a,b) \longrightarrow H^*(L_Y,0_Y;a,b)$$
induced by the composition of the above and of the natural map  $$FH^*(L,0_X;a,b) \longrightarrow FH^*(L, \nu^*Y;a,b)$$  obtained by pant product with the generator $u_Y$ of $FH^0(0_X,\nu^*Y; - \varepsilon , \varepsilon )\simeq H^0(Y)$ and induced on the sheaf side by the natural restriction map
$H^*(X\times [a,b[, \cF_L) \longrightarrow H^*(Y\times [a,b[, \cF_L)$.
\end{enumerate} 
\end{prop} 
\begin{proof}
First the map  $r_Y$ is induced by $\cF_L  \longrightarrow \cF_L\cstar k_{Y\times [0,+\infty [}$. 
In other words, denoting by $1_Y$ the generator of $H^*(X\times ]- \varepsilon , \varepsilon [, k_{Y\times [0,+\infty [})\simeq H^*(Y)$, we want to first prove that the following is a commutative diagram 
$$
\xymatrix{FH^*(L,0_X;a,b)\ar[r]^{\cup_{pant}u_Y} \ar[d]^{\simeq}& FH^*(L, \nu^*Y;a,b)\ar[d]^{\simeq}\\
H^*(X\times [a,b[,\cF_L)\ar[r]^{\cup 1_Y} & H^*(Y\times [a,b[,\cF_L)
}$$

Note that we have a restriction map $\rho_{X,Y}: H^*(X\times [a,b[,\cF_L)\longrightarrow  H^*(Y\times [a,b[,\cF_L)$ induced by the inclusion of $Y$ in $X$ and also a continuation map
$c_{X,Y}: FH^*(L,0_X;a,b) \longrightarrow FH^*(L, \nu^*Y;a,b)$ obtained from the remark that the characteristic functions of $X$ and $Y$ satisfy $(1-\chi_X) \leq (1-\chi_Y)$. Because as we remarked the restriction map on the complex of sheaves $\cF_L$ is induced by a continuation map on $FC^*$ (see Remark \ref{Rem-7.2} (\ref{Rem-7.2-2})), the following diagram is indeed commutative

$$
\xymatrix{FH^*(L,0_X;a,b)\ar[r]^{c_{X,Y}} \ar[d]^{\simeq}& FH^*(L, \nu^*Y;a,b)\ar[d]^{\simeq}\\
H^*(X\times [a,b[,\cF_L)\ar[r]^{\rho_{X,Y}} & H^*(Y\times [a,b[,\cF_L)
}$$

so we must prove that $c_{X,Y}$ and $\rho_{X,Y}$ coincide with the cup-products above.

For $\rho_{X,Y}$ this is easy. Indeed, it is well-known that the restriction map from $X\times {\mathbb R} $ to $Y\times {\mathbb R} $ is induced by $\cF \longrightarrow \cF\otimes k_{Y\times {\mathbb R} }$. As a result, replace $\cF$ by an injective complex $\cI$ and $k_Y$ by another injective resolution, then $H^*(X, \cF) \longrightarrow H^*(Y, \cF)$ is induced by the map of complexes $\Gamma(X,\cI) \longrightarrow \Gamma(Y,\cI\otimes k_Y)$ defined by $s \mapsto s\otimes 1_Y$, but $1_Y\in \Gamma(X,k_Y)\simeq H^0(X,\cJ)$, so this map is equivalent to the map induced by the cup product with $1_Y \in \Gamma(X,k_Y)=H^0(X,k_Y)$. 

We now consider the case of $c_{X,Y}$.

\begin{lem} \label{lemma-10.2}
Let $f \in C^\infty(N)$ be a non-positive function such that $\max_{x\in N}f(x)=0$. Then  the pant product  $$FH^*(L,0_N;a,b) \longrightarrow FH^*(L, \Gamma_f;a,b)$$
by $u_f$ the unit in $FH^*(0_N,\Gamma_f; - \varepsilon , \varepsilon )$, coincides with the monotone continuation map associated to the homotopy $t\mapsto t\cdot f$. 
\end{lem} 
\begin{proof} 
Consider the diagram, where the maps $FH^*(L_2,L_2) \longrightarrow FH^*(L_2,L'_2)$ and $FH^*(L_1,L_2) \longrightarrow FH^*(L_1,L'_2)$ are induced by an increasing  continuation map, denoted respectively by $c^2_{L_2,L'_2}, c^1_{L_2,L'_2}$. Then the following is a commutative diagram  (see \cite{Abouzaid-Seidel}, proposition 3.15)
$$
\xymatrix{
FH^*(L_1,L_2) \otimes  FH^*(L_2,L_2) \ar[r]^-{\cup_{pant}} \ar[d]^{id\otimes {c^2_{L_2,L'_2}}} &  FH^*(L_1,L_2) \ar[d]^-{c^1_{L_2,L'_2}} \\
FH^*(L_1,L_2) \otimes FH^*(L_2,L'_2) \ar[r]^-{\cup_{pant}} & FH^*(L_1,L'_2)
}
$$

Thanks to monotonicity, we have also for all $ \varepsilon >0$ the following commutative diagram
$$
\xymatrix{
FH^*(L_1,L_2;\lambda) \otimes  FH^*(L_2,L_2; \varepsilon ) \ar[r]^-{\cup_{pant}} \ar[d]^{id\otimes {c^2_{L_2,L'_2}}} &  FH^*(L_1,L_2; \lambda + \varepsilon ) \ar[d]^-{c^1_{L_2,L'_2}} \\
FH^*(L_1,L_2, \lambda ) \otimes FH^*(L_2,L'_2, \varepsilon ) \ar[r]^-{\cup_{pant}} & FH^*(L_1,L'_2;\lambda + \varepsilon )
}
$$
therefore for $L_2=0_N$ and $L'_2=\Gamma_f$ for $f\leq 0$ we get, denoting by $c_f$ both continuation maps

\begin{equation*}
\thetag{$\ast$}
\begin{split}
\xymatrix{
FH^*(L_1,0_N;\lambda) \otimes  FH^*(0_N,0_N; \varepsilon ) \ar[rr]^-{\cup_{pant}} \ar[d]^-{id\otimes c_f} &&  FH^*(L_1,0_N; \lambda + \varepsilon ) \ar[d]^{c_f} \\
FH^*(L_1,0_N, \lambda ) \otimes FH^*(0_N,\Gamma_f, \varepsilon ) \ar[rr]^-{\cup_{pant}} && FH^*(L_1,\Gamma_f;\lambda + \varepsilon )
}
\end{split}
\end{equation*} 

and denoting $c_f(1)=1_f$, we get a diagram 
$$
\xymatrix{
FH^*(L_1,0_N;\lambda) \ \ar[rr]^-{\cup_{pant}1=id} \ar[d]^-{id} &&  FH^*(L_1,0_N; \lambda + \varepsilon ) \ar[d]^{c_f} \\
FH^*(L_1,0_N, \lambda )  \ar[rr]^-{\cup_{pant}1_f} && FH^*(L_1,\Gamma_f;\lambda + \varepsilon )
}
$$

and as $ \varepsilon $ goes to $0$, 
$$
\xymatrix{
FH^*(L_1,0_N;\lambda) \ \ar[rr]^-{\cup_{pant}1=id} \ar[d]^-{id} &&  FH^*(L_1,0_N; \lambda  ) \ar[d]^{c_f} \\
FH^*(L_1,0_N, \lambda )  \ar[rr]^-{\cup_{pant}1_f} && FH^*(L_1,\Gamma_f;\lambda )
}
$$
In other words $c_f$ coincides with $\cup_{pant}1_f$. 
\end{proof} 
End of proof of Proposition \ref{Prop-9.1}:
Note that as $f$ goes to $-\infty\cdot (1-\chi_Z)$, the lower right hand side converges to $FH^*(L,\nu^*Z;\lambda )=FH^*(L_Z,0_Z;\lambda )$, and $1_f\in FH^*(0_N,\Gamma_f; \varepsilon )$ converge to $u_Z \in FH^*(0_N,\nu^*Z; \varepsilon )$, so the map $r_Z$ is a continuation map.
\end{proof}

We would like to extend the above to a reduction map

$$FH^*(L_1,L_2;a,b) \longrightarrow FH^*((L_1)_Z,(L_2)_Z;a,b)$$
also given by a cup product by $u_Z \in FH^0(0_N,\nu^*Z)$, but which must be explained. We start by the case where $L_2=\Gamma_f$. The situation is now easy, because
$$FH^*(L_1,\Gamma_f;a,b)=FH^*(L_1-\Gamma_f, 0_N;a,b)$$ and $(\Gamma_f)_Z=\Gamma_{f_{\mid Z}}$.

\begin{lem} Assume $L_2=\Gamma_f$. Then using the identification  $FH^*(L_1,\Gamma_f;a,b)=FH^*(L_1-\Gamma_f, 0_N;a,b)$ we have that the map
$$r_Z: FH^*(L_1,L_2;a,b) \longrightarrow FH^*((L_1)_Z,(L_2)_Z;a,b)$$  given by the cup-product with $u_Z\in FH^0(0_N,\nu^*Z;0^-,0^+)=H^0(Z)$.
Here again $r_Z$ corresponds at the sheaf level to the restriction map
$$H^*(N\times [a,b[, R\hHom^\cstar(\cF_{L_1}, \cF_{L_2})) \longrightarrow H^*(Z\times [a,b[, R\hHom^\cstar(\cF_{L_1}, \cF_{L_2})) $$

\end{lem} 

Since for any $X \subset N$ we have $\nu^*X$ is the limit of $\Gamma_{f_j}$ for a sequence $f_j$ as above. Also for $X$ transverse to $Z$ we have $(\nu^*X)_Z=\nu^*(X\cap Z)$ we have also, according to the above result
\begin{cor} 
Let $Z$ be transverse to $X$ in $N$. The map $$r_Z:  FH^*(L,\nu^*X;a,b) \longrightarrow FH^*((L)_Z,(\nu^*X)_Z;a,b)$$  given by the cup-product with $u_Z\in FH^0(\nu^*X, \nu^*(X\cap Z);0^-,0^+)=H^0(X\cap Z)$.
Here again $r_Z$ corresponds at the sheaf level to the restriction map
$$H^*(X\times [a,b[, \cF_{L}) \longrightarrow H^*((X\cap Z)\times [a,b[, \cF_{L}) $$
\end{cor} 
\begin{proof} 
Indeed, as $\Gamma_f$ converges to $\nu^*X$, we have that $(\Gamma_f)_Z=\Gamma_{f_{\mid Z}}$ converges to $\nu^*(X\cap Z)$. Then we have 
 $$FH^*(L,\nu^*X;a,b)= \lim_j FH^*(L, \Gamma_{f_j};a,b) $$  the map from the lemma $$FH^*(L, \Gamma_{f_j};a,b) \longrightarrow FH^*(L_Z, \Gamma_{{f_j}_{\mid Z}};a,b)$$ and the 
 identification $$\lim_jFH^*(L_Z, \Gamma_{{f_j}_{\mid Z}};a,b)=FH^*(L_Z, \nu^*(X\cap Z);a,b)$$

\end{proof} 
Finally we can prove that reduction by $Z$ is given by a cup-product without  restriction on $L_2$
\begin{prop} 
Let $L_1, L_2$ be lagrangian branes in $T^*N$ and assume they have good reduction at $Z$. Then there is a cup product map
$$r_Z: FH^*(L_1,L_2;a,b) \longrightarrow FH^*((L_1)_Z,(L_2)_Z;a,b)$$  given by the cup-product with $u_Z\in FH^0(0_N,\nu^*Z; 0^-, 0^+)=H^0(Z)$.
Here again $r_Z$ corresponds at the sheaf level to the restriction map
$$\rho_Z: H^*(N\times [a,b[, R\hHom^\cstar(\cF_{L_1}, \cF_{L_2})) \longrightarrow H^*(Z\times [a,b[, R\hHom^\cstar(\cF_{L_1}, \cF_{L_2})) $$
\end{prop} 

\begin{proof} 
Indeed, we may identify $FH^*(L_1,L_2;a,b)$ to  $FH^*(-L_1\times L_2, \nu^*\Delta_N;a,b)$, and then according to the previous corollary, considering the reduction with respect to $Z\times Z$, and using the fact that  $(-L_1\times L_2)_{Z\times Z}=(-L_1)_Z\times (L_2)_Z$  and $\Delta_N\cap (Z\times Z)=\Delta_Z$, we get that the map from the Corollary is
$$FH^*(-L_1\times L_2, \nu^*\Delta_N;a,b) \longrightarrow FH^*((-L_1\times L_2)_{Z\times Z}, \nu^*(\Delta_Z))$$
but since $$FH^*((-L_1\times L_2)_{Z\times Z}, \nu^*(\Delta_Z))=FH^*((-L_1)_Z\times (L_2)_Z,\nu^*\Delta_Z)=FH^*((L_1)_Z,(L_2)_Z))$$ we get the announced map. 
Note that $FH^*(L_1\times L_2, \nu^*\Delta_N; a,b)$ corresponds to $$H^*(\Delta_N \times [a,b[, \check{\cF_{L_1}}*\cF_{L_2})=H^*(N\times [a,b[, \check{\cF_{L_1}}\cstar \cF_{L_2})= H^*(N\times [a,b[, R\hHom (\cF_{L_1},\cF_{L_2})).$$
\end{proof} 
There is another symplectic reduction in a cotangent bundle of a product (or more generally a fibration), $X\times Y$. Indeed, the submanifold  $C^X=0_X\times T^*Y$ is coisotropic, and the reduction of $L$, if $L$ is transverse to $C^X$ and the projection $L\cap C^X \longrightarrow T^*Y$ is an  embedding, is a Lagrangian $L^Y$. 
We then have 

\begin{prop} \label{Second-reduction}
We have, denoting by $\pi_Y$ the projection of $X\times Y\times {\mathbb R} $ on $Y\times {\mathbb R} $,  $$H^*(X)\otimes \cF_{L^Y}=(R\pi_Y)_*(\cF_L)$$
\end{prop} 
We shall postpone the proof. Note however, that it follows from uniqueness and \cite{K-S} (see also \cite{Viterbo-ISTST}, Proposition 9.7
) that $SS((R\pi_Y)_*(\cF_L))=L^Y$. However this happens with multiplicity, i.e. at $t=+\infty$, we have 
$H^*(\{y\}\times \{t\}, (R\pi_Y)_*(\cF_L))=H^*(X\times \{y\}\times \{t\},(\cF_L))=H^*(X\times \{y\}, k_{X\times Y})=H^*(X)$. 

Note also that a "heuristic" approach would tell us that $(R\pi_Y)_*(\cF_L)$ is obtained by sheafifying the presheaf associated to the  quasi-presheaf $V \mapsto FC^*(L, 0_X\times \nu^*V; a,b)=FC^*(L^Y,\nu^*V;a,b)$ since $\nu^*X=0_X$. The uniqueness theorem tells us that if $\F$ with $SS(\F)=\widehat L$ and $\cF=V$ at $+\infty$ and $0$ at $-\infty$,  then $\cF=\cF_L\otimes V$. 

\subsection{The pant product}

We now consider the pant product. 
We first look at the Floer cohomology for products manifolds, i.e. we look at $FC^*(L_1\times L_2, 0_N\times 0_N)$. We claim this has two filtrations, one coming from the first factor, an other from the second factor. In other words $A(\gamma_1,\gamma_2)=A_1(\gamma_1)+A_2(\gamma_2)$ and the boundary map is increasing both for $A_1$ and $A_2$, provided our almost complex structure is a product structure. 
So this two filtrations allow us to define $FC^*(L_1\times L_2, 0_N\times 0_N; P)$ for any $P$ open in $ {\mathbb R}^2$ such that $P$ is equal to $P+Q$ where $Q$ is the cone, $]-\infty, 0], \times ]-\infty, 0]$. By considering $PFC^*(L_1\times L_2, \nu^*U_1\times \nu^*U_2, P)$, we can repeat the arguments of section \ref{Rectif} to section \ref{section-8}  and build a sheaf $P{\cF}_{L_1\times L_2}$ on $N\times N \times {\mathbb R}_Q^2$ where ${\mathbb R}_Q^2$ is the space $ {\mathbb R}^2$ with the topology for which open sets are subsets $P$ such that $P=P+Q$ (this is largely used  in \cite{K-S}). Noting that $$PFC^*(L_1\times L_2, 0_N\times 0_N; P_\nu, {\mathbb R} ^2)=FC^*(L_1\times L_2,0_N\times 0_N; [\nu, \infty[)$$ and hence $$H^*(N\times N\times {\mathbb R}^2, N\times N \times P_\nu; P\cF_{L_1\times L_2})=H^*(N\times [\nu, +\infty[,\cF_{L_1\times L_2})$$
 we have the commutative diagram, using the notations of page \pageref{altproof-9}

\xymatrixcolsep{.1in}
\hskip -.3in
\xymatrix{FH^*(L_1,0_N;\lambda,\infty)\otimes FH^*(L_2,0_N;\mu,+\infty)\ar[d]^{\simeq}\ar@/_5pc/@<-8ex>[dddd]_{\cup_{pant}} & \simeq & H^*(N\times [\lambda , +\infty[,\cF_{L_1})\otimes H^*(N\times [\lambda , +\infty[,\cF_{L_2})\ar[d]^{\simeq} \ar@/^4pc/@<8ex>[dddd]^{\cup_{\cstar}} \\
PFH^*(-L_1\times L_2, 0_N\times 0_N;P_{\lambda,\mu}) \ar[d]& \simeq & H^*(N\times N \times P_{\lambda,\mu},P\cF_{-L_1\times L_2})\ar[d]^{(Rs)_!}\\
FH^*(-L_1\times L_2,0_N\times 0_N; \lambda+\mu, +\infty)\ar[d]^{r_{\Delta_N}} & \simeq & H^*(N\times N \times [\lambda+\mu, +\infty [,\cF_{-L_1\times L_2})\ar[d]^{\rho_{\Delta_N}} \\
FH^*(-L_1\times L_2, \nu^*\Delta_N; \lambda+\mu, +\infty)\ar[d]^{\simeq} & \simeq & H^*( N \times [\lambda+\mu, +\infty [ ; \check{\cF_{L_1}}\cstar\cF_{L_2})\ar[d]^{\simeq}\\
FH^*(L_1, L_2; \lambda+\mu, +\infty) &\simeq &H^*( N \times [\lambda+\mu, +\infty [ ; R\hHom^\cstar(\cF_{L_1}, \cF_{L_2})) \\
}

the fact that the long left-side arrow coincides with the pant product in Floer cohomology follows from Lemma \ref{Lem-10.11}, that will be proved later. 

We thus just proved:

\begin{prop} \label{pre-product}
Let $L_1, L_2$ be two Lagrangian branes. Then the pant product $$FH^*(L_1, 0_N; \lambda,+\infty ) \otimes FH^*(0_N, L_2;\mu,+\infty ) \longrightarrow FH^*(L_1,L_2;\lambda+ \mu ,+\infty)$$ corresponds to the sheaf product  defined in Proposition \ref{cup-product} 
$$\cup_{\cstar}: H^*(N\times [\lambda , +\infty[;\cF_{L_1}) \otimes H^*(N\times [\mu , +\infty[;\check{\cF_{L_2}}) \longrightarrow H^*(N\times [\lambda + \mu, +\infty[; \cF_{L_1}\cstar \check{\cF_{ L_2}})
$$
\end{prop} 

\begin{cor} For any two smooth functions $f,g$, we have a pant product $$FH^*(L_1, \Gamma_f; \lambda,+\infty ) \otimes FH^*(\Gamma_g, L_2;\mu,+\infty ) \longrightarrow FH^*(L_1,L_2;\lambda+ \mu ,+\infty)$$ corresponding to the sheaf cup-product  defined in Proposition \ref{cup-product} 
$$H^*(N\times [\lambda , +\infty[;\cF_{L_1-\Gamma_f}) \otimes H^*(N\times [\lambda , +\infty[;\check{\cF}_{L_2-\Gamma_g}) \longrightarrow H^*(N\times [\lambda + \mu, +\infty [; \cF_{L_1}\cstar \cF_{g-f}\cstar \check{\cF_{ L_2}})$$
We also have for $X$, a submanifold of $N$, 
$$FH^*(L_1, \nu^*X; \lambda,+\infty ) \otimes FH^*(\nu^*X, L_2;\mu,+\infty ) \longrightarrow FH^*(L_1,L_2;\lambda+ \mu ,+\infty)$$ corresponds to the sheaf cup-product  defined in Proposition \ref{cup-product} 
$$H^*(X\times [\lambda , +\infty[;\cF_{L_1}) \otimes H^*(X\times [\lambda , +\infty[;\check{\cF_{L_2}}) \longrightarrow H^*(X\times [\lambda + \mu, +\infty [; \cF_{L_1}\cstar \check{\cF_{ L_2}})$$
\end{cor} 
\begin{proof} The map
$$FH^*(L_1, \Gamma_f; \lambda,+\infty ) \otimes FH^*(\Gamma_g, L_2;\mu,+\infty ) \longrightarrow FH^*(L_1-\Gamma_{f-g},L_2;\lambda+ \mu ,+\infty)$$
can be identified to
$$FH^*(L_1-\Gamma_f, 0_N; \lambda,+\infty ) \otimes FH^*(0_N, \Gamma_g- L_2;\mu,+\infty ) \longrightarrow FH^*(L_1-\Gamma_{f}+,L_2-\Gamma_{g};\lambda+ \mu ,+\infty)$$
According to previous proposition, this corresponds to the map
$$H^*(N\times [\lambda , +\infty[;\cF_{L_1-\Gamma_f}) \otimes H^*(N\times [\lambda , +\infty[;\cF_{\Gamma_g-L_2}) \longrightarrow H^*(N\times [\lambda + \mu, +\infty [; \cF_{L_1}\cstar \cF_{f-g}\cstar \cF_{ L_2})$$
For the last map, we only have to check that $\cF_L\cstar \cF_f=\cF_{L+\Gamma_f}$, and (as a consequence) $\cF_f\cstar \cF_g=\cF_{f+g}$.
Now if $f_j \longrightarrow -\infty (1-\chi_X)$ and $g_j \longrightarrow -\infty (1-\chi_X), g_j$, we get the second statement. 
\end{proof}

\begin{prop} \label{Prop-9.2}
The pant product is induced by the composition of the  isomorphism
$$FH^*(L_1,L_2)\otimes FH^*(L_2,L_3)=FH^*(L_1\times (-L_2), \nu^*\Delta_N)\otimes FH^*(\nu^*\Delta_N, L_3\times (-L_2))
$$

with 
the pant product
$$ 
 \begin{gathered} \cup_{pant} : FH^*(L_1\times (-L_2), \nu^*\Delta_N)\otimes FH^*(\nu^*\Delta_N, L_3\times (-L_2))\\ \Big\downarrow \\ FH^*(L_1\times (-L_2),L_3\times (-L_2)) \simeq FH^*(L_1,L_3)\otimes FH^*((-L_2), (-L_2))
 \end{gathered} $$
and the map induced by the projection $FH^*((-L_2), (-L_2)) \longrightarrow FH^0((-L_2), (-L_2))\simeq k$

\end{prop} 
Assuming the above proposition, we can conclude the
\begin{proof} [Proof of (\ref{5}) of Theorem \ref{Main-theorem}]
 We  have the identifications 
 \begin{enumerate} 
 \item  $FH^*(L_1\times (-L_2), \nu^*\Delta_N; \lambda, +\infty)$ to $H^*(N\times N\times [\lambda, +\infty[;(\check{\cF_{L_1}}\ast {\cF_{L_2}})\cstar k_{\Delta_N\times [0,+\infty[})$ 
 \item  
 $FH^*(\nu^*\Delta_N, L_3\times (-L_2); \lambda, +\infty)$ to $H^*(N\times N\times [\lambda, +\infty[;({\cF_{L_3}}\ast \check{\cF_{L_2}})\cstar \check k_{\Delta_N\times [0,+\infty[})$ 
\end{enumerate}  and the pant product corresponds to 
$$ \begin{gathered}  H^*(N\times N \times [\lambda,+\infty[; (\check{\cF_{L_1}}\ast {\cF_{L_2}})\cstar k_{\Delta_N\times [0,+\infty[}) \otimes
H^*(N\times N \times [\mu,+\infty[; (\cF_{L_3}\ast\check{\cF_{L_2}})\cstar \check{k}_{\Delta_N\times [0,+\infty[})\\ \quad \Big \downarrow \cup_{\cstar} \\
H^*(N\times N\times [\lambda+\mu,+\infty[; (\check{\cF_{L_1}}\ast{\cF_{L_2}})\cstar (\cF_{L_3}\ast\check{\cF_{L_2}})  \cstar \underbrace{k_{\Delta_N\times [0,+\infty[}\cstar \check k_{\Delta_N\times [0,+\infty[}}_{k_{N\times [0,+\infty[}})))\end{gathered} 
$$
 
 Now $$(\check{\cF_{L_1}}\ast{\cF_{L_2}})\cstar (\cF_{L_3}\ast\check{\cF_{L_2}})=(\cF_{L_1} \cstar \cF_{L_3}) \ast (\check{\cF_{L_2}}\cstar \check{\cF_{L_2}})$$

and the morphism $(\check{\cF_{L_2}}\cstar \check{\cF_{L_2}}) \longrightarrow k_{N\times [0,+\infty[}$ induces a map
$$(\cF_{L_1}\cstar\check{\cF_{L_3}})\ast (\check{\cF_{L_2}}\cstar\check{\cF_{L_2}}) \longrightarrow (\cF_{L_1}\cstar\check{\cF_{L_3}})\ast  k_{N\times [0,+\infty[}
$$
hence a map
$$\begin{gathered} H^*(N\times N\times [\lambda+\mu,+\infty[; (\cF_{L_1}\cstar\check{\cF_{L_3}})\ast (\cF_{L_2}\cstar\check{\cF_{L_2}})) \\ \Big \downarrow \\
H^*(N\times [\lambda+\mu,+\infty[; (\check{\cF_{L_1}}\cstar {\cF_{L_3}}))\otimes H^*(N\times [- \varepsilon , \varepsilon [ ; \cF_{L_2}\otimes \cF_{L_2})
\end{gathered} $$
\end{proof} 

\begin{lem} \label{Lem-10.10}
Let  $L_1, L_2$ be lagrangians branes in an aspherical symplectic manifold $(M,\omega)$, and consider in $M\times \overline M$ (where $\overline M$ represents $M$ with the opposite symplectic form) the Floer cohomology  pant-product

$$ \xymatrix{FH^*(L_1\times L_2, \Delta_M;\lambda, +\infty ) \otimes FH^0(\Delta_M, L_3\times L_2; \mu, +\infty )
\ar[d]\\ FH^*(L_1\times L_2, L_3\times L_2; \lambda + \mu , +\infty)} 
$$
induces  a map $$\xymatrix{FH^*(L_1,L_2; \lambda , +\infty)\otimes FH^*(L_2,L_3; \mu, +\infty)\ar[d]\\ FH^*(L_1,L_3; \lambda +\mu, +\infty)\otimes FH^*(L_2,L_2; - \varepsilon , \varepsilon )}$$
through the identifications
$$FH^*(L_1\times L_2, \Delta_M; \lambda , +\infty)=FH^*(L_1,L_2; \lambda , +\infty)$$
$$FH^*(\Delta_M, L_3\times L_2; \mu, +\infty)=FH^*(L_2,L_3; \mu, +\infty)$$
and
$$FH^*(L_1\times L_2, L_3\times L_2, \Delta_M, \nu, +\infty)=FH^*(L_1,L_3; \nu, +\infty)\otimes FH^*(L_2,L_2; 0,0)$$
The projection of the above map on $FH^*(L_1,L_3; \lambda, +\infty)\otimes FH^0(L_2,L_2;0,0)$  coincides with the pant-product. 
\end{lem} 
\begin{proof} 

Indeed, the first map counts the number of holomorphic triangles $T$ in $M\times \overline M$ with  vertices  $(x_{1,2},x_{1,2})$, where $x_{1,2}\in L_1\cap L_2$, $(x_{2,3},x_{2,3})$ where $x_{2,3}\in L_2\cap L_3$, $(x_{1,3},u) $ where $x_{1,3} \in L_1\cap L_3$ and $u \in L_2$. The boundary of the triangle is the union of 
\begin{enumerate} 
\item a path contained in $\Delta_M$, connecting $(x_{1,2},x_{1,2})$ to 
$(x_{2,3},x_{2,3})$, so of the form $(\alpha, \alpha)$ where $\alpha$ is a path from $x_{1,2}$ to $x_{2,3}$. 
\item  a path contained in $L_1\times L_2$, connecting $(x_{1,2},x_{1,2})$ to $(x_{1,3},u)$, so of the form $(\beta_1,\beta_2)$ where $\beta_1$ connects $x_{1,2}$ to $x_{1,3}$ and $\beta_2$ connects $x_{1,2}$ to $u$
\item a path contained in $L_3\times L_2$ connecting $(x_{3,2},x_{3,2})$ to $(x_{1,3},u)$, so of the form $(\gamma_1,\gamma_2$ where $\gamma_1)$ connects $x_{3,2}$ to $x_{1,3}$ and $\beta_2$ connects $x_{3,2}$ to $u$
\end{enumerate} 

Now taking the product almost complex structure $J_1$ and $J_2$ on each factor of $M\times \overline M$, we see that the two projections of $T$ are holomorphic triangles. So we get two triangles in $M$ or $\overline M$ denoted by $T_1,T_2$, and such that 

\begin{enumerate} 
\item $T_1$ is a holomorphic triangle with vertices $x_{1,2},x_{3,2},x_{1,3}$ and boundaries $\alpha, \beta_1, \gamma_1$ where $\gamma_1\subset L_1$
\item $T_1$ is a holomorphic triangle with vertices $x_{1,2},u,x_{3,2}$ and boundaries $\alpha, \beta_2, \gamma_2$ where $\gamma_2\subset L_2$
\end{enumerate} 

Now we can glue this triangles,  as shown on figure \ref{Figure-5}. Reversing the orientation of $T_2$, this yields a holomorphic  triangle with vertices $x_{1,2}, x_{2,3}, x_{1,3}$ with sides $\beta_1$ in $L_1$, $\gamma_2\circ \beta_2$ in $L_2$, $\gamma_1$ in $L_3$, since $u$ is just an arbitrary points in $L_2$ respectively  and this is why we restrict to the component in $FH^0(L_2,L_2)$. Now define if $J(z,u)$ to be the almost-complex structure on the triangle $T_1\cup T_2$ equal to $J_1$ on $T_1$ and $J_2$ on $T_2$, we thus get a $J$ holomorphic triangle $T$. Conversely if $T$ is a $J$-holomorphic triangle of this type, identifying the triangle to the square of Figure \ref{Figure-5}, we get two triangles which correspond to $T_1,T_2$ and hence to a $(J_1,J_2)$ holomorphic triangle defining the pant-product. 

   \setlength{\fboxrule}{1pt}
   \begin{figure}[H]
  \begin{overpic}[width=6cm]
{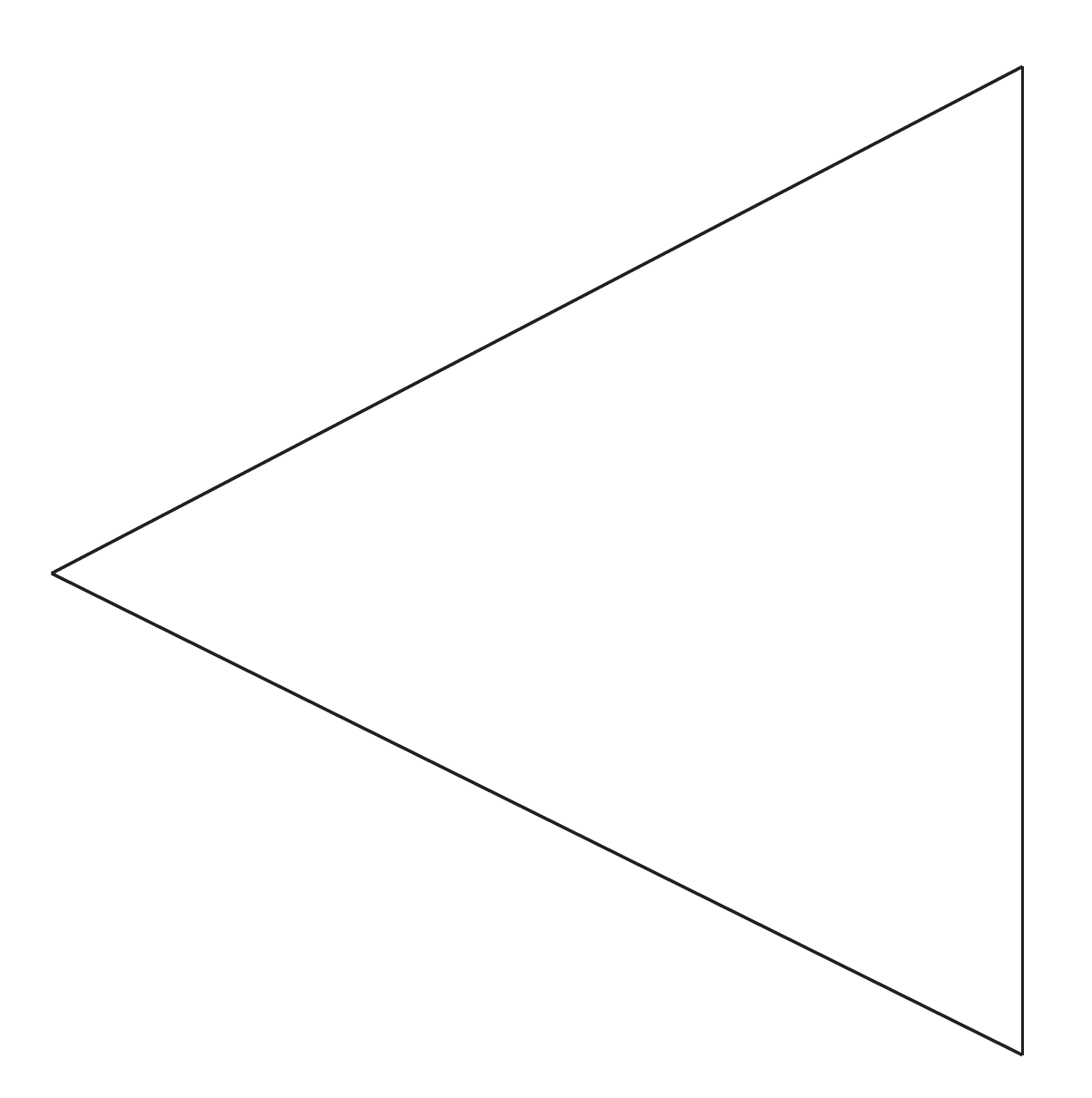}
 \put (-7,50) {$x_{1,3}$}
 \put(87,0){$x_{3,2}$}
 \put(87,97){$x_{1,2}$}
 \put(26,75){$L_1 \supset \beta_1$}
  \put(85,50){$\alpha$}
   \put(45,15){$\gamma_1 \subset L_3$}
     \put(45,50){\fbox{$T_1$}}
\end{overpic}
\hspace{1cm}
\begin{overpic}[width=6cm]
{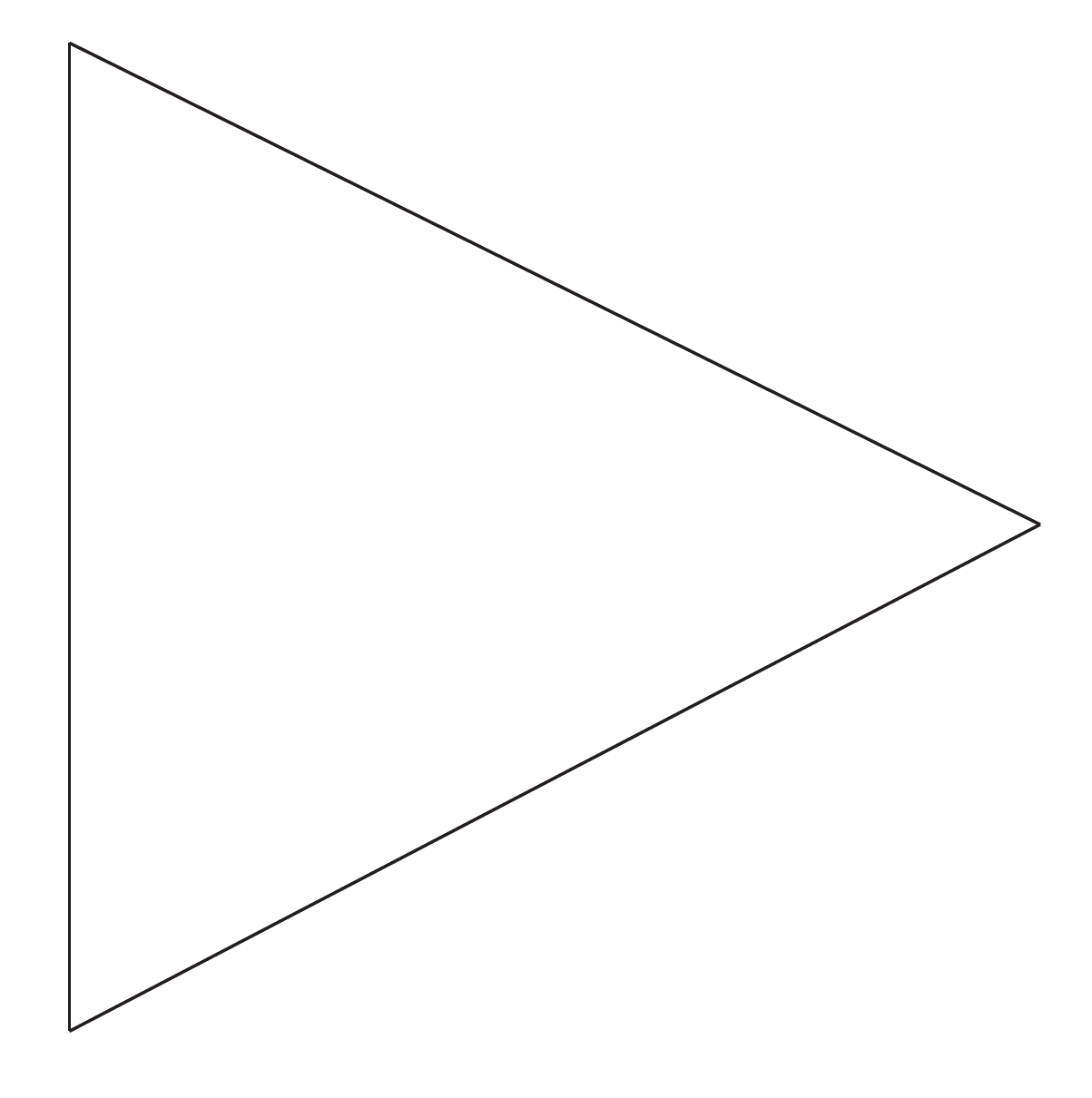}
 \put (1,97) {$x_{1,2}$}
 \put(2,0){$x_{3,2}$}
 \put(95,50){$u$}
 \put(50,83){$\gamma_2\subset L_2$}
  \put(10,50){$\alpha$}
   \put(50,25){$\beta_2\subset L_2$}
   \put(40,60){\fbox{$T_2$}}
\end{overpic}
\caption{Triangles $T_1, T_2$}
\end{figure}

 \begin{figure}[H]
\begin{center} 
\begin{overpic}[width=12cm]{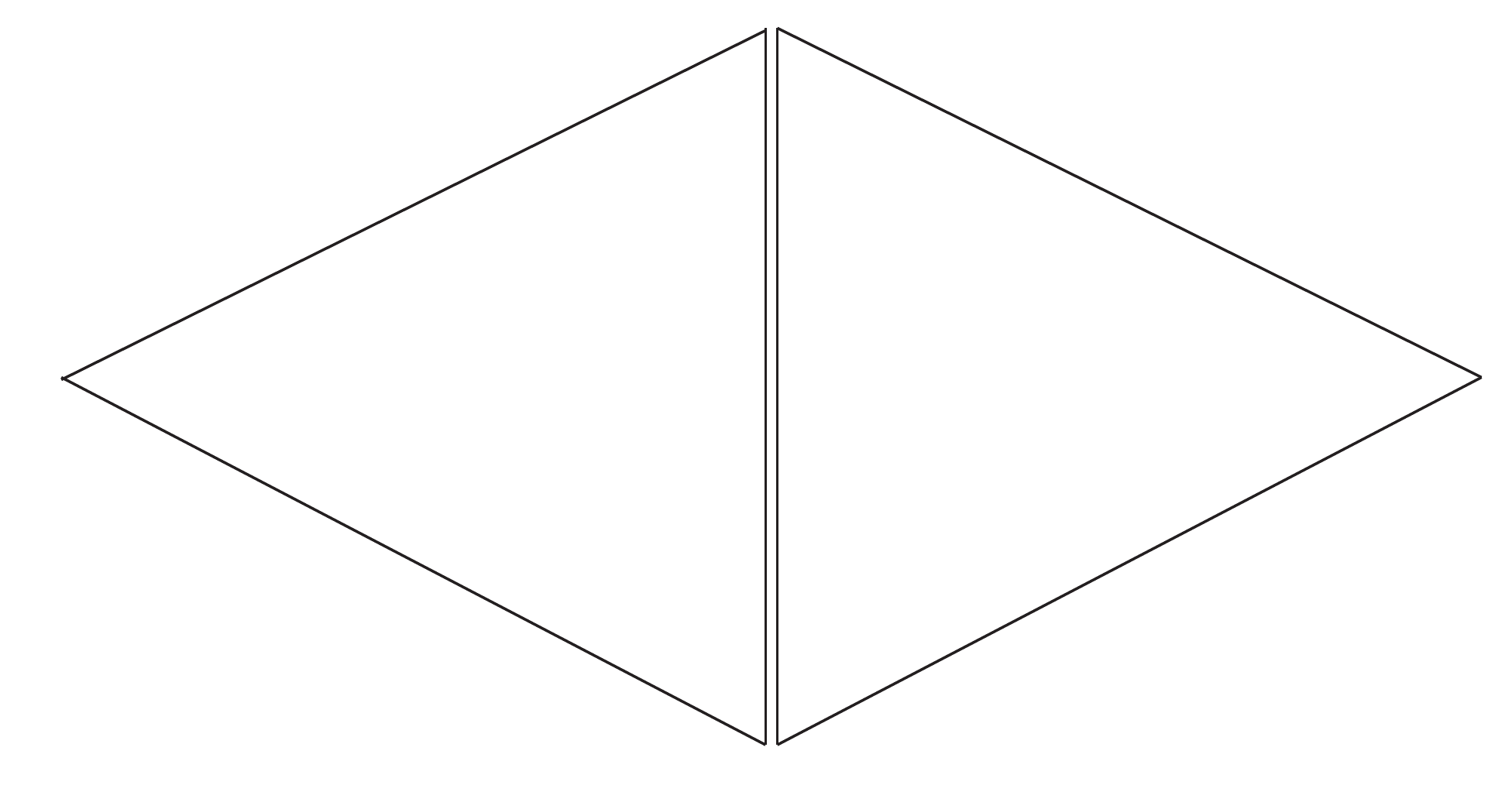}
\put (0,25) {$x_{1,3}$}
 \put(50,0){$x_{3,2}$}
 \put(50,53){$x_{1,2}$}
 \put(96,30){$u$}
   \put(72,10){$\beta_2$}
 \put(25,43){$\beta_1$}
  \put(25,10){$\gamma_1$}
  \put(45,25){$\alpha$}
   \put(72,43){$\gamma_2$}
\end{overpic}\caption{"Triangle" $T_1\#T_2$}\label{Figure-5}
\end{center}
\end{figure}
 
 This concludes our proof of the lemma and clearly of the proposition. 
\end{proof}

\begin{lem} \label{Lem-10.11}
Let $\Lambda$ be a brane in the symplectic manifold $(M,\omega)$ as in Lemma \ref{Lem-10.10}. The the pant product 
$$ \xymatrix{FH^*(L_1\times (-L_2), \Lambda\times \Lambda ; \lambda,+\infty)\otimes  FH^*(L_1\times (-L_2), \Lambda\times \Lambda ; \mu,+\infty)
\ar[d]\\ FH^*(L_1\times (-L_2), \Delta_M; \lambda+\mu, +\infty)=FH^*(L_1,L_2;\lambda+\mu, +\infty) } 
$$

can be identified to the cup-product 
$$FH(L_1,\Lambda)\otimes FH^*(\Lambda,L_2;\mu,+\infty) \longrightarrow FH^*(L_1,L_2;\lambda+\mu, +\infty)$$

\end{lem} 
\begin{proof} 
This is the same idea as in lemma \ref{Lem-10.10}. The cup product counts holomorphic triangles connecting $(x_1,x_2)\in L_1\times (-L_2)\cap \Lambda\times \Lambda$, i.e. $x_1\in L_1\cap \Lambda$ and $x_2\in \Lambda \cap (-L_2)$,  $(u,u)\in \Lambda\times \Lambda \cap \Delta_M$ i.e. $u\in \Lambda$, and $(y,y)\in L_1\times (-L_2)\cap \Delta_M$, i.e. $y\in L_1\cap (-L_2)$. 

This is equivalent to a pair of holomorphic triangles, the first connecting $x_1,u,y$, the second connecting  $u,x_2,y$. More precisely
we get two triangles  as below, and gluing them we get a triangle with vertices $x_1,x_2,y$ with $x_1\in L_1\cap \Lambda, x_2\in L_2\cap \Lambda, y \in L_1\cap L_2$, where $u$ is left free on $\Lambda$ (this corresponds to taking $1 \in FH^0(\Lambda\times \Lambda, \Delta_M,0^-,0^+)\simeq H^0(\Lambda)$. This concludes our proof. 

   \setlength{\fboxrule}{1pt}
   \begin{figure}[H]
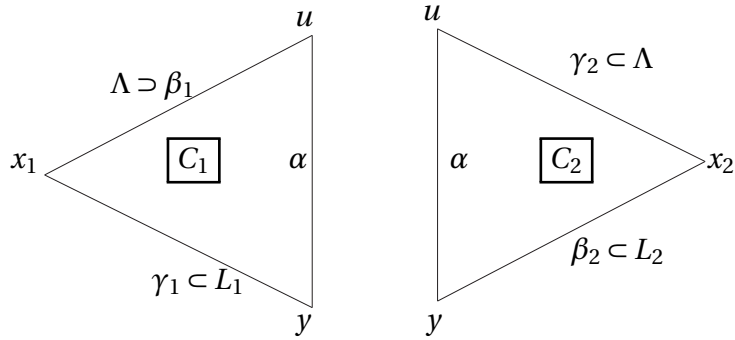
 
  \begin{overpic}[width=4cm]
{Triangle4b-sym.pdf}
 \put (-7,50) {$x_{1}$}
 \put(87,-3){$y$}
 \put(87,97){$u$}
 \put(26,75){$\Lambda \supset \beta_1$}
  \put(85,50){$\alpha$}
   \put(40,10){$\gamma_1 \subset L_1$}
     \put(45,50){\fbox{$C_1$}}
\end{overpic}
\hspace{1cm}
\begin{overpic}[width=4cm]
{Triangle4b.pdf}
 \put (1,99) {$u$}
 \put(2,-3){$y$}
 \put(95,50){$x_2$}
 \put(50,83){$\gamma_2\subset \Lambda$}
  \put(10,50){$\alpha$}
   \put(50,20){$\beta_2\subset L_2$}
   \put(40,50){\fbox{$C_2$}}
\end{overpic}
\caption{Triangles $C_1, C_2$}
\end{figure}

 \begin{figure}[H]
\begin{center} 
\begin{overpic}[width=8cm]{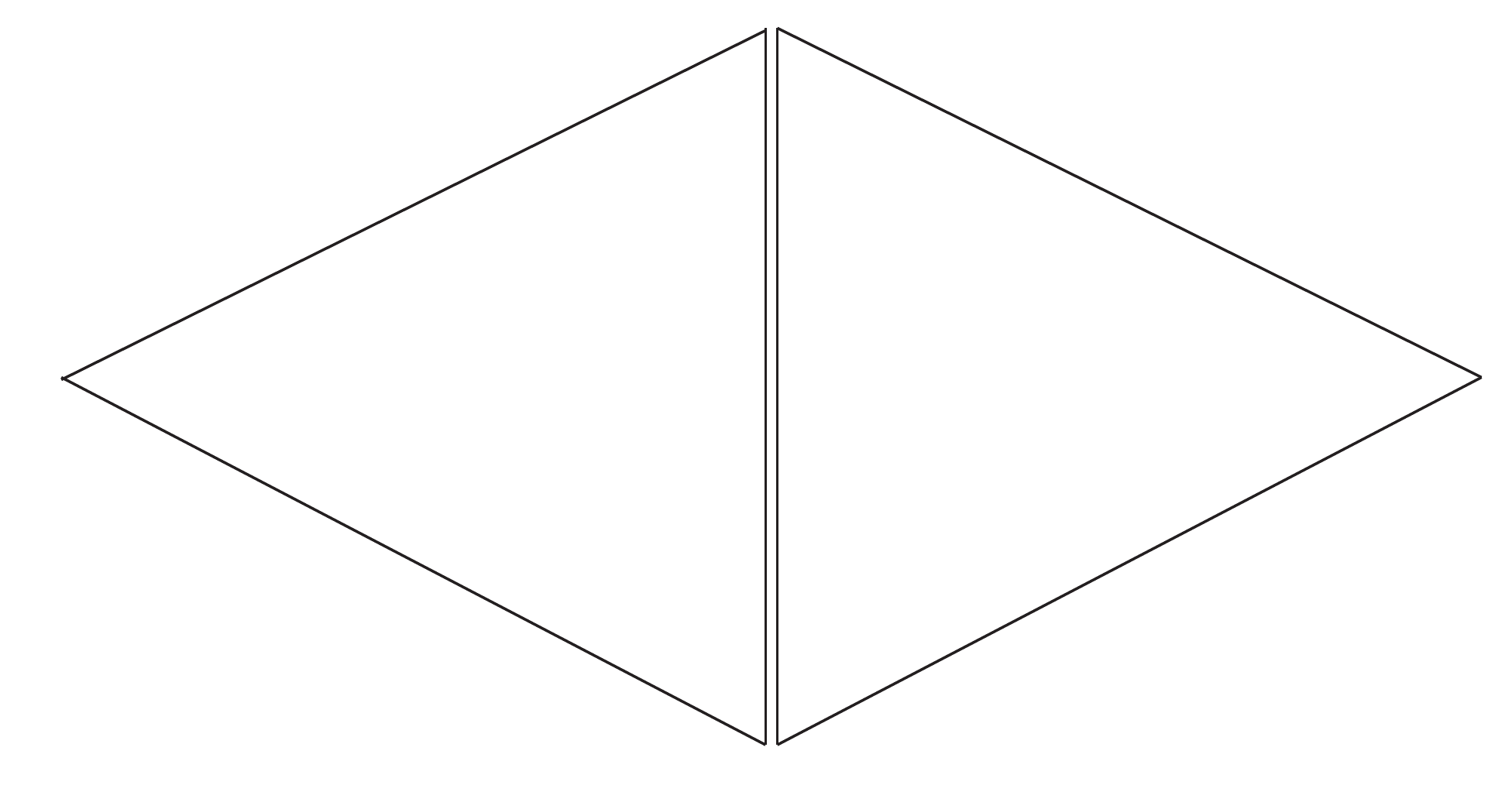}
\put (-3,25) {$x_{1}$}
 \put(50,-2){$y$}
 \put(50,53){$u$}
 \put(100,27){$x_2$}
   \put(72,10){$\beta_2$}
 \put(25,43){$\beta_1$}
  \put(25,10){$\gamma_1$}
  \put(45,25){$\alpha$}
   \put(72,43){$\gamma_2$}
\end{overpic}\caption{"Triangle" $C_1\#C_2$}\label{Figure-6}
\end{center}
\end{figure}

\end{proof} 
\section{Other applications, generalizations, etc.}

An "obvious" application is the Fukaya-Seidel-Smith theorem and Kragh-Abouzaid
\begin{thm}[\cite{Fukaya-Seidel-Smith, Kragh2}] \label{Fukaya-Seidel-Smith theorem}\index{Nadler's theorem}\index{Theorem!Fukaya-Seidel-Smith}\index{Theorem!Nadler}
Let $L$ be an exact Lagrangian in the cotangent bundle $T^{*}N$ of a simply connected manifold. Then the Maslov class of $L$ vanishes and the  projection $\pi:L \to N$ induces an isomorphism in cohomology. 
\end{thm}  

Once we have the quantization, the proof is the same in the simply connected case as in the above references, and Guillermou (\cite{Guillermou}) proves directly the vanishing of the Maslov class, while this is proved differently by Kragh and Abouzaid (see \cite{Kragh2}).

\subsection{Lagrangian cobordism}

Let $C$ be an exact Lagrange cobordism in $T^*(N\times [0,1])$. This means $C\cap \{t=0\}/(\tau ) =L_0\subset  T^*N$  and $C\cap \{t=0\}/(\tau ) =L_1\subset  T^*N$ with the proper orientations. Then $FH^*(L_1,L)=FH^*(L_2,L)$  for all $L$. 

 Biran and Cornea (see \cite{Biran-Cornea-a}, \cite{Biran-Cornea-b}) pointed out that if $C$ is a cobordism between $L_{0}$ and $L_{1}$ (all manifolds being exact and with vanishing Maslov class), then $FH^{*}(L_{0},L)\simeq FH^{*}(L_{1},L)$. 

Let us prove this in the case of cotangent bundles. Indeed, let $\widehat C$ be the lift of $C$ to a homogeneous lagrangian. We have to prove that if $\C$ is a sheaf over $X\times [0,1]$, such that $ SS(\C)=\widehat C, SS(\C_{\mid X\times \{i\}})=\widehat L_{i}$
$$H^{*}(X, \C_{0})=H^{*}(X,\C_{1})$$

It is enough to check that if $\psi(x,t,s)=t$ for $(x,t,s)\in M\times [0,1]\times {\mathbb R} $, then $$SS(\C)\cap L_{\psi}=\emptyset$$

Indeed,

$$SS (\C)=\widehat C=\{ (x,t,\sigma p, \sigma t, s, \sigma) \mid (x,t,p,\tau) \in C, ds=pdx+\tau dt\}$$

$$L_{\psi}=\{(x,t,p,\tau, s,\sigma) \mid p=\sigma=0\}$$

So the intersection corresponds to $\sigma=0, \sigma \tau=1$ which is impossible. As a result, $SS(\C)\cap L_{\psi}=\emptyset$ hence 
$H^{*}(\{\psi \leq c \}, \C)$ does not depend on $t$, and this implies $H^{*}(\{\psi = 0 \}, \C)=H^{*}(\{\psi = 1 \}, \C)$ that is 
$$H^{*}(X, \C_{0})=H^{*}(X,\C_{1})$$
hence $FH^{*}(L_{0})=FH^{*}(L_{1})$

\begin{question}
What happens if $Y$ is a cobordism between $X_{0}$ and $X_{1}$, and we have $C$ as above in $T^{*}Y$, such that $\partial C =L_{0}\cup L_{1}$. Can one estimate the changes in Floer cohomology ?
\end{question}

\subsection{Floer cohomology with sheaf coefficients}

A number of objects can be defined. For example, if $\cG$ is an element in $D^b(N\times {\mathbb R} ) $ and $L$ an exact Lagrangian in $T^*N$ we may define

\begin{defn} 
We set $FH^*(L,\cG)= R\Hom^\cstar(\cF_L, \cG)$. 
By definition, we have $FH^*(L_1,\cF_{L_2})=FH^*(L_1,L_2)=R\Hom^\cstar(\cF_{L_1} \cF_{L_2})$. All these are filtered in the obvious way. 
\end{defn}

 \section{Appendix: On quasi-presheaves and operations}\label{Appendix-operations}

If $\Cat{C}$ is a category, let $N\Cat{C}$ be its nerve, that is the category with the same objects as $\Cat{C}$ and such that $\Mor(X,Y)$ is a simplicial set,  the set of its $k$ simplices $\Mor(X,Y)_k$ being the set of sequences $X=X_0\overset{f_0}\to X_1\overset{f_1}\to X_2\overset{f_2}\to...... X_{k-1}\overset{f_{k}}\to X_{k+1}=Y$, with the map $\partial_j : \Mor(X,Y)_k \longrightarrow \Mor(X,Y)_{k-1}$ sending 
$$X=X_0\overset{f_0}\to X_1\overset{f_1}\to X_2\overset{f_2}\to...X_{j-1}\overset{f_{j-1}}\to X_{j}\overset{f_{j}}\to X_{j+1}\overset{f_{j+1}}\to... X_{k}\overset{f_{k}}\to X_{k+1}=Y$$
to
$$X=X_0\overset{f_0}\to X_1\overset{f_1}\to X_2\overset{f_2}\to...X_{j-1}\overset{f_j\circ f_{j-1}}\to  X_{j+1}\overset{f_{j+1}}\to... X_{k-1}\overset{f_{k}}\to X_{k+1})=Y$$
and we have that the composition sends $\Mor(X,Y)_k \times \Mor(Y,Z)_l$  to 
$\Mor(X,Z)_{k+l}$. 
We denote by $\partial =\sum_{j=0}^k(-1)^j\partial_k$. 

Now on $\Cat{Ch^b}$ the category of bounded chain complexes, we have a simplicial structure, where $\Mor(X_*,Y_*)_k$ is the set of linear  maps shifting degree by $k$ (but not necessarily chain maps) together with $D: \Mor(X_*,Y_*)_k \longrightarrow \Mor(X_*,Y_*)_{k-1}$ given by $D(f)=d\circ f+(-1)^kf\circ d$. 

Now a quasi-presheaf  on the category $\Cat{C}$ is by definition a presheaf on the category  $N\Cat{C}$ with values in $\Cat{Ch^b}$. 
In other word to any object $X$ in $\Cat{C}$ we associate $F(X)$ in $\Cat{Ch^b}$, that is a chain complex, and for each $k$-simplex $\sigma$  in $\Mor(X,Y)_k$, we associate $F(\sigma)_k$ in $\Mor(X_*,Y_*)_k$. 
This is exactly our construction in  section \ref{Rectif}, where $\Cat{C}=\Cat{F(N)}$, and the quasi-presheaf is given by $f \mapsto FC^*_L(f)$, or for the filtered situation, $\Cat{C}=\Cat{F(N)\times ( {\mathbb R}, \leq)}$ and $(f,\lambda) \mapsto FC^*_L(f,\lambda)$.
What we did is to replace the quasi-presheaf by a sheaf, in our case $(f,\lambda) \mapsto \widehat{FC}^n_L(f, \lambda)$, such that there is a chain homotopy equivalence
$FC^*_L(f,\lambda) \mapsto \widehat{FC}^n_L(f, \lambda)$, functorial in $(f,\lambda)$. 

This is part of the general process known as rectification (see \cite{Vogt, Goerss-Jardine, Lurie}).

We shall need the following remark  \label{App-rem}:

Let $C_*^t$ be a filtered graded module. We can consider it as a complex of sheaves on $( {\mathbb R} , \geq)$, where open sets are $]s,+\infty[$, and we have an inclusion $C_*^s \longrightarrow C_*^t$ for $s\geq t$, that is injective, hence on the dual (cohomological level) we have
$C_t^* \longrightarrow C^s_*$ that is surjective. If we set $C^*(]-\infty,s[)=C^*_s$, we get a complex of sheaves on $( {\mathbb R} , \leq)$, which are flabby, hence acyclic\footnote{Terminology is misleading, flabbiness is a weaker form of injectivity, but is related to restriction maps being....surjective}. As a result, if $s: {\mathbb R} \times {\mathbb R} $ is given by $s(u,v)=u+v$, this is a continuous map (with the product topology of $( {\mathbb R} ,\leq)\times ( {\mathbb R} ,\leq)$ to $( {\mathbb R} ,\leq)$, and since $C^*\boxtimes D^*$ is also made of acyclic sheaves, we get that $(Rs)_!(C^*\boxtimes D^*)$ coincides with $s_!(C^*\boxtimes D^*)$ on such filtered graded modules.

\printindex


\begin{thebibliography}{AAAAAAA}
\renewcommand {\em}{\bf }

\bibitem[Ab-S]{Abouzaid-Seidel}
M. Abouzaid, and  P. Seidel. 
\newblock{\em An open string analogue of Viterbo functoriality.} 
\newblock{ Geom. Topol. 14, n$^0$ 2 (2010): 627-718. https://doi.org/10.2140/gt.2010.14.627.}


\bibitem[Ad]{Adams}
J.F. Adams, 
\newblock{\em On the cobar construction}
\newblock{Proc. Nat.  Acad.  Sci. U.S.A. } vol.42 
(1956),  409-412. 


\bibitem[Ar]{Arnold}
V.I. Arnold,
\newblock{\em On a characteristic class entering a quantization condition.}
\newblock Functional Analysis and Its Applications 1, 1-14 (1967)
 


\bibitem[Au]{Auroux}
D. Auroux,
\newblock{A beginner's introduction to Fukaya categories.}
\newblock{\em In:
Contact and Symplectic Topology, Volume 26 of the series Bolyai Society Mathematical Studies pp 85-136
June 2014}


\bibitem[B-C]{Biran-Cornea-a}
P. Biran and O. Cornea
\newblock{\em Lagrangian cobordism I.}
\newblock{ Jour. of the Amer. math. Soc.
Volume 26, Number 2, April 2013, pp. 295-340.}

\bibitem[B-C2]{Biran-Cornea-b}
P. Biran and O. Cornea
\newblock{\em Lagrangian cobordism and Fukaya Categories.}
\newblock{Geom. Funct. Anal. Vol. 24 (2014) 1731-1830.}


\bibitem[B-O]{Bourgeois-Oancea}
F. Bourgeois and A. Oancea, 
 \newblock{\em   Symplectic homology, autonomous Hamiltonians, and Morse-Bott moduli spaces},
\newblock{Duke Math. J.}, vol. 146, n. 1 (2009), pp. 71--174.


\bibitem[Bro]{Brown}
E.H. Brown, 
\newblock{ Twisted tensor products, I.}
\newblock {\em Annals of Mathematics}, Second Series, Vol. 69, No. 1 (Jan., 1959), pp. 223-246.


\bibitem[Cor]{Cordier}
J.M. Cordier
\newblock{Sur la notion de diagramme homotopiquement coh\'erent.}
\newblock {\em Cahiers de Topologie et G\'eom\'etrie Diff\'erentielle Cat\'egoriques}, 23.1 (1982): 93-112. <http://eudml.org/doc/91292>.

\bibitem[Fl1]{Floer1} A. Floer, A relative Morse index for the   
  symplectic action, {\it Comm. Pure Appl. Math.} {\bf 41} (1988),   
  393-407   

\bibitem[Fl2]{Floer2} A. Floer, The unregularized gradient flow of the
symplectic action, {\it Comm. Pure Appl. Math.} {\bf 41} (1988), 775-813
   
\bibitem[Fl3]{Floer3} A. Floer, Symplectic fixed points and   
  holomorphic spheres, {\it Comm. Math. Phys.} {\bf 120} (1989),   
  575-611    
   
\bibitem[Fl4]{Floer4} A. Floer, Witten's complex and infinite dimensional   
  Morse theory, {\it J. Diff. Geom.} {\bf 30} (1989), 207-221   

\bibitem[Fl-Ho]{Fl-Ho}
A.Floer and H. Hofer, 
\newblock{Symplectic homology I. Open sets in $ {\mathbb C}^n$}
\newblock {\em Math. Zeit. 215 (1994), pp. 37-88.}

\bibitem[Fukaya]{Fukaya}
K. Fukaya,
\newblock{\em Floer   homology for families- report   of   a   project in progress.}
\newblock Available  at
http://www.math.kyoto-u.ac.jp/~fukaya/familyy.pdf

\bibitem[Fuk-Oh]{Fukaya-Oh}
K. Fukaya, Y.G. Oh,
\newblock{\em Zero-loop open strings in the cotangent bundle and Morse homotopy.}
\newblock{
Asian J. Math. 1997 International Press
Vol. 1, No. 1, pp. 96?180, March 1997
}

\bibitem[F-S-S]{Fukaya-Seidel-Smith}
K. Fukaya, P. Seidel and I. Smith,
\newblock{\em Exact Lagrangian submanifolds in simply-connected cotangent bundles.}
\newblock {Inventiones}, 2008, Volume 172, Number 1, Pages 1-27.

\bibitem[F-S-S2]{Fukaya-Seidel-Smith2}
K. Fukaya, P. Seidel and I. Smith,
\newblock{\em The symplectic geometry of cotangent bundles from a categorical  viewpoint.}
\newblock arXiv:0705.3450v4 and in:

Homological Mirror Symmetry: New Developments and Perspectives, pp. 1--26. Berlin, Heidelberg: Springer Berlin Heidelberg, 2009. https://doi.org/10.1007/978-3-540-68030-7\_1.

\bibitem[Ga]{Gabber}
O. Gabber,
\newblock{\em The integrability of the characteristic variety.}
\newblock{Amer. Journ. Math.103(1981),  445Ð468.}

\bibitem[G-P-S]{G-P-S}
S. Ganatra, J. Pardon, V. Shende.
\newblock{\em Microlocal Morse theory of wrapped Fukaya categories.}
 https://arxiv.org/abs/1809.08807

\bibitem[Go-Ja]{Goerss-Jardine}
P.G. Goerss and J.F. Jardine
\newblock{\em Simplicial Homotopy Theory.}
\newblock  Modern Birkh\"auser classics, 2009 (reprint of the 1999 edition)

\bibitem[Gu]{Guillermou}
S. Guillermou
\newblock{\em Quantization of conic Lagrangian submanifolds of cotangent bundles.}
\newblock{ArXiv: 1212.5818, v.2}

\bibitem[G-K-S]{G-K-S}
S. Guillermou, M. Kashiwara and P. Schapira,
\newblock {\em Sheaf quantization of Hamiltonian isotopies and applications to non displaceability problems.}
\newblock	 arXiv:1005.1517

\bibitem[Hut]{Hutchings}
M. Hutchings
\newblock {\em Floer homology of families I.}
\newblock Alg. and Geom. Top., vol.8 (2008), pp. 435-492

\bibitem[Igu]{Igusa}
K. Igusa,
\newblock {\em Iterated integrals of Superconnections.}
\newblock arXiv:0912.0249v1

\bibitem[Igu2]{Igusa2}
K. Igusa,
\newblock {\em Twisting cochains and higher torsion}
\newblock arXiv:math/0212383

\bibitem[Iv]{Iversen}
B. Iversen,
\newblock {\em The cohomology of Sheaves.}
\newblock Universitext, Springer-Verlag

\bibitem [J]{Jones}
C.K.R.T. Jones,
\newblock {\em Geometric singular perturbation theory.}
\newblock  {Lectures given at the 2nd Session of the Centro Internazionale Matematico Estivo (C.I.M.E.) held in Montecatini Terme, Italy, June 13-22, 1994, pp. 44-118.}
\newblock{Springer-Verlag, 1995.}

\bibitem [K]{K}
M. Kashiwara and P. Schapira,
\newblock {\em Systemps of microdifferential equations.}
\newblock Progress in Math., vol. 34, Birkh\"auser,  Boston, 1983.

\bibitem[K-S]{K-S}
M. Kashiwara and P. Schapira,
\newblock {\em Sheaves on manifolds.}
\newblock Grundlehren der Math. Wissenschaften, vol. 292, Springer-Verlag, 1990.

\bibitem[Kelly]{Kelly}
G.M. Kelly,
\newblock{\em Basic concepts of enriched category theory}, 
\newblock Reprints in Theory and Applications of Categories, No. 10, 2005.
\newblock Sydney

\bibitem[Komani]{Komani}
D. Komani,
\newblock{\em Continuation maps in Morse theory.}
\newblock{Doctoral Thesis, ETH Z\"urich, 2012.}

\bibitem[Kragh]{Kragh}
T. Kragh,
\newblock {\em Nearby Lagrangians and Fibered Spectra.}
\newblock In preparation, 2011.

\bibitem[Kragh1]{Kragh1}
T. Kragh,
\newblock {\em The Viterbo transfer map as a map of spectra.}
\newblock arXiv:0712.2533v3

\bibitem[Kragh2]{Kragh2}
T. Kragh,
\newblock {\em Parametrized ring-spectra and
the nearby Lagrangian conjecture. Appendix by M. Abouzaid.}
\newblock Geometry \&Topology 17 (2013) 639?731. 

\bibitem[La]{Lang}
S. Lang,
\newblock {\em Algebra.}
\newblock GTM  Springer, 2002.

\bibitem[Lau]{Laudenbach}
F. Laudenbach 
\newblock{\em A Morse complex on manifolds with boundary.}
\newblock  Geometriae Dedicata 1, 153 (2011) 47 - 57.

\bibitem[Lau-Sik]{Laudenbach-Sikorav}
F. Laudenbach and J.-C. Sikorav,
\newblock{\em Persistance d'intersection avec la section nulle au cours d'une isotopie hamiltonienne dans un fibr\'e cotangent.}
\newblock{\em Inventiones Math.} vol. 82 (1985), pp. 349-357.

\bibitem[LeP-N-V]{LePeutrec-Nier-Viterbo}
D. LePeutrec, F. Nier and C. Viterbo,
\newblock{\em Precise Arrhenius law for $p$-forms.}
\newblock{Ann. Henri Poincar\'e}, vol. 14 (2013), 567-610.


\bibitem[Lurie]{Lurie}
J. Lurie,
\newblock{\em Higher Topos Theory.}
\newblock{Princeton University Press.}, Annals of Math Studies 170, (2009). ISBN: 9781400830558 | 

\bibitem[MacLane]{MacLane}
S. Mac Lane,
\newblock{\em Categories for the working mathematician}
\newblock{GTM 5}, 2nd edition 1998, Springer-Verlag. 

\bibitem[May]{May}
P. May,
\newblock{\em 	Simplicial objects in algebraic topology}
1982

\bibitem[M]{Meinrenken}
E. Meinrenken, 
\newblock{\em On Riemann-Roch formulas for multiplicities.}
\newblock  Journal of the American Mathematical Society 9 (2): 373 - 389, doi:10.1090/S0894-0347-96-00197-X, MR 1325798.

\bibitem[Mil]{Milnor-hcobordism}
J. Milnor, 
\newblock{\em On the h-cobordism theorem.}
\newblock{ Princeton University Press}

\bibitem[Mit]{Mitchell}
B. Mitchell,
\newblock{ \em Spectral Sequences for the Layman.}
\newblock{The American Mathematical Monthly}, Vol. 76, No. 6 (1969), pp. 599-605. 
\newblock Stable URL: http://www.jstor.org/stable/2316659

\bibitem[Nad]{Nadler}
David Nadler,
\newblock{\em Microlocal branes are constructible sheaves.}
\newblock ArXiv math/0612399.

\bibitem[Nad-Z]{Nadler-Zaslow}
D. Nadler and  E. Zaslow, 
\newblock{\em Constructible sheaves and the Fukaya category.}
\newblock {J. Amer. Math. Soc.} 22 (2009), 233-286. 

\bibitem[Oan]{Oancea}
A. Oancea
\newblock{\em  Fibered symplectic cohomology and the Leray-Serre spectral sequence.}
\newblock  J. of Symplectic Geom. Volume 6, Number 3 (2008), pp. 267-351.

\bibitem[Riehl]{Riehl}
E. Riehl, 
\newblock{\em On the structure of simplicial categories associated to quasi-categories.}
\newblock Math.Proc. Camb. Phil. Soc. vol. 159(2011), pages 489-504


\bibitem[Seg]{Segal}
G. Segal, 
\newblock{\em Categories and cohomology theories.}
\newblock{Topology, vol 13,(1974), pp. 253-312. }
\newblock {Doi 10.1016/0040-9383(74)90022-6}
	

\bibitem[Sei]{Seidel}
P. Seidel, 
\newblock{\em Fukaya category and Picard-Lefschetz theory.}
European Math. Soc. 2008.

\bibitem[Sei2]{Seidel2}
P. Seidel
\newblock{Graded lagrangian submanifolds.}
\newblock{\em Bulletin de la Soci\'et\'e Math\'ematique de France} (2000)
Volume: 128, Issue: 1, page 103-149

\bibitem[Sik1]{Sikorav-GF}
J.-C. Sikorav, 
\newblock{Probl\`emes d'intersection et de points fixes en G\'eom\'etrie Hamiltonienne.}
\newblock{\em Commentarii Math. Helv.,} vol.62(1987), 62-73.

\bibitem[Sh1]{Shende1}
V. Shende
\newblock{\em Generating families and constructible sheaves.}
\newblock arXiv:1504.01336

\bibitem[Sik2]{Sikorav}
J.-C. Sikorav, 
\newblock{Some properties of holomorphic curves in almost complex manifolds.}
\newblock{\em  in Holomorphic
Curves in Symplectic Geometry} 
\newblock Birkh\"auser (1994), 165?189.

\bibitem [Theret]{Theret}
D. Th\'eret,
\newblock{ A complete proof of Viterbo's uniqueness theorem on generating functions.}
\newblock{\em Topology and its Applications} vol. 96, Issue 3(1999), Pages 249-266


\bibitem[Tam]{Tamarkin}
D. Tamarkin,
\newblock {\em Microlocal condition for non-displaceablility.}
\newblock arXiv:0809.1584 

\bibitem[Tam2]{Tamarkin2}
D. Tamarkin,
\newblock {\em Microlocal category.}
\newblock ArXiv:1511.08961 [math.SG]
	
\bibitem[To\"en]{Toen}
B. To\" en,
\newblock{\em Lectures on DG-categories.}
\newblock{Topics in Algebraic and Topological K-Theory,
Baum, Paul Frank, .... / Lectures Notes in Mathematics 2008, Springer-Verlag, 2011}

\bibitem[Tra]{Traynor}
L. Traynor,
\newblock{\em Symplectic homology via generating functions.}
\newblock  {Geometric and Functional Analysis 4(6):718-748, January 1994}

\bibitem[Vak]{Vakil}
R. Vakil,
\newblock {\em The rising sea. Foundations of Algebraic Geometry}
\newblock {http://math.stanford.edu/~vakil/216blog/FOAGdec2915public.pdf}

\bibitem[Vic1]{Vichery-th}
N. Vichery, 
\newblock{\em Th\`ese. Partie II. Applications
de la th\'eorie des faisceaux \`a la topologie symplectique.}
\newblock {CMLS, \'Ecole polytechnique, 2012.}
\newblock{ https://pastel.archives-ouvertes.fr/file/index/docid/780016/filename/Manuscript.pdf}

\bibitem[Vic2]{Vichery}
N. Vichery, 
\newblock{\em Homological differential calculus.}
\newblock arXiv:1310.4845v1

\bibitem[Vit]{Viterbo-STAGGF} C.~Viterbo, \newblock { Symplectic  topology as the geometry of generating
functions.} \newblock  {\em Mathematische Annalen}, vol 292, (1992), pp.
685--710.

 \bibitem[Vit2]{Viterbo-FCFH1}
C. Viterbo,
\newblock  { Functors and Computations in Floer cohomology I.}
\newblock {\em GAFA, Geom. funct. anal.} Vol. 9 (1999) pp. 985--1033

 \bibitem[Vit3]{Viterbo-FCFH2}
C. Viterbo,
\newblock  { Functors and Computations in Floer cohomology II.}
\newblock {\em ArXiv } 

 \bibitem[Vit4]{Viterbo-ISTST}
C. Viterbo,
\newblock  {An Introduction to Symplectic Topology  through Sheaf theory}
\newblock {\em Preprint }

 \bibitem[Vogt]{Vogt}
R.M. Vogt,
\newblock  { Homotopy limits and colimits.}
\newblock {\em Math. Z.}, 34 (1973) 11-52

\bibitem[Weib]{Weibel}
C. Weibel,
\newblock {\em An introduction to Homological Algebra.}
\newblock Cambridge Studies in Advanced Math., vol. 29, CUP, 1994.

\bibitem[Wein]{Weinstein}
A. Weinstein,
\newblock {\em Lectures on symplectic manifolds.}
\newblock CBMS Regional Conference series in Mathematics, vol. 29, Amer. Math. Soc. 1977.
\end{thebibliography}
\end{document}